\documentclass[11pt]{article}

\usepackage{comment}
\usepackage{makecell}
\usepackage{amsmath,amssymb,amsthm}
\usepackage{caption}
\usepackage{mathrsfs}
\usepackage{epsfig}
\usepackage{enumitem}
\usepackage{framed}
\usepackage{longtable}
\usepackage{graphicx,epstopdf}
\usepackage{subfigure}
\usepackage{url}
\usepackage{cases}
\usepackage{verbatim}
\usepackage{algpseudocode}
\usepackage{hyperref}
\usepackage{rotating}
\usepackage{multirow}
\usepackage{booktabs}
\usepackage{algorithmicx,algorithm}
\usepackage[font=scriptsize]{caption}
\allowdisplaybreaks[4]

\newtheorem{theorem}{Theorem}
\newtheorem{proposition}{Proposition}
\newtheorem{lemma}{Lemma}

\newtheorem{remark}{Remark}

\newtheorem{definition}{Definition}

\newcommand{\gap}{\vspace{0.04in}}

\newcommand{\interior}{\operatorname{int}}
\newcommand{\Conv}{\operatorname{con}}
\newcommand{\dom}{\operatorname{dom}}
\newcommand{\dist}{\operatorname{dist}}
\newcommand{\Hausd}{\mathbb{H}}
\newcommand{\D}{\mathbb{D}}
\newcommand{\N}{\mathbb{N}}
\newcommand{\R}{\mathbb{R}}
\newcommand{\ER}{\overline{\R}}
\newcommand{\Min}{\displaystyle\operatornamewithlimits{minimize}}
\newcommand{\Limsup}{\displaystyle\operatornamewithlimits{Lim \, sup}}
\newcommand{\Epiconv}{\overset{\text{e}}{\rightarrow}}
\newcommand{\Contconv}{\overset{\text{c}}{\rightarrow}}
\newcommand{\Pointconv}{\overset{\text{p}}{\rightarrow}}

\DeclareMathOperator*{\argmin}{arg\,min}

\title{Variational Theory and Algorithms for a Class of Asymptotically Approachable Nonconvex Problems}

\author{Hanyang Li
\thanks{Department of Industrial Engineering and Operations Research, University of California, Berkeley,  Berkeley 94720 (\href{mailto:hanyang_li@berkeley.edu}{hanyang\_li@berkeley.edu})}
\and Ying Cui
\thanks{Department of Industrial Engineering and Operations Research, University of California, Berkeley, Berkeley 94720 (\href{mailto:yingcui@berkeley.edu}{yingcui@berkeley.edu})}
}

\begin{document}
\maketitle
\begin{abstract}
    We investigate a class of composite nonconvex functions, where the outer function is the sum of univariate extended-real-valued convex functions and the inner function is the limit of difference-of-convex functions. A notable feature of this class is that the inner function may fail to be locally Lipschitz continuous. It covers a range of important yet challenging applications, including inverse optimal value optimization and problems under  value-at-risk constraints.
    We propose an asymptotic decomposition of the composite function that guarantees epi-convergence to the original function, leading to necessary optimality conditions for the corresponding minimization problem. The proposed decomposition also enables us to design a numerical algorithm such that any accumulation point of the generated sequence, if exists, satisfies the newly introduced optimality conditions. These results expand on the study of so-called amenable functions introduced by Poliquin and Rockafellar in 1992, which are compositions of convex functions with smooth maps, and the prox-linear methods for their minimization. To demonstrate that our algorithmic framework is practically implementable, we further present verifiable termination criteria and preliminary numerical results.

\paragraph{Keywords:}
epi-convergence; optimality conditions; nonsmooth analysis; difference-of-convex functions
\end{abstract}

\section{Introduction.}\label{sec:introduction}

We consider a class of composite optimization problems of the form:
\begin{equation}
\label{eq:cvx_composite}
\begin{array}{rl}
	\Min_{x\in \R^n}\;\,&\;\, \displaystyle\sum_{p=1}^{m} \left[F_p(x) \triangleq \varphi_p \big(f_p(x)\big)\right],
\end{array}
\tag{CP$_0$}
\end{equation}
where for each $p = 1, \cdots, m$, the outer function $\varphi_p:\R \to \R \cup \{+\infty\}$ is proper, convex, lower semicontinuous (lsc), and the inner function $f_p:\R^n \rightarrow \R$ {is not necessarily locally Lipschitz continuous}.

\gap
If each inner function $f_p$ is continuously differentiable, then the objective in \eqref{eq:cvx_composite} belongs to the family of \textsl{amenable functions} under a constraint qualification \cite{poliquin1992amenable, poliquin1993calculus}.
For a thorough exploration of the variational theory of amenable functions, readers are referred to \cite[Chapter 10(F)]{rockafellar2009variational}.
The properties of amenable functions have also led to the development of prox-linear algorithms, where convex subproblems are constructed through the linearization of the inner smooth mapping \cite{fletcher1982model, burke1985descent, burke1995gauss, lewis2016proximal,drusvyatskiy2019efficiency}.

\gap
However, there are various applications of composite optimization problem in the form of \eqref{eq:cvx_composite} where the inner function $f_p$ is nondifferentiable. In the following, we provide two such examples.

\gap\gap
\noindent
\textbf{Example 1.1} (The inverse optimal value optimization).
For $p=1, \cdots, m$, consider the optimal value function
\begin{equation}\label{eq:value_func}
    f_p(x) \triangleq \displaystyle{\inf_{y \in {\R^d}}} \left\{  (c^{\,p} + C^{\,p} x)^\top y + \frac 1 2 \,y^\top Q^{\,p} \, y
    \;\middle|\;  A^{\,p} x + B^{\,p} y \leq b^{\,p}\,\right\} \qquad x\in {\R^n}
\end{equation}
with appropriate dimensional vectors $b^{\,p}$ and $c^{\,p}$, and matrices $A^{\,p}, B^{\,p}, C^{\,p}$ and $Q^{\,p}$. The function $f_p$ is not smooth in general. 
The inverse (multi) optimal value problem \cite{ahmed2005inverse, paleologo2001bandwidth}  finds a vector $x \in {\R^n}$ that minimizes the discrepancy between observed optimal values ${\{\nu_p\}_{p=1}^{m}}$ and true optimal values $\{f_p(x)\}_{p=1}^{m}$ based on a prescribed metric, such as the $\ell_1$-error:
\begin{equation}\label{eq:inverse}
    \Min_{x \in {\R^n}} \;\sum_{p=1}^{m} \left|{\nu_p} - f_p(x)\right|.
\end{equation}
If $f_p$ is real-valued for $p=1, \cdots, m$, one can express problem \eqref{eq:inverse} in the form of \eqref{eq:cvx_composite} by defining the outer function $\varphi_p(t) = |{\nu_p} - t|$.

\gap\gap

\noindent
\textbf{Example 1.2} (The portfolio optimization under a value-at-risk constraint).
Given a random variable $Y$, the Value-at-risk (VaR) of $Y$ at a confidence level $\alpha\in (0,1)$ is defined as $\mbox{VaR}_\alpha(Y) \triangleq {\min}  \left\{\gamma \in \R\, \mid\, \mathbb{P}(Y\leq \gamma ) \geq \alpha\right\}$. 
Let $Z$ be the random return of investments and $c(\cdot,\cdot)$ be a lsc function representing the profit of $Z$ parameterized by $x\in \R^n$. 
An agent's goal is to maximize the expected profit, denoted by $\mathbb{E}[\,c(x,Z)]$, while also controlling the risk via a constraint on $\mbox{VaR}_\alpha[c(x,Z)]$ under a prescribed level $r$. The model can be written as
\begin{equation}\label{eq:risk management}
    {\displaystyle\operatornamewithlimits{minimize}_{x \in \R^n}  \;\; -\mathbb{E}\left[\, c(x,Z) \right]} \quad 
    \mbox{subject to}  \;\; 
    \mbox{VaR}_\alpha[\, c(x,Z)] \, \geq r.
\end{equation}
Define $\delta_A$ as the indicator function of a set $A$, where $\delta_A(t) = 0$ for $t \in A$ and $\delta_A(t) = +\infty$ for $t \notin A$. Problem \eqref{eq:risk management} can then be put into the framework \eqref{eq:cvx_composite} by setting $\varphi_1(t) = -t$, $f_1(x) = \mathbb{E}[\, c(x,Z)]$, $\varphi_2(t) = \delta_{[r, +\infty)} (t)$, and $f_2(x) = \mbox{VaR}_\alpha[\, c(x,Z)]$. We note that the  function $\mbox{VaR}_\alpha[\, c(\cdot,Z)]$ can be nondifferentiable even if the  function $c(\cdot, z)$ is differentiable for {every} $z$.

\gap

{
Due to the nondifferentiablity of the inner function $f_p$ in \eqref{eq:cvx_composite}, the overall objective  is not amenable
and the prox-linear algorithm \cite{fletcher1982model} is not applicable to solve this composite optimization problem. 
In this paper, we develop an algorithmic framework for a subclass of \eqref{eq:cvx_composite}, where each inner function $f_p$, although nondifferentiable, {can be derived from DC functions through a limiting process}. We refer to this class of functions as \textit{approachable difference-of-convex (ADC) functions} (see section \ref{subsec: ADC and properties} for the formal definition). 
It is important to note that ADC functions are ubiquitous. In particular, we will show that the optimal value function $f_p$ in \eqref{eq:value_func} and $\mbox{VaR}_\alpha[\,c(\cdot,Z)]$ in \eqref{eq:risk management} are instances of ADC functions under mild conditions. In fact, based on the result recently shown in \cite{royset2020approximations}, any lsc function is the epi-limit of piecewise affine DC functions.

\gap

With this new class of functions in hand, we have made a first step to understand the variational properties of the composite ADC minimization problem \eqref{eq:cvx_composite}, including an in-depth analysis of its necessary optimality conditions.
The novel optimality conditions are defined through a handy approximation of the subdifferential mapping $\partial f_p$ that explores the ADC structure of $f_p$. Using {the notion of \textsl{epi-convergence}}, we further show that these optimality conditions are necessary conditions for any local solution of \eqref{eq:cvx_composite}. Additionally, we propose a double-loop algorithm to solve \eqref{eq:cvx_composite}, where the outer loop dynamically updates the DC functions approximating each $f_p$, and the inner loop finds an approximate stationary point of the resulting composite DC problem through successive convex approximations. It can be shown that any accumulation point of the sequence generated by our algorithm satisfies the newly introduced optimality conditions.

Our strategy to handle the nondifferentiable and possibly discontinuous inner function $f_p$ through a sequence of DC functions shares certain similarities with the approximation frameworks in the existing literature. For instance, Ermoliev et al. \cite{ermoliev1995minimization} have designed smoothing approximations for lsc functions utilizing convolutions with bounded mollifier sequences, a technique akin to local ``averaging".  
Research has sought to identify conditions that ensure gradient consistency for the smoothing approximation of composite nonconvex functions \cite{chen2012smoothing, burke2013gradient, burke2013epi, burke2017epi}. Notably, Burke and Hoheisel \cite{burke2013epi} have emphasized the importance of epi-convergence for the approximating sequence, a less stringent requirement than the continuous convergence assumed in earlier works \cite{chen2012smoothing, beck2012smoothing}. In recent work, Royset \cite{royset2022consistent} has studied the consistent approximation of the composite optimization in terms of the global minimizers and stationary solutions, where the inner function is assumed to be locally Lipschitz continuous. 
Our notion of subdifferentials and optimality conditions for \eqref{eq:cvx_composite} takes inspiration from these works but adapts to accommodate nonsmooth approximating sequences that exhibit the advantageous property of being DC.

\gap
The rest of the paper is organized as follows. Section \ref{sec2: ADC functions} presents a class of ADC functions and introduces a new associated notion of subdifferential. 
In section \ref{sec3:composite model}, we investigate the necessary optimality conditions for problem \eqref{eq:cvx_composite}. 
Section \ref{sec4:algorithm} is devoted to an algorithmic framework for solving \eqref{eq:cvx_composite} and its convergence analysis to the newly introduced optimality conditions. We also discuss termination criteria for practical implementation in section \ref{subsec:terminate}. Preliminary numerical experiments on the inverse optimal value problems are presented in the last section.

\vskip 0.1in
\noindent
{\bf Notation and Terminology.}
{Let $\|\cdot\|$ denote the Euclidean norm in $\R^n$.} We use the symbol $\mathbb{B}(\bar x, \delta)$ to denote the Euclidean ball $\{x \in \R^n \mid \|x - \bar x\| \leq \delta\}$. The set of nonpositive and nonnegative are {denoted by} $\R_-$ and $\R_+$, respectively, and the set of nonnegative integers is {denoted by} $\mathbb{N}$. We write $\mathbb{N}_\infty^\sharp\triangleq \{N \subset \mathbb{N} \,\mid\, N \mbox{ infinite}\}$ {and $\mathbb{N}_\infty \triangleq \{ N \,\mid\, \mathbb{N}\;\backslash N \text{ finite} \}$}. 
Notation $\{t^k\}$ is used to simplify the expression of any sequence $\{t^k\}_{k \in \N}$, where the elements can be points, sets, or functions. 
By $t^k \to t$ and $t^k \to_N t$, we mean that the sequence $\{t^k\}$ and the subsequence $\{t^k\}_{k \in N}$ indexed by $N \in \mathbb{N}^\sharp_\infty$ converge to $t$, respectively.

\gap
Given  two sets $A$ and $B$ in $\R^n$ {and a scalar $\lambda \in \R$}, the Minkowski sum and the scalar multiple are defined as $A + B \triangleq \{a + b \mid a \in A, b \in B\}$ and $\lambda \, A \triangleq \{\lambda \, a \mid a \in A\}$. We also define $0 \cdot \emptyset = \{0\}$ and $\lambda \cdot \emptyset = \emptyset$ whenever $\lambda \neq 0$. When $A$ and $B$ are nonempty and closed, we define the one-sided deviation of $A$ from $B$ as $\D (A, B) \,\triangleq\,  \sup_{x \in A} \; \dist(x, B)$, where $\dist(x, B) \triangleq \inf_{y \in B} \| y - x \|$. The Hausdorff distance between $A$ and $B$ is given by  $\Hausd (A, B) \triangleq \max\{\D(A, B), \,\D(B, A)\}$. 
The boundary and interior of $A$ are denoted by $\operatorname{bdry}(A)$ and $\interior(A)$. The topological closure and the convex hull of $A$ are indicated by $\operatorname{cl}(A)$ and $\Conv A$.

\gap
For a sequence of sets $\{C^k\}$, we define its outer limit as
\[
    \Limsup_{k \rightarrow +\infty} C^k \triangleq \{ u \mid \exists \, N \in \mathbb{N}_\infty^\sharp,\, u^k \to_N u \text{ with } u^k \in C^k \},
\]
and the horizon outer limit as
\[
    {\Limsup_{k \rightarrow +\infty}}^\infty \, C^k \triangleq \{0\} \cup \left\{u \mid \exists \, N \in \mathbb{N}_\infty^\sharp,\, \lambda_k \downarrow 0,\,\lambda_k u^k \rightarrow_N u  \text{ with }  u^k \in C^k\right\}.
\]
The outer limit of a set-valued mapping $S: \R^n \rightrightarrows \R^m$ is defined as
\[
    \displaystyle\Limsup_{x \rightarrow \bar x} S(x) \triangleq \bigcup_{x^k \rightarrow \bar x} \Limsup_{k \rightarrow +\infty} S(x^k) = \{ u \mid \exists \, x^k \rightarrow \bar x,\, u^k \rightarrow u \text{ with } u^k \in S(x^k)\} \quad \bar{x}\in \R^n.
\]
We say $S$ is outer semicontinuous (osc) at $\bar x \in \R^n$ if $\operatorname{Lim \, sup}_{x \rightarrow \bar x} S(x) \subset S(\bar x)$.
Consider some index set $N \in \N^{\sharp}_{\infty}$. A sequence of sets $\{C^k\}_{k \in N}$ is {equi-bounded} if there exists a bounded set $B$ such that $C^k \subset B$ for all $k \in N$. Otherwise, the sequence is {unbounded}. If there is an integer $K \in N$ such that $\{C^k\}_{k \in N, k \geq K}$ is equi-bounded, then the sequence $\{C^k\}_{k \in N}$ is said to be {eventually bounded}. Interested readers are referred to \cite[Chapter 4]{rockafellar2009variational} for a comprehensive study of set convergence.

\gap
The regular normal cone and the limiting normal cone of a set $C \subset \R^n$ at $\bar x \in C$ are given by
\[
    \widehat{\mathcal{N}}_C(\bar x) \triangleq \left\{v \,\middle|\, v^\top(x - \bar x) \leq o(\|x - \bar x\|) \text{ for all } x \in C\right\}\quad \mbox{and}\quad
    \mathcal{N}_C(\bar x) \triangleq \Limsup_{x(\in C)\rightarrow \bar x} \widehat{\mathcal{N}}_C(x).
\]
The proximal normal cone of a set $C$ at $\bar x \in C$ is defined as $\mathcal{N}^p_C(\bar x) \triangleq \{\lambda(x - \bar x) \mid \bar x \in P_C(x), \lambda \geq 0\}$, where $P_C$ is the projection  onto $C$ that maps any $x$ to the set of points in $C$ that are closest to $x$.

\gap
For an extended-real-valued function $f:\R^n \rightarrow \ER \triangleq \R \cup \{\pm\infty\}$, we write its  effective domain as $\dom f \triangleq \{x \in \R^n \mid f(x) < +\infty\}$, and the epigraph as $\operatorname{epi} f \triangleq \{(x, \alpha)\in \R^{n+1} \mid \alpha \geq f(x)\}$. We say $f$ is proper if $\dom f$ is nonempty and $f(x) > -\infty$ for all $x \in \R^n$. {We adopt the common rules for extended arithmetic operations, {and} the lower and upper limits of a sequence of scalars in $\ER$ (cf. \cite[Chapter 1(E)]{rockafellar2009variational}).}

\gap
Let $f: \R^n \to \ER$ be a proper function. We write $x \rightarrow_f \bar x$, if $x \rightarrow \bar{x} \text{ and } f(x) \rightarrow f(\bar{x})$. The regular subdifferential and the limiting subdifferential of $f$ at $\bar x  \in \dom f$ are respectively defined as
\[
    \widehat\partial f(\bar x) \triangleq \{v \mid f(x) \geq f(\bar x) + v^\top(x - \bar x) + o(\|x - \bar x\|) \text{ for all } x \}\quad \mbox{and}\quad
    \partial f(\bar x) \triangleq \Limsup_{x \rightarrow_f \bar x} \widehat\partial f(x).
\]
For any $\bar x \notin \dom f$, we set $\widehat\partial f(\bar x) = \partial f(\bar x) = \emptyset$. When $f$ is locally Lipschitz continuous at $\bar{x}$, $\Conv\partial f(\bar x)$ equals to the Clarke subdifferential $\partial_{C} f(\bar x)$. We further say $f$ is subdifferentially regular at $\bar{x}\in \dom f$ if $f$ is lsc at $\bar{x}$ and $\widehat\partial f(\bar{x}) = \partial f(\bar{x})$.  When $f$ is proper and convex, $\widehat\partial f$, $\partial f$, and $\partial_{C} f$ coincide with the concept of the subdifferential in convex analysis.

\gap
Finally, we introduce the notion of function convergence. A sequence of functions $\{f^k: \R^n \to \ER\}$ is said to converge pointwise to $f: \R^n \rightarrow \ER$, written $f^k \Pointconv f$, if $\lim_{k \to +\infty} f^k(x) = f(x)$ for any $x \in \R^n$. 
The sequence $\{f^k\}$ is said to epi-converge to $f$, written $f^k \Epiconv f$, if for any $x$, it holds
	\[
		\left\{\begin{array}{ll}
        \displaystyle\liminf_{k \rightarrow +\infty} f^k(x^k) \geq f(x) &\;\text{for every sequence }x^k \rightarrow x,\\
        \displaystyle\limsup_{k \rightarrow +\infty} f^k(x^k) \leq f(x) &\;\text{for some sequence }x^k \rightarrow x.
        \end{array}\right.
	\]
The sequence $\{f^k\}$ is said to converge continuously to $f$, written $f^k \Contconv f$, if $\lim_{k \rightarrow +\infty} f^k(x^k) = f(x)$ for any $x$ and any sequence $x^k \rightarrow x$.

\section{Approachable difference-of-convex functions.}
\label{sec2: ADC functions}
In this section, we formally introduce a class of functions that can be asymptotically approximated by DC functions. A new concept of subdifferential that is defined through the approximating functions is proposed. At the end of this section, we provide several examples that demonstrate  the introduced concepts.

\subsection{Definitions and properties.}
\label{subsec: ADC and properties}

An extended-real-valued function can be approximated by a sequence of functions in various {notions of convergence}, as comprehensively investigated in \cite[Chapter 7(A-C)]{rockafellar2009variational}. Among these approaches, epi-convergence has a notable advantage in its ability to preserve the global minimizers \cite[Theorem 7.31]{rockafellar2009variational}.  Our focus lies on a particular class of approximating functions, wherein each function exhibits a DC structure. 

\begin{definition}
    A function $f$ is said to be \textsl{DC on its domain} if there exist proper, lsc and  convex functions $g, h: \R^n \to \ER$ such that $\dom f = [\,\dom g \cap \dom h\,]$ and $f(x) = g(x) - h(x)$ for any $x \in \dom f$.
\end{definition}

\noindent With this definition, we introduce the concept of ADC functions.

\begin{definition}[ADC functions]
\label{def:ADC}
    Let $f: \R^n \rightarrow \ER$ be a proper function.\\
    (a) $f$ is said to be \textsl{pointwise approachable DC (p-ADC)} if there exist proper functions $\{f^k: \R^n \to \ER\}$, DC on their respective domains, such that $f^k \Pointconv f$.\\
    (b) $f$ is said to be \textsl{epigraphically approachable DC (e-ADC)} if there exist proper functions $\{f^k: \R^n \to \ER\}$, DC on their respective domains, such that $f^k \Epiconv f$.\\
    (c) $f$ is said to be \textsl{continuously approachable DC (c-ADC)} if there exist proper functions $\{f^k: \R^n \to \ER\}$, DC on their respective domains, such that $f^k \Contconv f$.

{A function $f$ is said to be ADC associated with $\{f^k\}$ if $\{f^k\}$ confirms one of these convergence properties.}
By a slight abuse of notation, we denote the DC {decomposition} of each $f^k$ as $f^k=g^k-h^k$, although the equality may only hold for $x \in \dom f^k$.
\end{definition}

A p-ADC function may not be lsc. An example is given by $f(x) = \mathbf{1}_{\{0\}}(x) + 2 \cdot \mathbf{1}_{(0,+\infty)}(x)$, where for a set $C \subset \R^n$, we write $\mathbf{1}_{C}(x) = 1$ if $x \in C$ and $\mathbf{1}_{C}(x) = 0$ if $x \notin C$. In this case, $f$ is not lsc at $x=0$. However, $f$ is p-ADC associated with
$f^k(x) = \max\left(\, 0, \, 2kx+ 1\, \right) - \max\left(\, 0, \,2kx-1 \, \right)$.
In contrast, any e-ADC function must be lsc \cite[Proposition 7.4(a)]{rockafellar2009variational}, and any c-ADC function is continuous \cite[Theorem 7.14]{rockafellar2009variational}.

\gap
The relationships among different notions of function convergence, including the unaddressed uniform convergence, have been thoroughly examined in \cite{rockafellar2009variational}. Generally, pointwise convergence and epi-convergence do not imply one another, but they coincide when the sequence $\{f^k\}$ is asymptotically equi-lsc everywhere \cite[Theorem 7.10]{rockafellar2009variational}.
In addition, {$\{f^k\}$ converges continuously} to $f$ if and only if both $f^k \Epiconv f$ and $(-f^k) \Epiconv (-f)$ are satisfied \cite[Theorem 7.11]{rockafellar2009variational}. 
While verifying epi-convergence is often challenging, it becomes simpler for a monotonic sequence $\{f^k\}$ that converges pointwise to $f$ \cite[Proposition 7.4(c-d)]{rockafellar2009variational}.


\subsection{Subdifferentials of ADC functions.}
\label{sub: subgradients}

Characterizing the limiting and Clarke subdifferentials can be challenging when dealing with functions that exhibit  complex composite structures. Our focus in this subsection is on { numerically computable approximations of the limiting subdifferentials}. We begin with the definitions.

\begin{definition}[{approximate} subdifferentials]
\label{df:appro_subdiff}
	Consider an ADC function $f: \R^n \rightarrow \ER$ associated with $\{f^k= g^k - h^k\}$. 
	The \textsl{{approximate subdifferential}} of $f$ (associated with $\{f^k=g^k-h^k\}$) at $\bar x \in \R^n$ is defined as
	\[
	\begin{array}{rl}
		\partial_{A} f(\bar x)
		& \triangleq \, \bigcup\limits_{ x^k \rightarrow \bar x} \Limsup_{k \rightarrow +\infty} \;\big[\partial g^k(x^k) - \partial h^k(x^k)\big].
	\end{array}
	\]
	The \textsl{{approximate horizon subdifferential}} of $f$ (associated with $\{f^k=g^k-h^k\}$) at $\bar x \in \R^n$ is defined as
	\[
		\partial^\infty_{A} f(\bar x)
		 \triangleq \, \bigcup\limits_{ x^k \rightarrow \bar x}{\Limsup_{\,k \rightarrow +\infty}}^\infty \;\big[\partial g^k(x^k) - \partial h^k(x^k)\big].
	\]
\end{definition}

Unlike the limiting subdifferential which requires $x^k \rightarrow_f \bar x$, $\partial_{A} f(x)$ is defined using {all the} sequences $x^k \rightarrow \bar x$ without necessitating the convergence of function values.
It follows directly from the definitions that the mappings $x \mapsto \partial_A f(x)$ and $x \mapsto \partial_A^\infty f(x)$ are osc. The following proposition presents a sufficient condition for $\partial_{A} f(\bar x) = \partial f(\bar x) = \emptyset$ at any $\bar x \notin \dom f$.

\begin{proposition}
\label{prop:approach_osc}
    Let $\bar x \notin \dom f$. Then $\partial_{A} f(\bar x) = \emptyset$
    if for any sequence $x^k \rightarrow \bar x$, we have $x^k \notin \dom f^k$ for all sufficiently large $k$. The latter condition is particularly satisfied whenever $\dom f$ is closed and $\dom f^k \subset \dom f$ for all sufficiently large $k$.
\end{proposition}
\begin{proof}
    Note that for any $x^k \rightarrow \bar x \notin \dom f$, we have $[\partial g^k(x^k) - \partial h^k(x^k)] = \emptyset$ for all sufficiently large $k$ due to $x^k \notin \dom f^k = [\dom g^k \cap \dom h^k]$. Thus, $\partial_{A} f(\bar x) = \emptyset$ for any $\bar x \notin \dom f$. 
\end{proof}

In the subsequent analysis, we restrict our attention to $\bar x \in \dom f$. Admittedly, the set $\partial_A f(\bar x)$ depends on {the approximating sequence $\{f^k\}$ and} the  DC decomposition of each $f^k$,  which may contain irrelevant information concerning the local geometry of $\operatorname{epi} f$. In fact, for a given ADC function $f$, we can make the set $\partial_A f(\bar x)$ arbitrarily large by adding the same nonsmooth functions to both $g^k$ and $h^k$.
By Attouch's theorem (see for example \cite[Theorem 12.35]{rockafellar2009variational}),  for  proper, lsc, convex functions $f$ and $\{f^k\}$, if $f^k \Epiconv f$, we immediately have $\partial_{A} f = \partial f$ when taking $g^k=f^k$ and $h^k=0$.
In what follows, we further explore the relationships among $\partial_{A} f$ and other commonly employed subdifferentials in the literature beyond the convex setting. As it turns out, with respect to an arbitrary DC function $f^k$ that is lsc, $\partial_{A} f(\bar x)$ contains the limiting subdifferential of $f$ at any $\bar x \in \dom f$ whenever $f^k \Epiconv f$. 

\begin{theorem}[subdifferentials relationships]
\label{thm:eADC_subdiff}
	Consider an ADC function $f: \R^n \to \ER$. The following statements hold for any $\bar x \in \dom f$.\\
	(a) If $f$ is e-ADC associated with $\{f^k\}$ and $f^k$ is lsc, then $\partial f(\bar x) \subset \partial_A f(\bar x)$ and $\partial^\infty f(\bar x) \subset \partial^\infty_A f(\bar x)$.\\
	(b) If $f$ is locally Lipschitz continuous and bounded from below, then there exists a sequence of DC functions $\{f^k\} $ such that $f^k \Contconv f$, $\partial f(\bar x) \subset \partial_A f(\bar x) \subset \partial_C f(\bar x)$, and $\partial^\infty_A f(\bar x) = \{0\}$. Consequently, $\Conv \partial_{A} f(\bar x) = \partial_{C} f(\bar x)$, the set $\partial_{A} f(\bar x)$ is nonempty and  bounded, and $\partial f(\bar x) = \partial_{A} f(\bar x)$ when $f$ is subdifferentially regular at $\bar x$.
\end{theorem}
\begin{proof}
    (a) {Let $g^k - h^k$ be a DC decomposition of $f^k$.} Since $f$ is e-ADC, it must be lsc \cite[Proposition 7.4(a)]{rockafellar2009variational}. {Using epi-convergence of $\{f^k\}$ to $f$, we know from \cite[corollary 8.47(b)]{rockafellar2009variational} and \cite[Proposition 8.46(e)]{rockafellar2009variational} that any element of $\partial f(\bar x)$ can be generated as a limit of regular subgradients at $x^k$ with $x^k \rightarrow_N \bar{x}$ and $f^k(x^k) \rightarrow_N f(\bar x)$ for some  $N \in \mathbb{N}_\infty$}. Indeed, we can further restrict  $x^k \in \dom f^k$ since $f^k(x^k) \rightarrow_N f(\bar x)$ and $\bar x \in \dom f$. Then, we have
    \[
		\partial f(\bar x) \subset
		\bigcup\limits_{x^k (\in \dom f^k) \rightarrow \bar x} \Limsup_{k \rightarrow +\infty} {\widehat\partial} f^k(x^k) \subset
		\bigcup\limits_{x^k (\in \dom f^k) \rightarrow \bar x} \Limsup_{k \rightarrow +\infty} \big[\partial g^k(x^k) - \partial h^k(x^k)\big] \subset \partial_{A} f(\bar x),
    \]
    {where the second inclusion can be verified as follows: Firstly, due to the lower semicontinuity of $f^k$ and $h^k$,  and $x^k \in \dom f^k \subset \dom g^k$, it follows from the sum rule of regular subdifferentials \cite[corollary 10.9]{rockafellar2009variational} that $\widehat\partial g^k(x^k) \supset \widehat\partial f^k(x^k) + \widehat\partial h^k(x^k)$. Consequently, $\widehat\partial f^k(x^k) \subset \widehat\partial g^k(x^k) - \widehat\partial h^k(x^k) = \partial g^k(x^k) - \partial h^k(x^k)$ since $g^k$ and $h^k$ are proper and convex \cite[Proposition 8.12]{rockafellar2009variational}.} Similarly, by \cite[corollary 8.47(b)]{rockafellar2009variational}, we have
    {
    \[
        \partial^\infty f(\bar x)
        \subset \bigcup\limits_{x^k (\in \dom f^k) \rightarrow \bar x} {\displaystyle\Limsup_{k \rightarrow +\infty}}^\infty\;{\widehat\partial} f^k(x^k)
	\subset \bigcup\limits_{x^k (\in \dom f^k) \rightarrow \bar x} {\displaystyle\Limsup_{k \rightarrow +\infty}}^\infty \big[\partial g^k(x^k) - \partial h^k(x^k)\big] \subset \partial^\infty_{A} f(\bar x).
    \]}
    
    (b) For a locally Lipschitz continuous function $f$, consider its Moreau envelope $e_{\gamma} f(x) \triangleq \inf_{z} \{f(z) + \|z - x\|^2/(2\gamma)\}$ and the set-valued mapping $P_{\gamma f}(x) \triangleq \operatornamewithlimits{argmin}_{z} \{f(z) + \|z - x\|^2/(2\gamma)\}$. For any  sequence $\gamma_k \downarrow 0$, we demonstrate in the following that $\{f^k \triangleq e_{\gamma_k} f \}$ is the desired sequence of approximating functions. Firstly, since $f$ is bounded from below, it must be prox-bounded and, thus, each $f^k$ is continuous and $f^k(\bar x) \uparrow f(\bar x)$ for all $\bar x$ (cf. \cite[Theorem 1.25]{rockafellar2009variational}). By the continuity of $f$ and $f^k$, we have $f^k \Contconv f$ from \cite[Proposition 7.4(c-d)]{rockafellar2009variational}. 
  It then follows from part (a) that $\partial f(\bar x) \subset \partial_A f(\bar x)$. Consider the following DC decomposition of each $f^k$:
    \[
        f^k(x) = \underbrace{\frac{\|x\|^2}{2\gamma_k}}_{\triangleq g^k(x)} - \underbrace{\sup_{z \in \R^n} \left\{-f(z)  - \frac{\|z\|^2}{2 \gamma_k} + \frac{z^\top x}{\gamma_k} \right\}}_{\triangleq h^k(x)} \qquad x \in \R^n.
    \]
    {It is clear that $f(z) + \|z\|^2/(2\gamma_k) + z^\top x/\gamma_k$ is level-bounded in $z$ locally uniformly in $x$, since for any $r \in \R$ and any bounded set $X \subset \R^n$, the set
    \[
    	\left\{ z \in \R^n \middle|\, x \in X, f(z) + \frac{\|z\|^2}{2 \gamma_k} - \frac{z^\top x}{\gamma_k} \leq r\right\}
    	\subset \left\{ z \in \R^n \middle|\, x\in X, \|z - x\|^2 \leq \|x\|^2 + 2\gamma_k \left[r - \inf_{z} f(z)\right] \right\}
    \]
    is bounded.} Due to the level-boundedness condition, we can apply the subdifferential formula of the parametric minimization \cite[Theorem 10.13]{rockafellar2009variational} to get
    \[
        \partial(-h^k)(x)
        \subset \bigcup_{z \in P_{\gamma_k f}(x)} \left\{y \,\middle|\, (0,y) \in \partial_{(z,x)}\left(f(z)+ \frac{\|z\|^2}{2 \gamma_k} - \frac{z^\top x}{\gamma_k} \right)\right\}
        {\subset} \bigcup_{z \in P_{\gamma_k f}(x)}\left\{\partial f(z) - \frac{x}{\gamma_k}\right\},
    \]
    {where the last inclusion is due to the calculus rules \cite[Proposition 10.5 and exercise 8.8(c)]{rockafellar2009variational}.} Since $h^k$ is convex, we have $-\partial h^k(x) = \partial_C (-h^k)(x) = \Conv \partial(-h^k)(x)$ by \cite[Theorem 9.61]{rockafellar2009variational}, which further yields that
    \begin{equation}\label{eq:subdiff_outter}
       \left[\partial g^k(x) - \partial h^k(x)\right] \subset \Conv\, \bigcup \left\{\partial f(z) \,\middle|\, z \in P_{\gamma_k f}(x) \right\} \qquad\forall\, x \in \R^n, \; {k \in \N}.
    \end{equation}
For any $x^k \rightarrow \bar x$ and any $z^k \in P_{\gamma_k f}(x^k)$, we have
 \[
    \frac{1}{2\gamma_k} \|z^k - x^k\|^2 + \inf_{x} f(x) \leq \frac{1}{2\gamma_k} \|z^k - x^k\|^2 + f(z^k) \leq \frac{1}{2\gamma_k} \|\bar x - x^k\|^2 + f(\bar x).
 \]
Then, $\|z^k - x^k\| \leq \sqrt{\|\bar x - x^k\|^2 + 2\gamma_k [f(\bar x) - \inf_{x} f(x)]} \rightarrow 0$ due to the assumption that $f$ is bounded from below and therefore $z^k \rightarrow \bar x$. By the {local Lipschitz continuity} of $f$, it follows from \cite[Theorem 9.13]{rockafellar2009variational} that {the mapping $\partial f: x \mapsto \partial f(x)$ is locally bounded at $\bar x$. Thus,} there is a bounded set $S$ such that $\bigcup \{\partial f(z^k) \mid z^k \in P_{\gamma_k f}(x^k)\} \subset S$ for all sufficiently large $k$. It follows directly from \cite[Example 4.22]{rockafellar2009variational} and the definition of the approximate horizon subdifferential that $\partial^\infty_{A} f(\bar x) = \{0\}$.

Next, we will prove  $\partial_A f(\bar x) \subset \partial_C f(\bar x)$. For any $u \in \partial_A f(\bar x)$, from \eqref{eq:subdiff_outter}, there exist sequences of vectors $x^k \rightarrow \bar x$ and $u^k \rightarrow u$ with each $u^k$ taken from the convex hull of a bounded set $\bigcup \{\partial f(z^k) \mid z^k \in P_{\gamma_k f}(x^k)\}$. By Carath\'eodory's Theorem (see, e.g. \cite[Theorem 17.1]{rockafellar1970convex}), 
for each $k$, we have $u^k = \sum_{i=1}^{n+1} \lambda_{k,i} \;v^{k,i}$ for some nonnegative scalars $\{\lambda_{k,i}\}^{n+1}_{i=1}$ with $\sum_{i=1}^{n+1} \lambda_{k,i} = 1$ and a sequence $\big\{v^{k,i} \in \partial f(z^{k,i})\big\}^{n+1}_{i=1}$ with $\{z^{k,i} \in P_{\gamma_k f}(x^k)\}^{n+1}_{i=1}$.   
It is easy to see that the sequences $\{\lambda_{k,i}\}_{k \in \N}$ and $\{v^{k,i}\}_{k \in \N}$ are bounded for each $i$. We can then obtain convergent subsequences $\lambda_{k,i} \to_N \bar\lambda_i \geq 0$ with $\sum_{i=1}^{n+1} \bar\lambda_i = 1$ and $v^{k,i} \to_N {\bar v}^{\, i}$ for each $i$. Since $z^{k,i} \to \bar x$, we have ${\bar v}^{\, i} \in \partial f(\bar x)$ by using the outer semicontinuity of $\partial f$. Thus, $u^k \rightarrow_N u = \sum_{i=1}^{n+1} \bar\lambda_i \, {\bar v}^{\,i} \in \Conv \partial f(\bar x) = \partial_{C} f(\bar x)$.
This implies that $\partial_A f(\bar x) \subset \partial_{C} f(\bar x)$. The {rest of the statements} in (b) {follows} from the fact that $\partial_{C} f(\bar x)$ is nonempty and bounded whenever $f$ is locally Lipschitz continuous \cite[Theorem 9.61]{rockafellar2009variational}. 
\end{proof}

\gap
Under suitable assumptions, Theorem \ref{thm:eADC_subdiff}(b) guarantees the existence of an ADC decomposition that has its approximate subdifferential contained in the Clarke subdifferential of the original function. Notably, this decomposition may not always be practically useful due to the necessity of computing {the Moreau envelope for a generally nonconvex function.}
{Another noteworthy remark is that the assumptions and results of Theorem \ref{thm:eADC_subdiff} can be localized to any specific point $\bar{x}$. This  can be accomplished by defining a notion of ``local epi-convergence'' at $\bar x$ and extending the result of \cite[corollary 8.47]{rockafellar2009variational} accordingly.}

\subsection{Examples of ADC functions.}
\label{subsec:Examples ADC}

{In this subsection, we provide examples of ADC functions}, including functions that are discontinuous relative to their domains, with explicit and computationally tractable approximating sequences. Moreover, we undertake an investigation into the approximate subdifferentials of these ADC functions.

\gap\gap

\noindent
\textbf{Example 2.1} (implicitly convex-concave functions).
The concept of \textsl{implicitly convex-concave (icc) functions} is introduced in the monograph \cite{cui2021modern}, and is further generalized to extended-real-valued functions in \cite{li2022decomposition}. A proper function $f: \R^n \rightarrow \ER$ is icc {if there exists} a lifted function $\overline{f}: \R^n \times \R^n \rightarrow \ER$ such that the following three conditions hold:

(i) $\overline{f}(z,x) = +\infty$ if $z \notin \dom f, x \in \R^n$, and $\overline{f}(z,x) = -\infty$ if $z \in \dom f, x \notin \dom f$;

(ii) $\overline{f}(\cdot, x)$ is convex for any fixed $x \in \dom f$, and $\overline{f}(z, \cdot)$ is concave for any fixed $z \in \dom f$;

(iii) $f(x) = \overline{f}(x, x)$ for any $x \in \dom f$.

\gap
{\noindent A notable example of icc functions is the optimal value function $f_p$ in \eqref{eq:value_func}}, which is associated with the lifted function defined by (the subscripts/superscripts $p$ are omitted for brevity):
\begin{equation}\label{eq:icc lifted}
\begin{array}{lr}
    \overline f(z,x) \triangleq \displaystyle{\inf_{y\in {\R^d} }} \left\{(c+Cx)^\top y + \frac{1}{2} y^\top Q \, y
    \;\middle|\; Az + By\leq b\right\}
    \qquad (x, z) \in \dom f \times \dom f.
\end{array}
\end{equation}
Let $\partial_1 \overline{f}(\cdot,x)$ and $\partial_2(-\overline{f})(z,\cdot)$ denote the subdifferentials of the convex functions $\overline{f}(\cdot,x)$ and $(-\overline{f})(z,\cdot)$, respectively, for any $(x, z) \in \dom f \times \dom f$. 
For any $\gamma > 0$, the \textsl{partial Moreau envelope} of an icc function $f$ associated with $\overline{f}$ is given by
\begin{equation}
\label{eq:PME}
   \inf_{z \in {\R^n}}\left\{ \overline{f}(z,x) + \frac{1}{2 \gamma}\|z - x\|^2 \right\}
    = \underbrace{\frac{\|x\|^2}{2 \gamma}}_{\triangleq g_\gamma(x)} - \underbrace{\sup_{z \in {\R^n}}\left\{- \overline{f}(z,x) - \frac{\|z\|^2}{2 \gamma} + \frac{z^\top x}{\gamma} \right\}}_{\triangleq h_\gamma(x)}
    \qquad x \in \dom f.
    \vspace{-0.1in}
\end{equation}
This decomposition, established in \cite{li2022decomposition}, offers computational advantages compared to the standard Moreau envelope, as the maximization problem defining $h_\gamma$ is concave in $z$ for any fixed $x$. {In what follows, we present new results on the conditions under which the icc function $f$ is e-ADC and c-ADC based on the partial Moreau envelope. Additionally, we explore a relationship between $\partial_A f(\bar x)$ and $\partial_1 \overline{f}(\bar x, \bar x) - \partial_2 (-\overline{f}) (\bar x, \bar x)$, where the latter is known to be an outer estimate of $\partial_C f(\bar x)$ \cite[Proposition 4.4.26]{cui2021modern}. The proof is deferred to Appendix A.}

\begin{proposition}\label{proposition: icc ADC}
    Let $f: {\R^n} \to \ER$ be a proper, lsc, icc function associated with $\overline{f}$, where $\dom f$ is closed and $\overline{f}$ is lsc on ${\R^n} \times \dom f$, bounded below on $\dom f \times \dom f$, and {continuous relative to $\interior(\dom f) \times \interior(\dom f)$}. Given a sequence of scalars $\gamma_k \downarrow 0$, we have:\\
    (a) $f$ is e-ADC associated with $\{f^k\}$, where each $f^k(x) \triangleq {g_{\gamma_k}(x) - h_{\gamma_k}(x)} + \delta_{\dom f}(x)$. In addition, if $\dom f = {\R^n}$, then $f$ is c-ADC associated with $\{f^k\}$.\\
    (b) $\partial_{A} f(\bar x) \subset \partial_1 \overline{f}(\bar x,\bar x) - \partial_2 (-\overline{f})(\bar x,\bar x)$ and $\partial^\infty_A f(\bar x) = \{0\}$ for any $\bar x \in \interior(\dom f)$.
\end{proposition}

\gap

\gap\gap
\noindent
\textbf{Example 2.2} (VaR for continuous random variables).
Given a {continuous} random variable $Y: \Omega \to \R$, { its conditional value-at-risk (CVaR)  at a confidence level $\alpha \in (0,1)$ is defined as $\mbox{CVaR}_{\alpha}(Y) \triangleq \mathbb{E}[\,Y \mid Y \geq \mbox{VaR}_{\alpha}(Y)]$, where 
$\mbox{VaR}_{\alpha}$ is the value-at-risk given  in Example 1.2 (see, e.g., \cite{rockafellar2000optimization}).}
For any $\alpha \in (0,1)$ and  $k > 1/{\alpha}$, we define
\begin{equation}\label{eq:VaR_ADC}
    g^k(x) \triangleq [k(1-\alpha)+1]\, {\mbox{CVaR}_{\alpha-1/k}}[\, c(x,Z)],
    \quad
    h^k(x) \triangleq k(1-\alpha)\, {\mbox{CVaR}_{\alpha}}[\, c(x,Z)]
    \quad x \in \R^n.
\end{equation}
The following properties of VaR for continuous random variables hold,
with proofs provided in Appendix A.

\begin{proposition}\label{proposition: ADC of Var}
    Let $c:\R^n \times \R^m \to \R$ be a lsc function and $Z:\Omega \to \R^m$ be a random vector. Suppose that $c(\cdot, z)$ is convex for any fixed $z \in \R^m$, and $c(x,Z)$ is {a random variable having a continuous distribution induced by that of $Z$ for any fixed $x \in \R^n$.} Additionally, assume that $\mathbb{E}[\,|c(x,Z)|\,] < +\infty$ for any $x \in \R^n$. For any given constant $\alpha \in (0,1)$, the following properties hold.\\
    (a) $\mbox{VaR}_\alpha [\, {c(\cdot,Z)} ]$ is lsc and e-ADC associated with $\{g^k-h^k\}$ {(with the definitions of $g^k$ and $h^k$ in \eqref{eq:VaR_ADC})}. 
    Additionally, if $c(\cdot,\cdot)$ is continuous, then $\mbox{VaR}_\alpha [\, {c(\cdot, Z)} ]$ is continuous and c-ADC associated with $\{g^k-h^k\}$.\\
    (b) If there exists a measurable function $\kappa: \R^m \to \R_+$ such that $\mathbb{E}[\,\kappa(Z)] < +\infty$ and $|c(x,z) - c(x^\prime,z)| \leq \kappa(z) \|x-x^\prime\|$  for all $x,x^\prime \in \R^n$ and $z \in \R^m$, then for any $\bar x \in \R^n$,
    \[
        \partial_{A} \mbox{VaR}_\alpha [\, c(\cdot,Z)](\bar x) = \bigcup_{x^k \rightarrow \bar x} \Limsup_{k \rightarrow +\infty} \mathbb{E}\left[\,\partial_1\, c(x^k,Z) \,\middle|\, \mbox{VaR}_{\alpha-1/k}[\,c(x^k,Z)] < c(x^k,Z) < \mbox{VaR}_{\alpha}[\,c(x^k,Z)] \,\right],
        \vspace{-0.08in}
    \]
    where $\mathbb{E}[\,\mathcal{A}(Z) \mid B]$ for a random set-valued mapping $\mathcal{A}$ and an event $B$ is defined as the set of conditional expectations $\mathbb{E}[\,a(Z) \mid B]$ for all measurable selections $a(Z) \in \mathcal{A}(Z)$.
\end{proposition}

\section{The convex composite ADC functions and minimization.}
\label{sec3:composite model}

This section aims to derive necessary optimality conditions for \eqref{eq:cvx_composite}, particularly focusing on the inner function $f_p$ that lacks local Lipschitz continuity. 
{Throughout the rest of this paper, we assume that $\varphi_p: \R \to \R \cup \{+\infty\}$ is proper, convex, lsc and $f_p: \R^n \to \R$ is real-valued for all $p=1,\cdots,m$.
Depending on whether $\varphi_p$ is nondecreasing or not, we partition $\{1,\cdots,m\}$ into two categories:
\begin{equation}\label{defn: index for varphi_p}
    I_1 \triangleq \left\{\, p \in \{1,\cdots,m\} \,\middle|\, \varphi_p \text{ nondecreasing}\right\}
    \quad\text{and}\quad
    I_2 \triangleq\{1,\cdots,m\} \backslash I_1.
\end{equation}
We do not specifically address the case where $\varphi_p$ is nonincreasing, as one can always redefine $\widetilde\varphi_p(t) = \varphi_p(-t)$ and $\widetilde f_p(x) = -f_p(x)$, enabling the treatment of these indices in the same manner as those in $I_1$. Therefore, the set $I_2$ should be viewed as the collection of indices $p$ where $\varphi_p$ is not monotone. We further make the following assumptions on the functions $\varphi_p$ and $f_p$.}


 \vskip 0.1in

\begin{center}
	\fbox{\parbox{0.98\textwidth}{
 \noindent{\bf Assumption 1} For each $p$, we have
	\begin{itemize}
	\item[(a)] $f_p$ is e-ADC associated with $\{f^k_p = g^k_p-h^k_p\}_{k \in \N}$, and $\dom g^k_p = \dom h^k_p = \R^n$;
	
	\item[(b)] $-\infty < \displaystyle\liminf_{x^\prime \rightarrow x,\, k \rightarrow +\infty} f^k_p(x^\prime) \leq \limsup_{x^\prime \rightarrow x,\, k \rightarrow +\infty} f^k_p(x^\prime) < + \infty$ for all $x \in \R^n$;

    \item[(c)] $\big[F^k_p \triangleq \varphi_p \circ f^k_p\big] \Epiconv F_p$.
	\end{itemize}
	}}
\end{center}

From Assumption 1(a), each $f^k_p$ is locally Lipschitz continuous since any real-valued convex function is locally Lipschitz continuous. Obviously, $f^k_p \Contconv f_p$ is sufficient for Assumption 1(b) to hold. Since $f^k_p \Epiconv f_p$, we have $\liminf_{x^\prime \rightarrow x, k \rightarrow +\infty} f^k_p(x^\prime) \geq f_p(x) > -\infty$ for each $p$ at any $x \in \R^n$. However, $\limsup_{x^\prime \rightarrow x, k \rightarrow +\infty} f^k_p(x^\prime) <+\infty$ does not hold trivially. For example, consider a continuous function $f$ and 
\[
    f^k(x) =\left\{\begin{array}{cl}
			f(x) + k^2 x + k &\text{ if } x \in [- 1/k, 0]\\
			f(x) - k^2 x + k &\text{ if } x \in (0, 1/k]\\
			f(x) &\text{ otherwise}
		  \end{array}\right.,
\]
which results in $f^k \Epiconv f$ but ${\limsup_{k \rightarrow +\infty} f^k(0) = + \infty}$. Additionally, Assumption 1(b) ensures that at each point $x$ and for any sequence $x^k \to x$, the sequence $\{f^k_p(x^{k})\}_{k \in \N}$ must be bounded. 

\gap

It follows from \cite[Exercise 7.8(c)]{rockafellar2009variational} and \cite[Theorem 2.4]{royset2022consistent} that 
there are several sufficient conditions for Assumption 1(c) to hold, which differ based on the monotonicity of each $\varphi_p$: (i) For $p \in I_1$, either $\varphi_p$ is real-valued or $f^k_p \leq f_p$; (ii) For $p \in I_2$, $f_p$ is c-ADC and for all $x$ with $f_p(x) \in \operatorname{bdry}(\dom {\varphi_p})$, there exists a sequence ${x^k} \to x$ with $f(x^k) \in \interior(\dom \varphi_p)$. In addition, 
according to \cite[Proposition 7.4(a)]{rockafellar2009variational}, Assumption 1(c) implies that $F_p = \varphi_p \circ f_p$ is lsc. We also note that Assumption 1(c) doesn't necessarily imply $\sum_{p=1}^{m} F^k_p \Epiconv \sum_{p=1}^{m} F_p$. 
To maintain epi-convergence under addition of functions, one may refer to the sufficient conditions in \cite[Theorem 7.46]{rockafellar2009variational}.

\subsection{Asymptotic stationarity under epi-convergence.}
\label{subsec:asymptotic stationarity}

In this subsection, we introduce a novel stationarity concept for problem \eqref{eq:cvx_composite}, grounded in a monotonic decomposition of univariate convex functions. We demonstrate  that under certain constraint qualifications, epi-convergence of approximating functions ensures this stationarity concept as a necessary optimality condition. Alongside the known fact that epi-convergence {also ensures the consistency  of global optimal solutions} \cite[Theorem 7.31(b)]{rockafellar2009variational}, this highlights the usefulness of epi-convergence as a  tool for studying the approximation of problem \eqref{eq:cvx_composite}.

\gap
{The following lemma is an extension of \cite[Lemma 6.1.1]{cui2021modern} from real-valued univariate convex functions to extended-real-valued univariate convex functions.}

\begin{lemma}[{a monotonic decomposition of univariate convex functions}]
\label{lem:cvx_monotone_decomposition}
    Let $\varphi: \R \rightarrow \ER$ be a proper, lsc and convex function. Then there exist a proper, lsc, convex and nondecreasing function $\varphi^{\uparrow}$,  as well as a proper, lsc, convex and nonincreasing function $\varphi^{\downarrow}$, such that $\varphi = \varphi^\uparrow + \varphi^\downarrow$. In addition, if $\interior(\dom\varphi) \neq \emptyset$, then $\partial \varphi(z) = \partial\varphi^\uparrow(z) + \partial\varphi^\downarrow(z)$ for any $z \in \dom\varphi$.
\end{lemma}
\begin{proof}
	From the convexity of $\varphi$, $\dom \varphi$ is an interval on $\R$, possibly unbounded. 
	In fact, we can explicitly construct $\varphi^\uparrow$ and $\varphi^\downarrow$ in following two cases.

	\noindent\underline{\textsl{Case 1.}} If $\varphi$ has no direction of recession, i.e., there does not exist $d \neq 0$ such that for any $z$, $\varphi(z+\lambda d)$ is a nonincreasing function of $\lambda>0$, it follows from \cite[Theorem 27.2]{rockafellar1970convex} that $\varphi$ attains its minimum at some $z^\ast \in \dom \varphi$. Define
	\[
		\varphi^\uparrow(z) = \left\{
		\begin{array}{cl}
			\varphi(z^\ast) &\text{if } z \leq z^\ast\\
			\varphi(z) &\text{if } z > z^\ast
		\end{array}\right.
	\quad \mbox{and}\quad 
		\varphi^\downarrow(z) = \left\{
		\begin{array}{cl}
			\varphi(z) - \varphi(z^\ast) &\text{if } z \leq z^\ast\\
			0 &\text{if } z > z^\ast
		\end{array}\right..
	\]
    
	Observe that $\emptyset \neq \interior(\dom\varphi) \subset \big[\interior(\dom\varphi^\uparrow) \cap \interior(\dom\varphi^\downarrow)\big]$. Consequently, from \cite[Theorem 23.8]{rockafellar1970convex}, we have $\partial\varphi(z) = \partial\varphi^\uparrow(z) + \partial\varphi^\downarrow(z)$ for any $z \in \R$. 

	\noindent\underline{\textsl{Case 2.}} Otherwise, there exists $d \neq 0$ such that for any $z \in \R$, $\varphi(z + \lambda d)$ is a nonincreasing function of $\lambda > 0$. Consequently, $\dom\varphi$ {must} be an unbounded interval on $\R$. Let $d = 1$ (or $-1$) be such a recession direction, then $\varphi$ is nonincreasing (or nondecreasing) on $\R$. We can set $\varphi^\uparrow = 0$ and $\varphi^\downarrow = \varphi$ (or $\varphi^\uparrow = \varphi$ and $\varphi^\downarrow = 0$). 
    In this case, it is obvious that $\partial\varphi(z) = \partial\varphi^\uparrow(z) + \partial\varphi^\downarrow(z)$ for any $z \in \R$. The proof is thus completed. 
\end{proof}

In the subsequent analysis, we use $\varphi^\uparrow$ and $\varphi^\downarrow$ to denote the monotonic decomposition of any univariate, proper, lsc, and convex function $\varphi$ constructed in the proof of Lemma \ref{lem:cvx_monotone_decomposition} and, in particular, we take $\varphi^\downarrow = 0$ whenever $\varphi$ is nondecreasing. We are now ready to present the definition of asymptotically stationary points.

\gap
\begin{definition}[asymptotically stationary points]
	Let each $f_p$ be an ADC function associated with $\{f^k_p=g^k_p-h^k_p\}_{k \in \N}$. For each $p$, define
    \begin{equation}\label{def:T_p}
        T_p(x) \triangleq \left\{t_p \,\middle|\, \exists\, N \in \mathbb{N}_\infty^\sharp, x^k \rightarrow x\text{ with } f^k_p(x^k) \rightarrow_N t_p\right\} \qquad x \in \R^n.
    \end{equation}
    We say that $\bar x$ is an \textsl{asymptotically stationary (A-stationary) point} of problem \eqref{eq:cvx_composite} if for each $p$, there exists $y_p \in \bigcup \{\partial \varphi_p(t_p) \mid t_p \in T_p(\bar x)\}$ such that
    \begin{equation}\label{eq:A-stationary}
	0 \, \in \,
	\sum\limits_{p=1}^{m} \left(\, \big\{ y_p \, \partial_{A} f_p(\bar x)\big\} \cup {\big[ \pm\partial^\infty_A f_p(\bar x)\backslash \{0\} \,\big]} \,\right).
    \end{equation}
    We say that $\bar x$ is a \textsl{weakly asymptotically stationary (weakly A-stationary) point} of problem \eqref{eq:cvx_composite} if for each $p$, there exist {$\bar t_p \in T_p(\bar x)$, $y_{p,1} \in \partial \varphi^\uparrow_p(\bar t_p)$ and $y_{p,2} \in \partial \varphi^\downarrow_p(\bar t_p)$} such that
    \[
	0 \in
	\sum\limits_{p=1}^{m} \left(\, \left\{ y_{p,1} \, \partial_{A} f_p(\bar x) + y_{p,2} \, \partial_{A} f_p(\bar x) \right\}
        \cup
        \big[ \pm\partial^\infty_A f_p(\bar x)\backslash \{0\} \big] \,\right).
    \]
\end{definition}

\begin{remark}\label{remark:A-stationary}
    (\romannumeral1) Given that the approximate subdifferential $\partial_A f_p$ is determined by the approximating sequence $\{f^k_p\}_{k \in \N}$ and their corresponding DC decompositions, the notion of (weak) A-stationarity also depends on these sequences and decompositions. 
    (\romannumeral2) It follows directly from Lemma \ref{lem:cvx_monotone_decomposition} that an A-stationary point must be a weakly A-stationary point if $\interior(\dom \varphi_p) \neq \emptyset$ for each $p=1,\cdots,m$. 
    (\romannumeral3) When each $\varphi_p$ is nondecreasing or nonincreasing, the concepts of weak A-stationarity and A-stationarity coincide. 
    (\romannumeral4) Given a point $\bar x$,  we can rewrite \eqref{eq:A-stationary} as
    \[
        0 \in \sum_{p \in I}\big[ \pm\partial^\infty_A f_p(\bar x)\backslash \{0\} \big] + \sum_{p \in \{1, \cdots, m\}\backslash I} \{y_p \, \partial_{A} f_p(\bar x)\}
    \]
    for some index set $I \subset \{1, \cdots, m\}$ that is potentially empty. For each $p \in I$, although the scalar $y_p$ does not explicitly appear in this inclusion, its existence implies that $\bigcup \{\partial\varphi_p(t_p) \mid t_p \in T_p(\bar x)\} \neq \emptyset$, which plays a role in ensuring $\bar{x} \in \dom (\varphi_p \circ f_p)$. For instance, if $f^k_p \Contconv f_p$ for some $p \in I$, then $T_p(\bar x) = \{f_p(\bar x)\}$, and the existence of $y_p \in \bigcup \{\partial\varphi_p(t_p) \mid t_p \in T_p(\bar x)\} = \partial\varphi_p(f_p(\bar x))$ yields $\bar x \in \dom(\varphi_p \circ f_p)$.
\end{remark}

\gap
In the following, we take a detour to compare the  A-stationarity with the stationarity defined in \cite{royset2022consistent}, where the author has focused on a more general composite problem {
\[
    \operatornamewithlimits{minimize}_{x\in \R^n} \; \varphi \left(f(x)\right),
\]
where $\varphi:\R^m \to \ER$ is proper, lsc, convex and $f \triangleq (f_1,\cdots, f_m): \R^n \to \R^m$ is a locally Lipschitz continuous mapping.}  
Consider the special case where $\varphi(z) = \sum_{p=1}^{m} \varphi_p(z_p)$ with $z = (z_1, \cdots, z_m)$. 
Under this setting, a vector $\bar x$ is called a stationary point in \cite{royset2022consistent} if there {exist} $\bar y$ and $\bar z$ such that
\begin{equation}\label{eq:S_consistent}
    0 \in S(\bar x, \bar y, \bar z) \triangleq \big\{(f_1(\bar x), \cdots, f_m(\bar x)) - \bar z\big\} \times \big\{\partial\varphi_1(\bar z_1) \times \cdots \times \partial\varphi_m(\bar z_m) - \bar y\big\} \times \left(\sum_{p=1}^{m} \bar y_p \, \partial_{C} f_p(\bar x)\right),
\end{equation}
which can be equivalently written as
\begin{equation}\label{eq:stationary_royset}
    0 \in \sum_{p=1}^m \bar y_p \, \partial_C f_p(\bar x) \;\text{ for some }\; \bar y_p \in \partial\varphi_p(f_p(\bar x)) \quad p=1,\cdots,m.
\end{equation}

For any fixed ${k \in \N}$, {a surrogate} set-valued mapping $S^k$ can be defined similarly  as $S$ in \eqref{eq:S_consistent} by substituting $f_p$ and $\varphi_p$ with $f^k_p$ and $\varphi^k_p$ for each $p$. The cited paper provides sufficient conditions to ensure $\operatorname{Lim \, sup}_{k \rightarrow +\infty} (\operatorname{gph} S^k) \subset \operatorname{gph} S$,
which asserts  that any accumulation point $(\bar x, \bar y, \bar z)$ of a sequence $\{(x^k, y^k, z^k)\} $ with $0 \in S^k(x^k, y^k, z^k)$ yields a stationary point $\bar x$. 
Our study on the asymptotic stationarity differs from \cite{royset2022consistent} in the following aspects:
\gap
\begin{enumerate}
    \item Our outer convex function $\varphi$ is assumed to have the separable form $\sum_{p=1}^{m} \varphi_p$, while  \cite{royset2022consistent} allows a general proper, lsc, convex function. In addition, each $\varphi_p$ is fixed in our approximating problem while \cite{royset2022consistent} considers a sequence of convex functions $\{\varphi^k_p\}_{k \in \N}$ that epi-converges to $\varphi_p$.
    \item We do not require the inner function $f_p$ to be locally Lipschitz continuous.
    \gap\\
    \indent\indent If each $f_p$ is locally Lipschitz continuous and bounded from below, it then follows from {Theorem} \ref{thm:eADC_subdiff} that $f_p$ is c-ADC associated with $\{f^k_p=g^k_p-h^k_p\}_{k \in \N}$ such that $\partial f_p(x) \subset \partial_{A} f_p(x) \subset \partial_{C} f_p(x)$ and $\partial^\infty_{A} f_p(x) = \{0\}$ for any $x$. Moreover, by $f^k_p \Contconv f_p$, one has $T_p(x) = \{f_p(x)\}$. Thus, for any A-stationary point $\bar x$ induced by these ADC decompositions, there exists $\bar y_p \in \partial\varphi_p(f_p(\bar x))$ for each $p$ such that
        \begin{equation}\label{eq:stationary comparison}
            0 \in \sum_{p=1}^{m} \left\{ \bar y_p \, \partial_{A} f_p(\bar x)\right\}
            \;\subset\; \sum_{p=1}^{m} \left\{ \bar y_p \, \partial_{C} f_p(\bar x)\right\}.
        \end{equation}
    Hence, $\bar x$ is also a stationary point defined in \eqref{eq:stationary_royset}. Indeed, A-stationarity here can be sharper than the latter one as the last inclusion in \eqref{eq:stationary comparison} may not hold with equality.
    \gap\\
    \indent\indent When $f_p$ fails to be locally Lipschitz continuous for some $p$, {it is not known if \eqref{eq:S_consistent} is still a necessary condition for a local solution of \eqref{eq:cvx_composite}. This situation further complicates the fulfillment of conditions outlined in \cite[Theorem 2.4]{royset2022consistent}, especially the requirement of $f^k_p \Contconv f_p$, due to the potential discontinuity of $f_p$. As will be shown in Theorem \ref{thm:necessary_cond} below, despite these challenges, weak A-stationarity continues to be a necessary optimality condition under Assumption 1.}
\end{enumerate}

\gap
\noindent To proceed, for each $p$ and any $x \in \dom(\varphi_p \circ f_p)$, we define $S_p(x)$ to be a collection of sequences:
\begin{equation}\label{defn: Sp}
   { S_p(x) \,\triangleq \, \left\{\, \{x^k_p\}_{k \in \N}  \, \middle| \, x^k_p \rightarrow x\text{ with }\varphi_p(f^k_p(x^k_p)) \rightarrow \varphi_p(f_p(x)) \,\right\}.}
\end{equation}

\begin{theorem}[necessary conditions for optimality]
\label{thm:necessary_cond}
    Let $\bar x \in \bigcap_{p=1}^{m} \dom F_p$ be a local minimizer of problem \eqref{eq:cvx_composite}. Suppose that Assumption 1 and the following two conditions hold:\\
	(i) For each $p$ and any sequence $\{x^k_p\}_{k \in \N} \in S_p(\bar x)$, there is a positive integer $K$ such that
	\begin{equation}\label{eq:CQ1}
		0 \notin \partial_{C} f^k_p(x^k_p) \quad\text{or}\quad \mathcal N_{\dom\varphi_p}(f^k_p(x^k_p)) = \{0\} \quad\forall\, k \geq K,
	\end{equation}
	and
	\begin{equation}\label{eq:CQ2}
		\bigg[0 \in y_p \, \partial_A f_p(\bar x), \; y_p \in \bigcup \big\{\mathcal{N}_{\dom\varphi_p}(t_p) \mid t_p \in T_p(\bar x) \big\}\bigg]
        \quad\Longrightarrow\quad
        y_p = 0, \quad p=1, \cdots, m.
	\end{equation}
	(ii) One has
	\begin{equation}\label{eq:CQ3}
		\left[\sum\limits_{p=1}^{m} w_p = 0, \; w_p \in \partial^\infty (\varphi_p \circ f_p)(\bar x)\right]
        \quad\Longrightarrow\quad
        w_1 = \cdots = w_m = 0.
	\end{equation}
	Then $\bar x$ is an A-stationary point of \eqref{eq:cvx_composite}. Additionally, $\bar x$ is a weakly A-stationary point of \eqref{eq:cvx_composite} if $\interior(\dom \varphi_p) \neq \emptyset$ for each $p=1,\cdots,m$.
\end{theorem}
\begin{proof}
	By using Fermat's rule \cite[Theorem 10.1]{rockafellar2009variational} and the sum rule of the limiting subdifferentials \cite[Corrollary 10.9]{rockafellar2009variational} due to the condition \eqref{eq:CQ3}, we have
	\begin{equation}\label{eq:neccessary1}
	\begin{split}
		0
		&\;\in \partial \left[\sum\limits_{p=1}^{m} (\varphi_p \circ f_p)(\bar x)\right]
            \subset \sum\limits_{p=1}^{m} \partial (\varphi_p \circ f_p)(\bar x)
		\;\overset{(\rm\romannumeral1)}{\subset}\; \sum\limits_{p=1}^{m} \bigcup_{\{x^k_p\}_{k \in \N} \in {S_p(\bar x)}} \Limsup_{k \rightarrow +\infty} \;\partial (\varphi_p \circ f^k_p)(x^k_p)\\
		&\overset{(\rm\romannumeral2)}{\subset} \sum\limits_{p=1}^{m} \bigcup_{\{x^k_p\}_{k \in \N} \in {S_p(\bar x)}} \Limsup_{k \rightarrow +\infty} \;\bigcup\left\{ \partial (y^k_p \, f^k_p)(x^k_p) \,\middle|\, y^k_p \in \partial\varphi_p(f^k_p(x^k_p)) \right\}\\
		&\overset{(\rm\romannumeral3)}{\subset} \sum\limits_{p=1}^{m} \bigcup_{\{x^k_p\}_{k \in \N} \in {S_p(\bar x)}} \Limsup_{k \rightarrow +\infty} \;\left\{ y^k_p \, v^k_p \,\middle| \, y^k_p \in \partial\varphi_p(f^k_p(x^k_p)), \, v^k_p \in \partial_C f^k_p(x^k_p) \right\}\\
		&\overset{(\rm\romannumeral4)}{\subset} \sum\limits_{p=1}^{m} \bigcup_{\{x^k_p\}_{k \in \N} \in {S_p(\bar x)}} \Limsup_{k \rightarrow +\infty} \;\left\{ y^k_p \, v^k_p \,\middle|\, y^k_p \in \partial\varphi_p(f^k_p(x^k_p)),\, v^k_p \in \left[\partial g^k_p(x^k_p) - \partial h^k_p(x^k_p)\right] \right\}.
	\end{split}
	\end{equation}
	The inclusion $(\rm\romannumeral1)$ is due to $\varphi_p \circ f^k_p \Epiconv \varphi_p \circ f_p$ in Assumption 1(c) and approximation of subgradients under epi-convergence \cite[corollary 8.47]{rockafellar2009variational} and \cite[Proposition 8.46(e)]{rockafellar2009variational}; $(\rm\romannumeral2)$ follows from the nonsmooth Lagrange multiplier rule \cite[Exercise 10.52]{rockafellar2009variational} due to the {local Lipschitz continuity} of $f^k_p$ \cite[Example 9.14]{rockafellar2009variational} and the condition \eqref{eq:CQ1}; $(\rm\romannumeral3)$ and $(\rm\romannumeral4)$ use the calculus rules of the Clarke subdifferential \cite[Chapter 2.3]{clarke1990optimization}. For each $p$, any sequence $\{x^k_p\}_{k \in \N} \in S_p(\bar x)$ and any element
	\[
		\bar w_p \in \Limsup_{k \rightarrow +\infty} \;\left\{ y^k_p \, v^k_p \,\middle|\, y^k_p \in \partial\varphi_p(f^k_p(x^k_p)), v^k_p \in \big[\partial g^k_p(x^k_p) - \partial h^k_p(x^k_p)\big] \right\},
	\]
	there is a subsequence $w^k_p \rightarrow_N \bar w_p$ with $w^k_p = y^k_p \, v^k_p$ for some $N \in \mathbb{N}^\sharp_\infty$. Next, we show the existence of $\bar y_p \in \bigcup \{\partial\varphi_p(t_p) \mid t_p \in T_p(\bar x)\}$ for each $p$ such that
	\begin{equation}\label{eq:w_infty_split}
		\bar w_p \in \left\{\,\bar y_p \,\partial_{A}f_p(\bar x) \,\right\} \, \cup \, \big[ \pm \partial^\infty_A f_p(\bar x)\backslash\{0\} \,\big].
	\end{equation}
	By Assumption 1(b), the subsequence $\{f^k_p(x^k_p)\}_{k \in N}$ is bounded. Taking a subsequence if necessary, we can suppose that $f^k_p(x^k_p) \to_N \bar z_p \in T_p(\bar x)$. If $\{y^k_p\}_{k \in N}$ is unbounded, then $\{v^k_p\}_{k \in N}$ has a subsequence converging to $0$ and, thus, $0 \in \partial_A f_p(\bar x)$. Additionally, there exists $\widetilde y_p \neq 0$ such that
    \begin{equation}\label{eq:horizon2normal}
        \hskip -0.07in
        \frac{y^k_p}{|y^k_p|} \rightarrow_N \widetilde y_p \in {\Limsup_{k(\in N) \rightarrow +\infty}}^\infty \; \partial\varphi_p (f^k_p(x^k_p))
        \overset{(\rm\romannumeral5)}{=}{\Limsup_{k(\in N) \rightarrow +\infty}}^\infty \; \widehat\partial\varphi_p (f^k_p(x^k_p))
        \overset{(\rm\romannumeral6)}{\subset} \partial^\infty \varphi_p(\bar z_p) \overset{(\rm\romannumeral7)}{=} \mathcal{N}_{\dom\varphi_p}(\bar z_p).
    \end{equation}
    The equation $(\rm\romannumeral5)$ follows from \cite[Proposition 8.12]{rockafellar2009variational} by the convexity of $\varphi_p$. From $\{x^k_p\}_{k \in \N} \in {S_p(\bar x)}$ and $\bar x \in \dom F_p$, we must have $f^k_p(x^k_p) \in \dom\varphi_p$ for sufficiently large $k \in N$. Since $\varphi_p$ is lsc, it holds that $\varphi_p(\bar z_p) \leq \liminf_{k(\in N) \rightarrow +\infty} \varphi_p(f^k_p(x^k_p)) = \varphi_p(f_p(\bar x))$ and, thus, $\bar z_p \in \dom \varphi_p$. 
    {Also, notice that $\varphi_p$ is continuous relative to its domain as it is univariate convex and lsc \cite[Theorem 10.2]{rockafellar1970convex}.
    This continuity implies $\varphi_p(f^k_p(x^k_p)) \rightarrow_N \varphi_p(\bar z_p)$.} The inclusion $(\rm\romannumeral6)$ follows directly from the definition of the horizon subdifferential. Lastly, $(\rm\romannumeral7)$ is due to the lower semicontinuity of {the proper convex function} $\varphi$ and \cite[Proposition 8.12]{rockafellar2009variational}. Therefore, we have $(0 \neq)\widetilde y_p \in \bigcup\{\mathcal{N}_{\dom\varphi_p}(t_p) \mid t_p \in T_p(\bar x) \big\}$ {with $0 \in \widetilde y_p \partial_A \, f_p(\bar x)$ due to} $0 \in \partial_{A} f_p(\bar x)$, contradicting \eqref{eq:CQ2}. So far, we conclude that $\{y^k_p\}_{k \in N}$ is a bounded sequence. Suppose that $y^k_p \rightarrow_N \bar y_p$ and, thus, $\bar y_p \in \partial\varphi_p(\bar z_p)$ by the outer semicontinuity of $\partial\varphi_p$ \cite[Proposition 8.7]{rockafellar2009variational}.

    \gap
    
    \noindent\underline{\textsl{Case 1.}} If $\bar y_p = 0$, inclusion \eqref{eq:w_infty_split} holds trivially for $\bar w_p = 0$, and for $\bar w_p \neq 0$ we can find a subsequence $\{|y^k_p|\}_{k \in N^\prime} \downarrow 0$ such that {$\{|y^k_p| \, v^k_p\}_{k \in N^\prime}$ converges to $\bar w_p$ or $ -\bar w_p(\neq 0)$} with $v^k_p \in \big[\partial g^k_p(x^k_p) - \partial h^k_p(x^k_p)\big]$ for all $k \in N^\prime$. Therefore, \eqref{eq:w_infty_split} follows from
	\[
		\bar w_p \in \left[\,\Big(\pm {\displaystyle\Limsup_{k \rightarrow +\infty}}^\infty \;\big[\partial g^k_p(x^k_p) - \partial h^k_p(x^k_p)\big] \Big) \backslash \{0\}\,\right] \subset \big[\pm\partial^\infty_A f_p(\bar x) \backslash \{0\}\big].
	\]

	\noindent\underline{\textsl{Case 2.}} Otherwise, $\|v^k_p\| \rightarrow_N \|\bar w_p\| / |\bar y_p|$. This means that $\{v^k_p\}_{k \in N}$ is bounded. Suppose $v^k_p \to_N \bar v_p$. Then, $\bar v_p \in \operatorname{Lim \, sup}_{k \rightarrow +\infty} \big[\partial g^k_p(x^k_p) - \partial h^k_p(x^k_p)\big] \subset \partial_{A} f_p(\bar x)$, and \eqref{eq:w_infty_split} is evident from $\bar w_p = \bar y_p \, \bar v_p$.
    
    \vspace{0.05in}
    In either case, we have proved \eqref{eq:w_infty_split}. Combining \eqref{eq:neccessary1} with \eqref{eq:w_infty_split}, for some $\bar y_p \in \bigcup\{\partial \varphi_p(t_p) \mid t_p \in T_p(\bar x)\}$, we know that $\bar x$ is an A-stationary point of \eqref{eq:cvx_composite}. 
    {The final assertion follows from Remark \ref{remark:A-stationary}(ii).} 
\end{proof}

\subsection{An example of A-stationarity.}

We present an example to illustrate the concept of A-stationarity and to study its relationship with other known optimality conditions.

\gap\gap
\noindent
\textbf{Example 3.1} (bi-parametrized two-stage stochastic programs). 
Consider the following bi-parametrized two-stage stochastic program with fixed scenarios described in \cite{liu2020two}:
\begin{equation}\label{eq:2SP_1st}
    \Min_{x \in {\R^n}} \;\, \theta(x) + \frac{1}{m_1} \sum_{p=1}^{m_1} f_p(x) \quad
    \mbox{subject to} \;\; \phi_p (x) \leq 0, \quad p=1,\cdots,m_2,
\end{equation}
where $\theta, \phi_p: {\R^n} \to \R$ are convex, continuously differentiable for $p = 1, \cdots, m_2$, and $f_p$, as defined in \eqref{eq:value_func}, is real-valued for $p = 1, \cdots, m_1$. 
At $x = \bar x$, let $Y_p(\bar x)$ and $\Lambda_p(\bar x)$ represent the optimal solutions and multipliers for each second-stage problem \eqref{eq:value_func}. {Suppose that $Y_p(\bar x)$ and $\Lambda_p(\bar x)$ are bounded.}
Note that $\theta$ and $\phi_p$ are ADC functions since they are convex. Example 2.1 shows that $f_p$ is an ADC function, and therefore, problem \eqref{eq:2SP_1st} is a specific case of the composite model \eqref{eq:cvx_composite}. 
Given an A-stationary point $\bar x$ of \eqref{eq:2SP_1st}, under the assumptions of Example 2.1, we have
    \begin{equation}
    \begin{aligned}
    \label{eq:KKT}
        \hskip -0.1in
        0 &\in \nabla \theta(\bar x) + \frac{1}{m_1} \sum_{p=1}^{m_1}\big( \left\{ \partial_{A} f_p(\bar x) \right\} \cup \left[\,\pm \partial^\infty_{A} f_p(\bar x) \backslash \{0\} \right] \big) + \sum_{p=1}^{m_2} \bar\mu^{\,m_1 + p} \, \nabla \phi_p(\bar x)\\
        &\subset \nabla \theta(\bar x) + \frac{1}{m_1} \sum_{p=1}^{m_1} \left\{\partial_1 \overline{f}_p (\bar x, \bar x) - \partial_2(-\overline{f}_p) (\bar x, \bar x) \right\} + \sum_{p=1}^{m_2} \bar\mu^{\,m_1 + p} \, \nabla \phi_p(\bar x),
    \end{aligned}
    \end{equation}
    where $\bar\mu^{m_1 + p} \in \mathcal{N}_{(-\infty,0]} (\phi_p(\bar x))$ for $p=1,\cdots,m_2$ {and $\overline f_p$ is defined in \eqref{eq:icc lifted} for $p=1, \cdots, m_1$}. By assumptions, both $\Lambda_p(\bar x)$ and $Y_p(\bar x)$ are nonempty, bounded, and
    \[
        \Lambda_p(\bar x) \times Y_p(\bar x) = \left\{ (\bar y^p, \bar\mu^p) \,\middle|\,
            c^p + C^p \bar x + Q^p \, \bar y^p + (B^p)^\top \bar\mu^p = 0,\;
		 0 \leq b^p - A^p \bar x - B^p \bar y^p \; \perp \; \bar\mu^p \geq 0
        \right\}.
    \]
     It then follows from Danskin's Theorem \cite[Theorem 2.1]{clarke1975generalized} that
    \[
    \begin{aligned}
        \partial_1 \overline{f}_p (\bar x, \bar x) &= \Conv \left\{(A^p)^\top \bar\mu^p \,\middle|\, \bar\mu^p \in \Lambda_{p}(\bar x) \right\} = \left\{(A^p)^\top \bar\mu^p \,\middle|\, \bar\mu^p \in \Lambda_p(\bar x) \right\},\\[0.05in]
        \partial_2 (-\overline{f}_p) (\bar x, \bar x) &= \Conv \left\{-(C^p)^\top \bar y^p \,\middle|\, \bar y^p \in Y_p(\bar x) \right\} = \left\{-(C^p)^\top \bar y^p \,\middle|\, \bar y^p \in Y_p(\bar x) \right\}.
    \end{aligned}
    \]
    Combining these expressions with \eqref{eq:KKT}, we obtain
    \[
	\left\{
	\begin{array}{ll}
	  0 = \nabla \theta(\bar x) + \displaystyle\frac{1}{m_1} \sum\limits_{p=1}^{m_1} \left[ (C^p)^\top \bar y^p + (A^p)^\top \bar\mu^p \right] + \sum_{p=1}^{m_2} \bar\mu^{m_1 + p} \, \nabla\phi_p(\bar x),\\
	  c^{\,p} + C^p \bar x + Q^p \, \bar y^p + (B^p)^\top \bar\mu^p = 0,\;\,
        0 \leq b^p - A^p \bar x - B^p \bar y^p \; \perp \; \bar\mu^p \geq 0, \quad p = 1, \cdots, m_1,\\[0.05in]
        0 \leq \phi_p(\bar x) \; \perp \; \bar\mu^{\,m_1 + p} \geq 0, \quad p=1,\cdots,m_2,
	\end{array}
	\right.
    \]
    which are the Karush-Kuhn-Tucker (KKT) conditions for the deterministic equivalent of \eqref{eq:2SP_1st}.

\section{A computational algorithm.}
\label{sec4:algorithm}

In this section, we consider a double-loop algorithm for solving problem \eqref{eq:cvx_composite}. The inner loop finds an approximate stationary point of the perturbed composite optimization problem
\begin{equation}
\label{eq:cvx_composite_perturb}
    \Min_{x\in \R^n} \;\, \sum_{p=1}^{m} \left[F^k_p(x) \triangleq \varphi_p(f^k_p(x))\right]
\end{equation}
by solving a sequence of convex subproblems, while the outer loop drives $k \to +\infty$. It is important to note the potential infeasibility in \eqref{eq:cvx_composite_perturb} because $[F^k_p = \varphi_p\circ f^k_p] \Epiconv F_p$ in Assumption 1(c), together with $\dom(\varphi_p \circ f_p) \neq \emptyset$, does not guarantee $\dom (\varphi_p \circ f^k_p) \neq \emptyset$ for all ${k \in \N}$. {This can be seen from the example of $\varphi(t) = \delta_{(-\infty, 0]}(t)$,  $f(x) = \max\{ x, 0\} - 1/10$ and $f^k(x) = \max\{ x, 0\} + 1/k - 1/10$. Obviously $\dom(\varphi \circ f) = (-\infty, 1/10]$ and $\varphi \circ f^k \Epiconv \varphi \circ f$ by \cite[Theorem 2.4(d)]{royset2022consistent}, but we have $\dom(\varphi \circ f^k) = \emptyset$ for $k=1,\cdots,9$. Even though $\dom(\varphi_p \circ f^k_p) \neq \emptyset$ for all ${k \in \N}$ and each $p$, this does not imply the feasibility of convex subproblems used in the inner loop to approximate \eqref{eq:cvx_composite_perturb}.}

For simplicity of the analysis, we assume that  in problem \eqref{eq:cvx_composite}, $\varphi_p$ is real-valued for $p=1, \cdots, m_1$, and $\varphi_p = \delta_{(-\infty,0]}$ for $p=m_1 + 1, \cdots, m$. Namely, the problem takes the following form:
\begin{equation}
\label{eq:cvx_composite_constraints}
    \Min_{x\in \R^n}\;\, \sum_{p=1}^{m_1} \left[F_p(x) = \varphi_p \big(f_p(x)\big)\right]\quad
    \mbox{subject to} \;\, f_p(x) \leq 0, \;\, p=m_1 + 1, \cdots, m.
\tag{CP$_1$}
\end{equation}
For $p=1,\cdots,m_1$, the convexity of each real-valued function $\varphi_p$ implies its continuity by \cite[corollary 10.1.1]{rockafellar1970convex}. Consequently, the composite function $F^k_p = \varphi_p \circ f^k_p$ is also continuous for $p=1,\cdots,m_1$ and $k \in \N$ due to the continuity of each approximating function $f^k_p$. It is important to note that model \eqref{eq:cvx_composite_constraints} still covers discontinuous objective functions since each $f_p$ can be discontinuous, even though the approximating sequence $\{f^k_p\}_{k \in \N}$ only consists of locally Lipschitz continuous functions.

\subsection{Assumptions {and examples}.}
\label{subsec:Assump}

Firstly, we make an assumption to address the feasibility issue outlined at the start of this section. 
For all $k \in \N$ and $p=m_1 + 1, \cdots, m$, define
\[
    {\alpha^{k}_p \triangleq \sup_{x \in X^k} \left[f^{k+1}_p(x) - f^k_p(x)\right]_+}
    \text{ with }X^k \triangleq \left\{ x \in \R^n \,\middle|\, f^k_p(x) \leq 0,\, p=m_1 + 1, \cdots, m\right\}.
\]
Based on these auxiliary sequences, we need an initial point $x^0$ that is strictly feasible to the constraints $f^0_p(x) \leq 0$ for each $p = m_1+1,\cdots,m$.

\gap
\begin{center}
\fbox{\parbox{0.98\textwidth}{
\noindent{\bf Assumption 2} ({\bf strict feasibility}) There exist $x^0$ and nonnegative sequences $\big\{\widehat{\alpha^k_p}\big\}_{k \in \N}$ for $p=m_1+1, \cdots, m$, such that $\alpha^k_p \leq \widehat{\alpha^k_p}$ for all ${k \in \N}$ and
\vskip -0.1in
\[
    \sum_{{k^\prime}=0}^{+\infty} \widehat{\alpha^{{k^\prime}}_p} < +\infty, \quad f^0_p(x^0) \leq -\sum_{{k^\prime}=0}^{+\infty} \widehat{\alpha^{{k^\prime}}_p}, \qquad\, p=m_1+1, \cdots, m.
\]
\vskip -0.1in
}}
\end{center}

To streamline our notation and analysis, we extend the definitions of $\alpha^k_p$ and introduce $\widehat{\alpha^k_p}$ for $p=1, \cdots, m_1$ by setting $\alpha^k_p = \widehat{\alpha^k_p} = 0$ for all $k \in \N$ and $p=1, \cdots, m_1$.
{
Since the quantity $\alpha^k_p$ depends on the sequence $\{f^k_p\}_{k \in \N}$, Assumption 2 poses a condition for this approximating sequence. Consider a fixed index $p \in \{m_1+1, \cdots, m\}$. One can use the following way to construct $\{\alpha^k_p\}_{k\in N}$. Suppose that there exist a  function $\widetilde{f_p}: \R^n \times [0,1] \to \R$ and a nonnegative sequence $\gamma_k \downarrow 0$ such that
\[
    \widetilde{f_p}(x, \gamma_k) = f^k_p(x)
    \quad\text{and}\quad
    \widetilde{f_p}(x, 0) = f_p(x).
\]
Additionally, assume that for any fixed $x$, the function $\widetilde{f_p}(x, \cdot)$ is continuous on $[0,1]$ and differentiable on $(0,1)$, and there exists a constant $C_p$ such that $\big|\nabla_{\gamma} \widetilde{f_p}(x,\gamma)\big| \leq C_p$ for any $x$ and $\gamma \in (0,1)$. For any fixed $x$, by the mean value theorem, there exists a point $\bar\gamma_k \in (\gamma_{k+1}, \gamma_k)$  such that $f^{k+1}_p(x) - f^k_p(x) = (\gamma_{k+1} - \gamma_k) \nabla_{\gamma} \widetilde{f_p}(x, \bar\gamma_k)$. 
Thus,
\[
    \sum^{+\infty}_{k^\prime=0} \alpha^{k^\prime}_p
    \leq \sum^{+\infty}_{k^\prime=0} (\gamma_{k^\prime} - \gamma_{k^\prime+1}) \sup_{x \in \R^n} \left|\nabla_{\gamma} \widetilde{f_p}(x, \bar\gamma_{k^\prime}) \right|
    \leq \sum^{+\infty}_{k^\prime=0} \left[\,\widehat{\alpha^{k^\prime}_p} \triangleq C_p (\gamma_{k^\prime} - \gamma_{k^\prime+1})\right]
    = C_p \gamma_0 < +\infty.
\]
Two more assumptions on the approximating sequences $\{f_p^k\}_{k \in \N}$ are needed.

\begin{center}
\fbox{\parbox{0.98\textwidth}{
\noindent{\bf Assumption 3} ({\bf smoothness of  $g_p^k$ or $h_p^k$})
For each ${k \in \N}$, there exists $\ell_k>0$ such that
\[
	\min\Big\{\, \Hausd\big(\partial g^k_p(x), \partial g^k_p(x^\prime)\big),\,
	\Hausd\big(\partial h^k_p(x), \partial h^k_p(x^\prime)\big) \,\Big\} \leq \ell_k \|x^\prime - x\|
	\quad\forall\, x, x^\prime \in \R^n, \; p = 1, \cdots, m.
\]

\vskip -0.1in
\noindent{\bf Assumption 4} ({\bf level-boundedness}) For each ${k \in \N}$, the function ${H^k \triangleq} \sum_{p=1}^{m} F^k_p$ is level-bounded, i.e., for any $r \in \R$, the set
\[
    {\left\{ x \in \R^n \,\middle|\, \sum^{m_1}_{p=1} \varphi_p(f^k_p(x)) + \sum^{m}_{p=m_1+1} \delta_{(-\infty,0]}(f^k_p(x)) \leq r \right\} =}
    \left\{ x \in \R^n \,\middle|\, \sum\limits_{p=1}^{m_1} \varphi_p\left(f^k_p (x)\right) \leq r\right\}
    \cap X^k
\]
is bounded.
}}
\end{center}

Assumption 3 imposes conditions on the Lipschitz continuity of the subdifferential mapping $\partial g^k_p$ or $\partial h^k_p$, which will be used to determine the termination rule of the inner loop.
A straightforward sufficient condition for this assumption is that, for each $p$ and $k$, {at least one of the functions $g^k_p$ and $h^k_p$} is $\ell_k$-smooth, i.e., $\|\nabla g^k_p(x) - \nabla g^k_p(x^\prime)\| \leq \ell_k \|x - x^\prime\|$ or $\|\nabla h^k_p(x) - \nabla h^k_p(x^\prime)\| \leq \ell_k \|x - x^\prime\|$ for any $x, x^\prime \in \R^n$. {We also remark that Assumption 3 can hold even though both $g^k_p$ and $h^k_p$ are nondifferentiable. This can be seen from the following univariate example: $g^k_p(x) = |x|$ and $h^k_p(x) = |x - 1|$ for any $x \in \R$. It is not difficult to verify that Assumption 3 holds for $\ell_k = 2$.}
Assumption 4 is a standard condition to ensure the boundedness of the generated sequences for each ${k \in \N}$. 

\gap
In addition, we need a technical assumption to ensure the boundedness of the multiplier sequences in our algorithm.

\begin{center}
\fbox{\parbox{0.98\textwidth}{
\noindent{\bf Assumption 5} ({\bf an asymptotic constraint qualification}) For any {$\bar{x} \in {\bigcap_{p=1}^{m}} \dom F_p$}, if there exists $\{y_p\}_{p=1}^{m}$ satisfying $0 = \sum_{p=1}^{m} y_p \, v_p$ where for each $p$
(with the definition of $T_p(\bar{x})$ in \eqref{def:T_p}),
\vskip -0.1in
\begin{equation}\label{eq:ACQ}
    (y_p,\, v_p) \in \Big( \bigcup \big\{\mathcal{N}_{\dom\varphi_p}(t_p) \mid t_p \in T_p(\bar x) \big\} \times \Conv\partial_A f_p(\bar x) \Big)
    \cup
    \Big(\R \times \left[\,\partial^\infty_A f_p(\bar x) \backslash \{0\}\,\right]\Big),
\end{equation}
then we must have $y_{1} = \cdots = y_{m} = 0$.
}}
\end{center}

The normal cone $\mathcal{N}_{\dom \varphi_p}(t_p)$ in \eqref{eq:ACQ} reduces to $\{0\}$ for $p = 1, \cdots, m_1$ and $\mathcal{N}_{(-\infty, 0]} (t_p)$ for $p = m_1 + 1, \cdots, m$.
According to the definitions of $\partial_A f_p(\bar x)$ and $\partial^\infty_A f_p(\bar x)$, Assumption 5 depends on the approximating sequences $\{f^k_p\}_{k \in \N}$ for $p = 1,\cdots,m$. 
It holds trivially if each $\varphi_p$ is real-valued and $\partial^\infty_{A} f_p(\bar x) = \{0\}$. {By Theorem \ref{thm:eADC_subdiff}(b), the condition $\partial^\infty_A f_p(\bar x) = \{0\}$ holds when the ADC decompositions are constructed using the Moreau envelope, provided that $f_p$ is locally Lipschitz continuous and bounded from below. However, in general, Assumption 5 is not easy to verify.} For Example 3.1, the assumption translates into
\[
    \left[\sum_{p=1}^{m_2} \lambda_p \nabla\phi_p(\bar x) = 0,\; \lambda_p \in \mathcal{N}_{(-\infty,0]} (\phi_p(\bar x)),\;\; p=1,\cdots,m_2\right]
    \quad\Longrightarrow\quad \lambda_1=\cdots=\lambda_{m_2}=0.
\]
This is equivalent to the Mangasarian-Fromovitz constraint qualiﬁcation (MFCQ) for problem \eqref{eq:2SP_1st} by \cite[Example 6.40]{rockafellar2009variational}; see also \cite{rockafellar1993lagrange}. 

Furthermore, if each $f_p$ is c-ADC associated with $\{f^k_p = g^k_p - h^k_p\}_{k \in \N}$ such that $\Conv\partial_{A} f_p(\bar x) = \partial_C f_p(\bar x)$, and $\partial^\infty_{A} f_p(\bar x) = \{0\}$, Assumption 5 states that
\[
    \left[\, 0 \in \sum_{p=1}^{m} y_p \, \partial_C f_p(\bar x),\quad y_p \in \mathcal{N}_{\dom\varphi_p}(f_p(\bar x)), \;\; p=1, \cdots, m \,\right]
    \quad\Longrightarrow\quad
    y_1 = \cdots = y_{m} = 0.
\]
This condition aligns with the constraint qualification for the composite optimization problem in \cite[Proposition 2.1]{royset2022consistent}, and is stronger than the condition in the nonsmooth Lagrange multiplier rule \cite[Exercise 10.52]{rockafellar2009variational}. 
Finally, Assumption 5 implies the constraint qualifications \eqref{eq:CQ1}-\eqref{eq:CQ3} in Theorem \ref{thm:necessary_cond}. We formally present this conclusion in the following proposition. The proof of Proposition \ref{prop:assumption5} is given in {Appendix B}.


\begin{proposition}[consequences of Assumption 5]
\label{prop:assumption5}
    Suppose that Assumptions 1 and 5 hold, and $f^k_p \Epiconv f_p$ for each $p$. If $\sup\varphi_p=+\infty$ for $p\in I_1$, and $f_p$ is locally Lipschitz continuous for $p \in I_2$ (with the definitions of $I_1$ and $I_2$ in \eqref{defn: index for varphi_p}), then conditions \eqref{eq:CQ1}, \eqref{eq:CQ2}, and \eqref{eq:CQ3} hold {at any feasible point $\bar x$ of \eqref{eq:cvx_composite_constraints}. Consequently, any local solution of \eqref{eq:cvx_composite_constraints} is a (weakly) A-stationary point of \eqref{eq:cvx_composite_constraints}.}
\end{proposition}

{In the following, we use two examples to further illustrate Assumption 3 and the computation of $\{\widehat{\alpha^k_p}\}_{k \in \N}$ in Assumption 2.

\gap

\noindent\textbf{Example 4.1} (icc constraints). 
Let $f_p$ be {real-valued} and icc associated with $\overline f_p$, where $\overline f_p(\cdot,x)$ is Lipschitz continuous with modulus $L$ for any $x$. For the sequence $\{f^k_p\}_{k \in \N}$ in Example 2.1, it follows from $g^k_p(x) = \|x\|^2/(2\gamma_k)$ that Assumption 3 holds for $\ell_k = 1/\gamma_k$. To construct the quantities $\widehat{\alpha^k_p}$ in Assumption 2, we notice that
\begin{equation}\label{eq:alpha_estmate}
    \alpha^{k}_p
    \leq \sup_{x \in \R^n} \left[f^{k+1}_p(x) - f^k_p(x)\right]_+
    \,\leq\, \displaystyle \sup_{x \in \R^n} \left[f_p(x) - f^{k}_p(x)\right]_+
    \,\leq\, \frac{\gamma_k\, L^2}{2}
    \,\triangleq\, \widehat{\alpha^k_p} \qquad\forall\, {k \in \N},
\end{equation}
where the second inequality is due to $f^{k+1}_p(x) \leq f_p(x)$ for any $x$, and the last one uses the bound between the partial Moreau envelope and the original function \cite[Lemma 3]{li2022decomposition}. Thus, the sequence $\big\{\widehat{\alpha^k_p}\big\}_{k \in \N}$ satisfies $\sum_{{k^\prime}=0}^{+\infty} \widehat{\alpha^{{k^\prime}}_p} < +\infty$ if $\{\gamma_k\}$ is summable.

Alternatively, we can construct the quantities $\widehat{\alpha^k_p}$ as follows. Let the partial Moreau envelope in \eqref{eq:PME} be the function $\widetilde{f_p}$ jointly defined for $(x, \gamma) \in \R^n \times (0, 1]$, and $\widetilde{f_p}(x, 0) = f_p(x)$ for any $x$. 
{We claim that $\widetilde{f_p}(x, \cdot)$ is continuous on $[0,1]$ and differentiable on $(0,1)$ for any fixed $x$. Continuity in $\gamma$ can be simply checked by {a standard argument \cite[Theorem 1.17(c)]{rockafellar2009variational}}, noting that the optimal value is achieved at a unique point as the function $\overline f_p(\cdot, x) + \|\cdot-x\|^2/(2\gamma)$ is strongly convex for any fixed $x$. Differentiability follows from the Danskin's Theorem \cite[Theorem 2.1]{clarke1975generalized} that $\nabla_\gamma \widetilde{f_p} (x, \gamma) = -\|z-x\|^2 / (2\gamma^2)$ with $z$ satisfying $(x-z) / \gamma \in \partial_1 \bar f(z, x)$ for any $(x, \gamma) \in \R^n \times (0,1]$.
It then follows from the Lipschitz continuity of $\overline f_p(\cdot, x)$ that $\big|\nabla_{\gamma} \widetilde{f_p}(x,\gamma)\big| \leq L^2/2 \triangleq C_p$ for any $(x, \gamma) \in \R^n \times (0,1]$.} 
Therefore, $\alpha^k_p \leq C_p (\gamma_k - \gamma_{k+1}) \triangleq \widehat{\alpha^k_p}$ and $\sum^{+\infty}_{k^\prime=0} \widehat{\alpha^{k^\prime}_p} = C_p \gamma_0 < +\infty$ for any sequence $\{\widetilde{f_p}(\cdot, \gamma_k)\}_{k \in \N}$ defined by the partial Moreau envelope with $\gamma_k \downarrow 0$. 

\gap\gap

\noindent\textbf{Example 4.2} (VaR constraints for log-normal distributions). 
Consider $f_p(x) = \mbox{VaR}_{\alpha}[\, c(x, Z)]$ with  $c(x, Z) = \exp(\,x^\top Z)$ for some random vector $ Z \sim \mbox{Normal}(\mu, \Sigma)$, where $\Sigma$ is a positive definite covariance matrix. We have $c(x,Z) \sim \mbox{Lognormal}\big(x^\top \mu, \sqrt{x^\top \Sigma x}\big)$. The variable $x$ is restricted to a compact set $X \subset \R^n$. Denote the $\alpha$-quantile of the standard normal distribution by $q_\alpha$ and the cumulative distribution function of the standard normal distribution by $\Phi(\cdot)$. By direct calculation (cf. \cite[Section 3.2]{norton2021calculating}), we have
\[
    \mbox{VaR}_\alpha[\,c(x,Z)] = \exp\left({x^\top \mu + \sqrt{x^\top \Sigma \,x}\, q_{\alpha}}\right), \; 
    \mbox{CVaR}_\alpha [\,c(x,Z)]
    = \exp\left(x^\top \mu + \frac{x^\top \Sigma x}{2}\right)  \frac{\Phi\big(\sqrt{x^\top \Sigma x} - q_\alpha\big)}{1-\alpha}.
\]
Hence, $f_p(x) = \mbox{VaR}_{\alpha}[\, c(x, Z)]$ is neither convex nor concave if $q_{\alpha} < 0$.
For the sequence $\{f^k_p\}_{k \in \N}$ in Example 2.2, we can derive that
\[
    h^k_p(x) = k(1-\alpha) \, \mbox{CVaR}_\alpha [\,c(x,Z)]
    = k \, \exp\left(x^\top \mu + x^\top \Sigma x/2\right) \, \Phi\big(\sqrt{x^\top \Sigma x} - q_\alpha\big).
\]
Since $\Sigma$ is positive definite, it is easy to see that $h^k_p$ is twice continuously differentiable. Consequently, $h^k_p$ is $\ell_k$-smooth relative to the compact set $X$ for some $\ell_k$, and Assumption 3 holds (relative to $X$). 
Next, we define $\widetilde{f_p}(x, \gamma) = \frac{1}{\gamma}\int^{\alpha}_{\alpha - \gamma} \mbox{VaR}_t [\, c(x,Z)] \,\mbox{d}t$ for any $(x, \gamma) \in \R^n \times (0, \frac{\alpha}{2}]$ and $\widetilde{f_p}(x, 0) = f_p(x)$ for any $x$.
Obviously, $\widetilde{f_p}(x, \cdot)$ is continuous on $[0, \frac{\alpha}{2}]$ and differentiable on $(0, \frac{\alpha}{2})$ for any fixed $x$. By {using the Leibniz rule for differentiating the parametric integral}, for $\gamma \in (0, \frac{\alpha}{2})$, we have
\[
\begin{array}{rl}
    \left|\nabla_\gamma \widetilde{f_p}(x, \gamma)\right|
    =& \displaystyle\frac{1}{\gamma^2} \int^{\alpha}_{\alpha - \gamma} (\mbox{VaR}_t[\,c(x,Z)] - \mbox{VaR}_{\alpha-\gamma}[\,c(x,Z)]) \, \mbox{d}t \\[0.13in]
    \leq& \displaystyle\frac{1}{\gamma}\left(\mbox{VaR}_{\alpha}[\,c(x,Z)] - \mbox{VaR}_{\alpha-\gamma}[\,c(x,Z)]\right) \\[0.1in]
    = & \displaystyle \exp({x^\top \mu}) \;
    \frac{\exp\big({\sqrt{x^\top\Sigma x}\, q_\alpha}\big) - \exp\big({\sqrt{x^\top\Sigma x}\, q_{\alpha-\gamma}}\big)}{\gamma} \\[0.13in]
    =& \exp(x^\top \mu) \cdot \left[\exp\big(\sqrt{x^\top \Sigma x}\, q_{\alpha^\prime}\big) \sqrt{x^\top \Sigma x} \; \nabla_\alpha q_{\alpha^\prime}\right]
\end{array}
\]
for some $\alpha^\prime \in (\alpha - \gamma , \alpha)$ by the mean-value theorem. By using the fact $\nabla_\alpha q_\alpha = \sqrt{2\pi} \, \exp(q_\alpha^2 / 2)$, the monotonicity $q_{\alpha/2} < q_{\alpha^\prime} < q_{\alpha}$, and the compactness of $X$, we further have
\[
    {\sup_{x \in X}} \left|\nabla_\gamma \widetilde{f_p}(x, \gamma)\right|
    \leq \exp\left(\frac{\max\{q_\alpha^2, q_{\,\alpha/2}^2\}}{2}\right) \sup_{x \in X} \left\{\exp\big({x^\top \mu + \sqrt{x^\top \Sigma x}\, q_{\alpha}} \big)\sqrt{(2\pi) \, x^\top \Sigma x} \right\}
    \,\triangleq\, C_p < +\infty.
\]
Therefore, $\alpha^k_p \leq C_p (\gamma_k - \gamma_{k+1}) \triangleq \widehat{\alpha^k_p}$ and $\sum^\infty_{k^\prime=0} \widehat{\alpha^k_p} = C_p \gamma_0 < +\infty$ for any sequence $\{\widetilde{f_p}(\cdot, \gamma_k)\}_{k \in \N}$ with $\gamma_k \downarrow 0$.
}

\subsection{The algorithmic framework and convergence analysis.}

We now formalize the algorithm for solving \eqref{eq:cvx_composite_constraints}. {For $p=m_1+1,\cdots,m$, recall the nonnegative sequences $\big\{\widehat{\alpha^k_p}\big\}_{k \in \N}$ introduced in Assumption 2, and observe that $\sum^{+\infty}_{k^\prime=k} \widehat{\alpha^{k^\prime}_p} \rightarrow 0$ as $k \to +\infty$. For consistency of our notation, we also set $\widehat{\alpha^k_p} \equiv 0$ for all ${k \in \N}$ and $p=1,\cdots,m_1$.} At the $k$-th outer iteration and for $p=1, \cdots, m$, consider the upper and lower approximation of $f^k_p$ at a point $y$ by taking some $a^k_p \in \partial h^k_p(y)$, $b^k_p \in \partial g^k_p(y)$ and incorporating sequences $\big\{\widehat{\alpha^k_p}\big\}_{k \in \N}$:
\vspace{-0.12in}
\begin{equation}\label{eq:majorization}
\begin{array}{rll}
    f^{k,\text{upper}}_p(x; y) \triangleq& g^k_p(x) - h^k_p(y) - (a^k_p)^\top (x - y) + \sum\limits^{+\infty}_{k^\prime=k} \widehat{\alpha^{k^\prime}_p},\\[0,1in]
    f^{k,\text{lower}}_p(x; y) \triangleq& g^k_p(y) + (b^k_p)^\top (x - y) - h^k_p(x).
\end{array}
\end{equation}
{
Observe that, for fixed $y$, the upper approximation $f^{k,\text{upper}}_p(\cdot\,; y)$ is convex while the lower approximation $f^{k,\text{lower}}_p(\cdot\,; y)$ is concave. For $p=1, \cdots, m_1$, consider the following function
\begin{equation}\label{eq:majorization idea}
    \widehat{F^k_p}(x;y) \triangleq \varphi_p^{\uparrow}\left( f^{k,\text{upper}}_p(x; y) \right)
    + \varphi_p^{\downarrow}\left( f^{k,\text{lower}}_p(x; y) \right),
\end{equation}
{which is a convex majorization of $F^k_p$ at a point $y$ by the fact that $\varphi^\uparrow_p$ is nondecreasing and $\varphi^\downarrow_p$ is nonincreasing.}
For $p=m_1 + 1, \cdots, m$, consider the convex constraint $f^{k,\text{upper}}_p(x;y) \leq 0$ as an approximation for $f^k_p(x) \leq 0$.

We summarize the properties of all the surrogate functions as follows. Note that \eqref{eq:f_upper} and \eqref{eq:f_lower} hold for $p=1,\cdots,m$, while \eqref{eq:F_hat} holds only for $p=1,\cdots,m_1$.
\vspace{-0.1in}
\begin{subequations}
\begin{align}
        f^{k, \text{upper}}_p(x; y) \geq f^k_p(x) + \sum^\infty_{k^\prime=k} \widehat{\alpha^{k^\prime}_p} &\geq f^k_p(x),
        \qquad
        f^{k, \text{upper}}_p(x; x) = f^k_p(x) + \sum^\infty_{k^\prime=k} \widehat{\alpha^{k^\prime}_p}, \label{eq:f_upper}\\
        f^{k, \text{lower}}_p(x; y) &\leq f^k_p(x), 
        \qquad
        f^{k, \text{lower}}_p(x; x) = f^k_p(x), \label{eq:f_lower}\\[0.08in]
        \widehat{F^k_p}(x; y) &\geq F^k_p(x),
        \qquad
        \widehat{F^k_p} (x; x) = F^k_p(x). \label{eq:F_hat}
\end{align}
\end{subequations} 
}

The proposed method for solving problem \eqref{eq:cvx_composite_constraints} is outlined in Algorithm \ref{alg:double-loop}. {The inner loop of the algorithm (indexed by $i$) is terminated when the following conditions are satisfied:
\begin{equation}\label{eq:inner_stopping}
    \left\{\begin{array}{rll}
        f^{k,\text{upper}}_p (x^{k,i+1}; x^{k,i}) &\leq f^k_p(x^{k,i+1}) + \sum\limits^{+\infty}_{k^\prime=k} \widehat{\alpha^{k^\prime}_p} + \epsilon_k, &\quad p=1, \cdots, m,\\[0.05in]
        f^{k,\text{lower}}_p (x^{k,i+1}; x^{k,i}) &\geq f^k_p(x^{k,i+1}) - \epsilon_k, &\quad p \in I_2,\\[0.05in]
        \| x^{k,i+1} - x^{k,i} \| &\leq \delta_k/(\lambda + \ell_k),
    \end{array}\right.
\end{equation}
}


\begin{algorithm}[h]
\caption{The prox-ADC method for solving \eqref{eq:cvx_composite_constraints}}
\label{alg:double-loop}
\textbf{Input}: Given $x^0$ and $\big\{\widehat{\alpha^k_p}\big\}_{k \in \N}$ satisfying Assumption 2. Let $\{\ell_k\}$ be a sequence satisfying Assumption 3. Choose $\lambda > 0$, a positive sequence $(\epsilon_k, \delta_k) \downarrow 0$ such that $\delta_k / (\lambda + \ell_k) \downarrow 0$. Set $k=0$.

\begin{algorithmic}[1]
\While{a prescribed stopping criterion is {not met}}
    \State $x^{k,0} = x^k$
    \For{$i=0, 1, \cdots$}
        \State Take $a^{k,i}_p \in \partial\, g^k_p(x^{k, i})$ for $p=1, \cdots, m$ and $b^{k,i}_p \in \partial h^k_p(x^{k, i})$ for $p=1, \cdots, m_1$
        \State Solve the strongly convex subproblem:
        \begin{equation}\label{eq:subproblem}
            \hskip -0.15in
            x^{k,i+1} =
            \left[\begin{array}{cl}
            \displaystyle\operatornamewithlimits{argmin}_{x \in \R^n}&\;\; \sum\limits_{p=1}^{m_1} \widehat{F^k_p}(x;x^{k,i}) + \frac{\lambda}{2}\|x - x^{k,i}\|^2 \\[0.15in]
            \text{subject to}&\;\; f^{k,\text{upper}}_p(x; x^{k,i}) \leq 0,\, p=m_1 + 1, \cdots, m
            \end{array}\right]
        \end{equation}
        \If{the conditions \eqref{eq:inner_stopping} hold {for $\lambda, \ell_k, \epsilon_k, \delta_k$, and $\sum^{+\infty}_{k^\prime=k}\widehat{\alpha^{k^\prime}_p}$}}
            \State Break the for-loop
        \Else
            \State $i \leftarrow i+1$
        \EndIf
    \EndFor
    \State $x^{k+1} = x^{k,i}$
    \State $k \leftarrow k + 1$
\EndWhile
\end{algorithmic}
\end{algorithm}

\begin{figure}[h]
	\subfigure[\, \footnotesize $F_1 = \varphi_1 \circ f_1$ for a convex $\varphi_1$ and a smooth $f_1$.]{
	\begin{minipage}[t]{0.45\linewidth}
	\centering
	\includegraphics[scale=0.16]{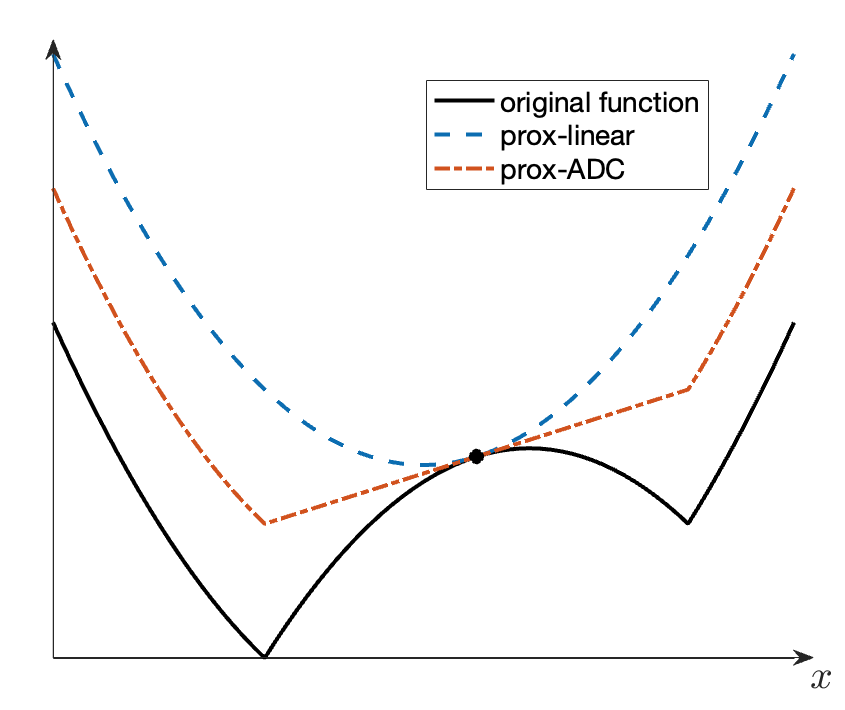}
	\end{minipage}}
	\subfigure[\, \footnotesize $F_1 = \varphi_1 \circ f_1$ for a convex nondecreasing $\varphi_1$ and a lsc $f_1$.]{
    \begin{minipage}[t]{0.52\linewidth}
	\centering
	\includegraphics[scale=0.17]{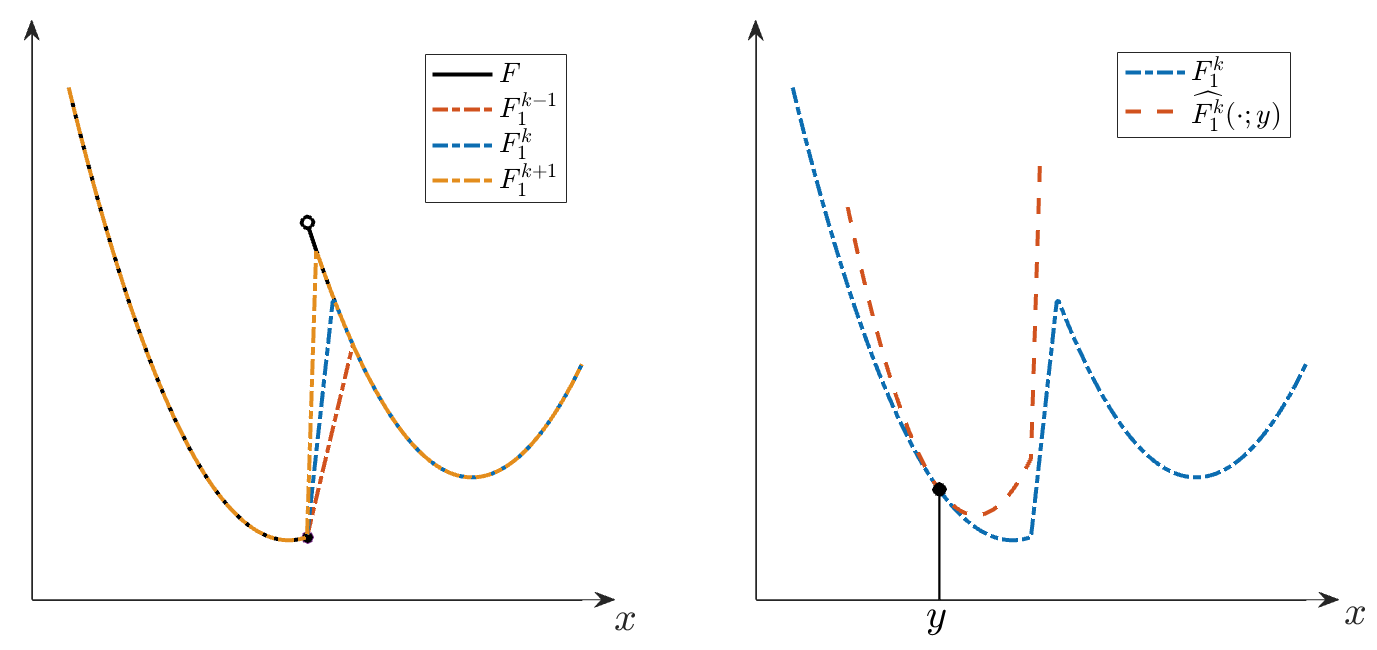}
	\end{minipage}}
	\centering
	\caption{\footnotesize Illustrations of the prox-ADC method. (a): a comparison of the prox-ADC and the prox-linear method for minimizing an amenable function. (b): asymptotic approximations of a discontinuous composite function $F_1 =\varphi_1 \circ f_1$ that are constructed by an epi-convergent sequence $\{F^k_1 = \varphi_1 \circ f^k_1\}$, and a convex majorization {$\widehat{F^k_1}(\cdot\,; y)$} for $F^k_1$.}
\label{fig:illustrations}
\end{figure}

In contrast to the prox-linear algorithm that is designed to minimize amenable functions and adopts complete linearization of the inner maps, the prox-ADC method retains more curvature information inherent in these maps (see Figure \ref{fig:illustrations}). 
We emphasize that the prox-ADC method differs from \cite[Algorithm 7.1.2]{cui2021modern} that is designed for solving a problem with a convex composite DC objective and DC constraints. Central to the prox-ADC method is the double-loop structure, where, in contrast to \cite[Algorithm 7.1.2]{cui2021modern}, the DC sequence $f^k_p$ is dynamically updated in the outer loop rather than remaining the same. This adaptation necessitates specialized termination criteria \eqref{eq:inner_stopping} and the incorporation of $\widehat{\alpha^k_p}$ to maintain feasibility with each update of $f^k_p$. 
In the following, we demonstrate the well-definedness of the prox-ADC method. Specifically, we establish that for each iteration $k$, the criteria detailed in \eqref{eq:inner_stopping} are attainable {in finitely many steps}.
\begin{theorem}[convergence of the inner loop]
\label{thm:converge1}
    Suppose that Assumptions 1-4 hold. 
    Then the following statements hold.\\
    (a) Problem \eqref{eq:subproblem} is feasible for any $k, i \in \N$.\\
    (b) The stopping rule of the inner loop is achievable in {finitely many steps}, i.e., the smallest integer $i$ satisfying conditions \eqref{eq:inner_stopping}, denoted by $i_k$, is finite for any ${k \in \N}$.
\end{theorem}
\begin{proof}
We prove (a) and (b) by induction. For $k=0$, notice from Assumption 2 {and \eqref{eq:f_upper}} that $f^{0,\text{upper}}_p(x^0; x^{0}) = f^0_p(x^0) + \sum^{+\infty}_{{k^\prime}=0} \widehat{\alpha^{{k^\prime}}_p} \leq 0$ for $p=m_1 + 1, \cdots, m$. Thus, problem \eqref{eq:subproblem} is feasible for $k=i=0$. Assume that \eqref{eq:subproblem} is feasible for $k=0$ and some $i = \bar i ~({\in \N})$. Consequently, $x^{0, \bar i+1}$ is well-defined and for $p=m_1 + 1, \cdots, m$,
\[
    f^{0,\text{upper}}_p(x^{0, \bar i+1}; x^{{0}, \bar i+1}) 
    \overset{\eqref{eq:f_upper}}{=} f^0_p(x^{0, \bar i+1}) + \sum\limits^{+\infty}_{{k^\prime}=0}\widehat{\alpha^{{k^\prime}}_p} \;
    \overset{\eqref{eq:f_upper}}{\leq}\; f^{0,\text{upper}}_p(x^{0, \bar i+1}; x^{0, \bar i})
    \leq 0,
\]
which yields the feasibility of \eqref{eq:subproblem} for $k=0$, $i = \bar i+1$. Hence, by induction, problem \eqref{eq:subproblem} is feasible for $k = 0$ and any $i \in \N$. 
To proceed, recall the function $H^k$ defined in Assumption 4. From the update of $x^{0,i+1}$, we have
\begin{equation}\label{eq:descent}
    H^0 (x^{0,i+1}) = \sum_{p=1}^{m_1} F^0_p(x^{0, i+1})\,
    {\overset{\eqref{eq:F_hat}}{\leq}}\, \sum_{p=1}^{m_1} \widehat{F^0_p}(x^{0, i+1}; x^{0, i})
    \leq H^0(x^{0,i}) - \frac{\lambda}{2}\|x^{0,i+1} - x^{0,i}\|^2 \quad\forall\, i {\in \N}.
\end{equation}
{The last inequality follows from the definition of $x^{0, i+1}$ and the second relation in \eqref{eq:F_hat} that $\widehat{F^0_p} (x^{0, i}; x^{0, i})= F^0_p (x^{0, i})$ for $p=1,\cdots,m_1$.} Observe that $H^0$ is bounded from below by the continuity of $F^0_p=\varphi_p \circ f^0_p$ for $p=1, \cdots, m_1$ {(see the discussion following  model \eqref{eq:cvx_composite_constraints})} and the level-boundedness of $H^0$. 
{Suppose for contradiction that the stopping rule of the inner loop is not achievable in {finitely many steps}.} 
Then from \eqref{eq:descent}, $\left\{H^0(x^{0,i})\right\}$ converges and $\sum_{i=0}^\infty \|x^{0,i+1} - x^{0,i}\|^2 < +\infty$. The latter further yields $\|x^{0,i+1} - x^{0,i}\| \rightarrow 0$ and thus the last condition in \eqref{eq:inner_stopping} is achievable in {finitely many iterations}. Next, to derive a contradiction, it suffices to prove that the first two conditions in \eqref{eq:inner_stopping} can also be achieved in {finitely many} steps. We only show the first one since the other can be done with similar arguments. {By the level-boundedness of $H^0$, the set $S^0 \triangleq \{ x \mid H^0(x) \leq H^0(x^{0,0})\}$ is compact. Notice that $x^{0,i} \in S^0$ for all $i \in \N$ due to \eqref{eq:descent}.} For $p=1, \cdots, m$, we then have
\[
    0 \leq \;f^{0, \text{upper}}_p (x^{0, i+1}; x^{0, i}) - f^0_p(x^{0,i+1}) - \sum^{+\infty}_{{k^\prime}=0} \widehat{\alpha^{{k^\prime}}_p}
    \;=\; h^0_p(x^{0, i+1}) - h^0_p(x^{0, i}) - (a^{0,i}_p)^\top (x^{0, i+1} - x^{0, i})
    \;\longrightarrow\; 0,
\]
because $h^0_p$ is uniformly continuous on the compact set $S^0$ and $\{a^{0,i}_p\}_{i \in \N} \subset \bigcup \left\{\partial h^0_p(x) \mid x \in S^0\right\}$ is bounded by \cite[Theorem 24.7]{rockafellar1970convex}. Therefore, for a fixed $\epsilon_0 > 0$, there exists some $i_0$ such that $f^{0,\text{upper}}_p (x^{0,i_0+1}; x^{0,i_0}) \leq f^0_p(x^{0,i_0+1}) + \sum^{+\infty}_{{k^\prime}=0} \widehat{\alpha^{{k^\prime}}_p} + \epsilon_0$ holds for $p=1, \cdots, m$. Thus, (a)-(b) hold for $k=0$.

Now assume that (a)-(b) hold for some $k = \bar k ~({\in \N})$ and, hence $i_{\bar k}$ is finite. 
It then follows from $x^{\bar k+1, 0} = {x^{\bar k, i_{\bar k}}} \in X^{\bar k}$ and $f^{\bar k,\text{upper}}_p({x^{\bar k, i_{\bar k}}}; x^{\bar k,i_{\bar k}}) \leq 0$ that for each {$p = m_1 + 1, \cdots, m$},
\[
\begin{array}{rl}
    &f^{\bar k+1, \text{upper}}_p(x^{\bar k+1, 0}; x^{\bar k+1, 0})
    {\overset{\eqref{eq:f_upper}}{=}} f^{\bar k+1}_p(x^{\bar k+1, 0}) + \sum\limits^{+\infty}_{{k^\prime}=\bar k+1} \widehat{\alpha^{{k^\prime}}_p}\\[0.15in]
    \leq& f^{\bar k}_p(x^{\bar k+1, 0}) + \sup\limits_{x \in X^{\bar k}} \left[ f^{\bar k+1}_p(x) - f^{\bar k}_p(x) \right]_+ + \sum\limits^{+\infty}_{{k^\prime}=\bar k+1} \widehat{\alpha^{{k^\prime}}_p} \\[0.18in]
    \leq& f^{\bar k}_p(x^{\bar k+1, 0}) + \sum\limits^{+\infty}_{{k^\prime}=\bar k} \widehat{\alpha^{{k^\prime}}_p}
    \;\;{\overset{\eqref{eq:f_upper}}{=}}\;\; f^{\bar k,\text{upper}}_p(x^{\bar k+1, 0}; x^{\bar k,i_{\bar k}})
    \leq 0.
\end{array}
\]
Thus, problem \eqref{eq:subproblem} is feasible for $k = \bar k + 1$ and any ${i \in \N}$. Building upon this, we can now clearly see the validity of (b) for $k=\bar k+1$, as we have shown similar results earlier in the case of $k=0$. By induction, we complete the proof of (a)-(b). 
\end{proof}

\noindent For any $k \in \N$, define the set of multipliers for problem \eqref{eq:subproblem} as
\[
    Y^k( x^{k+1} ) \triangleq
    \left\{
    \left(\begin{array}{c}
    y^k_{1,1} \\
    y^k_{1,2} \\
    \vdots \\
    y^k_{m,1} \\
    y^k_{m,2}
    \end{array}\right)
    \middle|
    \begin{array}{rl}
    &0 \in \sum\limits_{p=1}^{m} \Big[y^k_{p,1} \, \partial f^{k,\text{upper}}_p(x^{k, i_k + 1};x^{k, i_k}) + y^k_{p,2} \, \partial f^{k,\text{lower}}_p(x^{k, i_k + 1}; x^{k, i_k}) \Big] \\
    &\qquad + \lambda(x^{k, i_k + 1} - x^{k, i_k}), \\[0.05in]
    &y^k_{p,1} \in \partial\varphi^{\uparrow}_p(f_p^{k,\text{upper}}(x^{k, i_k + 1}; x^{k, i_k})), \, p = 1, \cdots, m, \\[0.05in]
    &y^k_{p,2} \in \partial\varphi^{\downarrow}_p(f_p^{k,\text{lower}}(x^{k, i_k + 1}; x^{k, i_k})), \, p = 1, \cdots, m.
    \end{array}\right\}.
\]
Here $x^{k, i_k + 1}$ is uniquely determined by $x^{k+1} = x^{k, i_k}$ as the minimizer of a strongly convex problem \eqref{eq:subproblem}. 
Notice that $y^k_{p,2}=0$ for $p \in I_1$ since $\varphi_p$ is nondecreasing and $\varphi^\downarrow_p = 0$ for $p \in I_1$.
Let $\{x^{k+1}\}_{k \in N}$ be a subsequence that converges to some point $\bar x$. As we will see in the following lemma, the asymptotic constraint qualification in Assumption 5 implies the non-emptiness and the compactness of $Y^k(x^{k+1})$ for all sufficiently large $k \in N$ and the eventual boundedness of $\{Y^k(x^{k+1})\}_{k \in N}$. These technical results play an important role in the convergence analysis of the prox-ADC method. However, a stronger property of equi-boundedness appears necessary for designing practical termination criteria for the algorithm. We will establish this strengthened property   in section \ref{subsec:terminate} under non-asymptotic constraint qualifications.

\begin{lemma}[non-emptiness and eventual boundedness of multipliers]
\label{lem:CQ_asymptotic}
    Let $\bar x \in \bigcap^m_{p=1} \dom F_p$ be a feasible point of problem \eqref{eq:cvx_composite_constraints}. Suppose that Assumptions 1-5 hold. Consider any sequence $\{x^{k}\}$ generated by the prox-ADC method, with a subsequence $\{x^{k+1}\}_{k \in N}$ converging to $\bar x$. The following statements hold.\\
    (a) The set of multipliers $Y^k( x^{k+1} )$ is non-empty and compact for all sufficiently large $k \in N$.\\
    (b) Additionally, if $\partial^\infty_A f_p(\bar x) = \{0\}$ for $p \in I_2$ (with the definition of $I_2$ in \eqref{defn: index for varphi_p}), then the sequence $\big\{Y^k( x^{k + 1} )\big\}_{k \in N}$ is eventually bounded.
\end{lemma}
\begin{proof}
(a) Observe that $x^{k,i_k+1} \to_N \bar x$ because $x^{k+1} = x^{k,i_k} \to_N \bar x$ and $\|x^{k, i_k} - x^{k, i_k + 1}\| \leq \delta_k / (\lambda + \ell_k) \downarrow 0$ by conditions \eqref{eq:inner_stopping}. The non-emptiness and compactness of $Y^k( x^{k+1} )$ for all sufficiently large $k \in N$ is a direct consequence of the nonsmooth Lagrange multiplier rule \cite[Exercise 10.52]{rockafellar2009variational} for problem \eqref{eq:subproblem} if we can show that, for all {sufficiently large} $k \in N$, $y^k_{m_1 + 1} = \cdots = y^k_{m} = 0$ is the unique solution of the following system
\begin{equation}\label{eq:CQ_contradict}
    0 \in \sum\limits_{p=m_1+1}^{m} y^k_{p} \, \partial f_p^{k,\text{upper}}(x^{k, i_k + 1};x^{k, i_k}),\;\,
    y^k_{p} \in \mathcal N_{(-\infty,0]} (f_p^{k,\text{upper}}(x^{k, i_k + 1}; x^{k, i_k})), \; p=m_1 + 1, \cdots, m.
\end{equation}
Suppose that the above claim does not hold. {Then, there exists a subsequence $N^\prime \subset N$ such that $y^k_{m_1+1} = \cdots = y^k_m = 0$ is not the unique solution of \eqref{eq:CQ_contradict} for all $k \in N^\prime$.} Without loss of generality, {suppose $N^\prime = N$ and} take $\{y^k_{p}\}_{k \in N}$ for $p=m_1 + 1, \cdots, m$ satisfying \eqref{eq:CQ_contradict} and $\sum_{p=m_1+1}^{m} |y^k_{p}| = 1$. For each $p$ and $k \in N$, define
\[
    A^k_p \triangleq \left\{y^k_{p} \, v^k_{p} \,\middle|\,
    v^k_{p} \in \left\{\partial g^k_p(x^{k, i_k}) - \partial h^k_p(x^{k, i_k})\right\} \cup \left\{\partial g^k_p(x^{k, i_k + 1}) - \partial h^k_p(x^{k, i_k + 1} {)} \right\}
    \right\}.
\]
Then, for all $k \in N$, we have
\[
\begin{array}{rl}
    & \dist\left(0, \sum\limits_{p=m_1+1}^{m} A^k_p\right)\\[0.15in]
    \overset{(\rm\romannumeral1)}{\leq}& 
    \dist\left( 0, \sum\limits_{p=m_1+1}^{m}
    y^k_{p} \left[\partial g^k_p(x^{k,i_{k}+1}) - \partial h^k_p(x^{k,i_k})\right] \right)
    + \sum\limits_{p=m_1+1}^{m} \D\left(y^k_{p} \left[\partial g^k_p(x^{k,i_{k}+1}) - \partial h^k_p(x^{k,i_k})\right], A^k_p\,\right) \\[0.15in]

    {\overset{(\rm\romannumeral2)}{\leq}}& {
    \dist\left( 0, \sum\limits_{p=m_1+1}^{m}
    y^k_{p} \, \partial f_p^{k,\text{upper}}(x^{k, i_k + 1};x^{k, i_k}) \right)}\\[0.12in]
    &\qquad\quad {+ \sum\limits_{p=m_1+1}^{m} |y^k_{p}| \cdot \min \Big\{\D\left(\,\partial g^k_p(x^{k,i_{k}+1}),\, \partial g^k_p(x^{k,i_{k}})\,\right),\, \D\left(\,\partial h^k_p(x^{k,i_k}), \partial h^k_p(x^{k,i_k + 1} \right) \Big\} } \\
    
    \overset{(\rm\romannumeral3)}{\leq}&
    0 + \sum\limits_{p=m_1+1}^{m} |y^k_{p}| \cdot \min\Big\{ \Hausd(\partial g^k_p(x^{k,i_k + 1}), \partial g^k_p(x^{k,i_k})),\,
	\Hausd(\partial h^k_p(x^{k,i_k + 1}), \partial h^k_p(x^{k,i_k})) \Big\} \\[0.1in]
	 
    \overset{(\rm\romannumeral4)}{\leq}&
    \sum\limits_{p=m_1+1}^{m} |y^k_{p}| \cdot \ell_k \, \|x^{k,i_k + 1} - x^{k,i_k}\|
    \;\overset{(\rm\romannumeral5)}{\leq}\; \delta_k,
\end{array}
\]
where $(\rm\romannumeral1)$ uses the inequalities $\D(A,C) \leq \D(A,B) + \D(B, C)$ and $\D(A+B, A^\prime + B^\prime) \leq \D(A,A^\prime) + \D(B, B^\prime)$; 
the first term in $(\rm\romannumeral2)$ is because of the construction of upper convex majorization \eqref{eq:majorization}; the second term in $(\rm\romannumeral2)$ is due to $\D(A, B \cup C) \leq \min \{\D(A, B), \D(A, C)\}$ so that
\[
\begin{array}{rl}
    &\D\left(y^k_{p} \left[\,\partial g^k_p(x^{k,i_{k}+1}) - \partial h^k_p(x^{k,i_k})\right],\, A^k_p\,\right)\\[0.08in]
    = & |y^k_{p}| \cdot \D\Big(\,\partial g^k_p(x^{k,i_{k}+1}) - \partial h^k_p(x^{k,i_k}),\, \left\{\partial g^k_p(x^{k, i_k}) - \partial h^k_p(x^{k, i_k})\right\} \cup \left\{\partial g^k_p(x^{k, i_k + 1}) - \partial h^k_p(x^{k, i_k + 1})\right\} \Big)\\[0.1in]
    \leq & |y^k_{p}| \cdot \min \Big\{\D\left(\,\partial g^k_p(x^{k,i_{k}+1}),\, \partial g^k_p(x^{k,i_{k}})\,\right),\, \D\left(\,\partial h^k_p(x^{k,i_k}), \partial h^k_p(x^{k,i_k + 1} \right) \Big\}.
\end{array}
\]
Inequality $(\rm\romannumeral3)$ is due to \eqref{eq:CQ_contradict} and $\D(A, B) \leq \Hausd(A, B)$; $(\rm\romannumeral4)$ is by Assumption 3; and $(\rm\romannumeral5)$ is implied by conditions \eqref{eq:inner_stopping} and $\sum_{p=m_1+1}^{m} |y^k_{p}| = 1$. 
Equivalently, for all $k \in N$ and $p=m_1 + 1, \cdots, m$, there exist $y^k_p \in \mathcal N_{(-\infty,0]}\left(f^{k,\text{upper}}_p(x^{k, i_k + 1}; x^{k, i_k})\right)$ with $\sum_{p=m_1+1}^{m} |y^k_p| = 1$ and 
\[
    v^k_{p} \in \left\{\partial g^k_p(x^{k, i_k}) - \partial h^k_p(x^{k, i_k})\right\} \cup \left\{\partial g^k_p(x^{k, i_k + 1}) - \partial h^k_p(x^{k, i_k + 1})\right\}
\]
such that $\|\sum_{p=m_1+1}^{m} y^k_p \, v^k_{p} \| \leq \delta_k$. {For $p=m_1+1, \cdots, m$, since the subsequence $\{f^k_p (x^{k,i_k+1})\}_{k \in N}$ is bounded by Assumption 1(b), we can assume without loss of generality that $f^k_p (x^{k,i_k+1})$ converges to some $\bar z_p \in T_p(\bar x)$ as $k (\in N) \to +\infty$. Furthermore, it can be easily seen from {\eqref{eq:f_upper} and} \eqref{eq:inner_stopping} that $f^{k,\text{upper}}_p(x^{k,i_k+1}; x^{k,i_k})$ {converges} to the same limit point $\bar z_p$ for $p=m_1+1, \cdots, m$.} 
Notice that $f^{k,\text{upper}}_p(x^{k, i_k + 1}; x^{i, i_k}) \leq 0$ for all $k \in N$ and $p=m_1+1,\cdots,m$ from Theorem \ref{thm:converge1}(a) and, thus, each $\bar z_p$ must satisfy $\bar z_p \leq 0$. Suppose that $y^k_p \rightarrow_N \bar y_p$ for each $p$. Then, by the outer semicontinuity of the normal cone \cite[Proposition 6.6]{rockafellar2009variational},
\[
    \bar y_p \in \mathcal{N}_{(-\infty,0]}(\bar z_p) \subset \bigcup \big\{\mathcal{N}_{\dom\varphi_p}(t_p) \mid t_p \in T_p(\bar{x})\big\}, \quad p=m_1 + 1, \cdots, m.
\]
Obviously, $\sum_{p=m_1+1}^{m} |\bar y_p| = 1$, and the sequence $\{\bar y_p\}^m_{p=m_1+1}$ has at least one nonzero element. Consider two cases.

\gap

\noindent\underline{\textsl{Case 1.}} If $\{v^k_p\}_{k \in N}$ is bounded for $p=m_1 + 1, \cdots, m$, then there are vectors $\{\bar v_p\}^m_{p=m_1+1}$ with $\bar v_p \in \partial_A f_p(\bar x)$ such that $v^k_p \rightarrow_N \bar v_p$ and $0 = \sum_{p=m_1+1}^{m} \bar y_p \, \bar v_p \in \sum_{p=m_1+1}^{m} \bar y_p \, \partial_A f_p(\bar x)$,
contradicting Assumption 5 since $\bar y_{m_1+1}, \cdots, \bar y_{m}$ are not all zeros. 

\gap

\noindent\underline{\textsl{Case 2.}} Otherwise, there exists some $p$ such that $\{v^k_p\}_{k \in N}$ is unbounded. {D}efine the index sets
\[
    I_\text{ub} \triangleq \left\{\, p \in \{m_1+1, \cdots, m\} \,\middle|\, \{v^k_p\}_{k \in N}  \text{ unbounded}\,\right\} \,(\neq \emptyset)\quad \mbox {and}\quad
    I_\text{b} \triangleq \{m_1+1, \cdots, m\} \backslash I_\text{ub}.
\]
Notice that $\big\{\sum_{p \in I_\text{b}} y^k_p \, v^k_p\big\}_{k \in N}$ is bounded. Without loss of generality, assume that this sequence converges to some $\bar w$ and, thus, $\sum_{p \in I_\text{ub}} y^k_p \, v^k_p \rightarrow_N (-\bar w)$.

\gap

\underline{\em Step 1:} Next we prove by contradiction that, for each $p \in I_\text{ub}$, the sequence $\{y^k_p \, v^k_p\}_{k \in N}$ is bounded. Suppose that the boundedness fails and $\sum_{p \in I_\text{ub}} \|y^k_p \, v^k_p\| \rightarrow_N +\infty$ by passing to a subsequence. Consider $\widetilde w^k_p \triangleq y^k_p \, v^k_p/\sum_{p \in I_\text{ub}} \|y^k_p \, v^k_p\|$ for $p \in I_\text{ub}$. Then $\sum_{p \in I_\text{ub}} \widetilde w^k_p \rightarrow_N 0$. 
Since $\sum_{p \in I_\text{ub}} \|\widetilde w^k_p\| = 1$ for all $k \in N$, we can assume that there exist $p_1 \in I_\text{ub}$ and $\widetilde w_{p_1} \neq 0$ such that $\widetilde w^k_{p_1} \rightarrow_N \widetilde w_{p_1}$. It then follows from the construction of $\widetilde w^k_p$ that $\{\widetilde w^k_p\}_{k \in N}$ has a subsequence converging to some element of $\pm\partial^\infty_A f_p(\bar x)$ for each $p \in I_\text{ub}$ and, in particular, $\widetilde w_{p_1} \in \big[ \pm\partial^\infty_A f_{p_1} (\bar x) \backslash\{0\} \big]$. From $\sum_{p \in I_\text{ub}} \widetilde w^k_p \rightarrow_N 0$, we obtain
\[
	0 \in \left[\,\pm\partial^\infty_A f_{p_1}(\bar x) \backslash \{0\}\,\right] + \sum_{p \in I_\text{ub}\backslash\{p_1\}} \left[\,\pm\partial^\infty_A f_p(\bar x)\,\right],
\]
a contradiction to Assumption 5 since the coefficient of the term $[\,\pm\partial^\infty_A f_{p_1}(\bar x) \backslash \{0\}\,]$ is nonzero. So far, we have shown the boundedness of $\{y^k_p \, v^k_p\}_{k \in N}$ for each $p \in I_\text{ub}$.

\gap

\underline{\em Step 2:} Now suppose that $y^k_p \, v^k_p \rightarrow_N \bar w_p$ for each $p \in I_\text{ub}$ with $\sum_{p \in I_\text{ub}} \bar w_p = -\bar w$. Thus $y^k_p \rightarrow_N 0$ and $\bar w_p \in \left[\,\pm\partial^\infty_A f_p(\bar x)\,\right]$ for each $p \in I_\text{ub}$. Since $\sum_{p=m_1+1}^{m} |\bar y_p| = \sum_{p \in I_\text{b}} |\bar y_p| = 1$, there exists $p_2 \in I_\text{b}$ such that $\bar y_{p_2} \neq 0$. Then $\sum_{p=m_1+1}^{m} y^k_p \, v^k_p \rightarrow_N 0$ implies
\[
	0 \in \bar y_{p_2}\,\partial_A f_{p_2}(\bar x) + \sum_{p \in I_\text{b}\backslash\{p_2\}} \bar y_p \,\partial_A f_p(\bar x) + \sum_{p \in I_\text{ub}} \left[\,\pm\partial^\infty_A f_p(\bar x)\,\right],
\]
which leads to a contradiction to Assumption 5. Thus, $Y^k( x^{k+1} )$ is non-empty and compact for all sufficiently large $k \in N$.

\gap

(b) By part (a), assume from now on that $Y^k( x^{k+1} ) \neq \emptyset$ for all $k \in N$ without loss of generality. {We also assume that $\{f^k_p(x^{k,i_k+1})\}_{k \in N}$ converges to some point $\bar z_p \in T_p(\bar x)$ for $p=1,\cdots,m$. Then, by {\eqref{eq:f_upper}, \eqref{eq:f_lower} and} \eqref{eq:inner_stopping}, $f^{k,\text{upper}}_p(x^{k, i_k+1}; x^{k, i_k}) \to_{N} \bar z_p$ for $p=1,\cdots,m$ and $f^{k,\text{lower}}_p(x^{k, i_k+1}; x^{k, i_k}) \to_{N} \bar z_p$ for $p \in I_2$.} For any $k \in N$ and any $(y^k_{1,1}, y^k_{1,2}, \cdots, y^k_{m,1}, y^k_{m,2}) \in Y^k( x^{k+1} )$, we have $y^k_{p,1} \in \partial\varphi^{\uparrow}_p( f^{k,\text{upper}}_p(x^{k, i_k+1}; x^{k, i_k}))$ and $y^k_{p,2} \in \partial\varphi^{\downarrow}_p( f^{k,\text{lower}}_p(x^{k, i_k+1}; x^{k, i_k}))$ for $p=1, \cdots, m$ satisfying
\begin{equation}\label{eq:inner_optimality}
    0 \in \sum\limits_{p=1}^{m} \Big[\,y^k_{p,1} \big[\partial g^k_p(x^{k,i_k + 1}) - \partial\, h^k_p(x^{k,i_k})\big] +
    y^k_{p,2} \big[\partial g^k_p(x^{k,i_k}) - \partial\, h^k_p(x^{k,i_k+1})\big]\Big] + \lambda(x^{k,i_k + 1} - x^{k, i_k}).
\end{equation}
Due to Assumption 3 and similar arguments in the proof of part (a), the optimality condition \eqref{eq:inner_optimality} implies that
\begin{equation}\label{eq:main_converge_error}
\left\{\begin{array}{l}
    \displaystyle\left\| \sum_{p=1}^{m} \big(y^k_{p,1} \, v^k_{p,1} + y^k_{p,2} \, v^k_{p,2}\big) \right\| {\leq \left[\lambda + \sum\limits_{p=1}^{m} \big(|y^k_{p,1}| + |y^k_{p,2}|\big) \ell_k\right] \frac{\delta_k}{\lambda + \ell_k} \leq \max\left\{ 1, \sum\limits_{p=1}^{m} \big(|y^k_{p,1}| + |y^k_{p,2}|\big)\right\} \delta_k}, \\[0.2in]
	\left[
	\begin{array}{cc}
	v^k_{p,1} \in \partial g^k_p(x^{k, i_k}) - \partial h^k_p(x^{k, i_k})\\[0.1in]
	v^k_{p,2} \in \partial g^k_p(x^{k, i_k + 1}) - \partial h^k_p(x^{k, i_k + 1})
	\end{array}
	\right]
	\,\text{or}\,
	\left[
	\begin{array}{cc}
	v^k_{p,1} \in \partial g^k_p(x^{k, i_k + 1}) - \partial h^k_p(x^{k, i_k + 1})\\[0.1in]
	v^k_{p,2} \in \partial g^k_p(x^{k, i_k}) - \partial h^k_p(x^{k, i_k})
	\end{array}
	\right],\, p=1, \cdots, m.
\end{array}\right.
\end{equation}
Note that, for $p \in I_1$, $\varphi_p$ is nondecreasing, i.e., $\varphi^\downarrow_p = 0$. Then ${y^k_{p,2}} = 0$ for all $k \in N$ and $p \in {I_1}$, and the first inequality of \eqref{eq:main_converge_error} is equivalent to
\begin{equation}\label{eq:main_converge_error2}
    \left\| \sum\limits_{p \in I_1} y^k_{p,1} \, v^k_{p,1} + \sum\limits_{p \in I_2}\big(y^k_{p,1} \, v^k_{p,1} + y^k_{p,2} \, v^k_{p,2}\big) \right\|
    \leq
    {\max\left\{1, \sum\limits_{p \in I_1} |y^k_{p,1}| + \sum\limits_{p \in I_2} \big(|y^k_{p,1}| + |y^k_{p,2}|\big) \right\} \delta_k}.
\end{equation}
Observe that the sequences $\{v^k_{p,1}\}_{k \in N}$ and $\{v^k_{p,2}\}_{k \in N}$ must be bounded for $p \in I_2$. Otherwise, we could assume $\|v^k_{p,1}\| \rightarrow_N +\infty$. Then every accumulation point of the unit vectors $\{v^k_{p,1}/\|v^k_{p,1}\|\}_{k \in N}$ would be in the set $\partial^\infty_A f_p(\bar x)$, contradicting our assumption that $\partial^\infty_A f_p(\bar x) = \{0\}$ for each $p \in I_2$.

\gap

For $p \in I_2 \subset \{1, \cdots, m_1\}$, given that $\varphi^\uparrow_p$ is convex, real-valued, and $f^{k,\text{upper}}_p(x^{k, i_k+1}; x^{k, i_k}) \rightarrow_N \bar z_p$, we can invoke \cite[Theorem 24.7]{rockafellar1970convex} to deduce the boundedness of $\{y^k_{p,1}\}_{k \in N}$. A parallel reasoning applies to demonstrate the boundedness of $\{y^k_{p,2}\}_{k \in N}$.

For $p \in I_1$, we proceed by contradiction to establish the boundedness of $\{y^k_{p,1}\}_{k \in N}$ based on Assumption 5. Suppose that $\big\{\sum_{p \in I_1} |y^k_{p,1}|\big\}_{k \in N}$ is unbounded and $\sum_{p \in I_1} |y^k_{p,1}| \rightarrow_N +\infty$ by passing to a subsequence. Consider the normalized subsequences $\big\{\widetilde y^k_{p,1} \triangleq y^k_{p,1} / \sum_{{p^\prime} \in I_1} |y^k_{{p^\prime},1}|\big\}_{k \in N}$ and {$\big\{\widetilde y^k_{p,2} \triangleq y^k_{p,2} / \sum_{p^\prime \in I_1} |y^k_{p^\prime,1}|\big\}_{k \in N}$} for each $p$. Consequently, $\widetilde y^k_{p,1} \rightarrow_N 0$ {and $\widetilde y^k_{p,2} \rightarrow_N 0$} for $p \in I_2$. By the triangle inequality and \eqref{eq:main_converge_error2}, we have
\[
\begin{array}{rl}
    &\left|\; \displaystyle\left\| \sum_{p \in I_1} \widetilde y^k_{p,1} \, v^k_{p,1}\right\| - \left\| \sum_{p \in I_2}\big(\widetilde y^k_{p,1} \, v^k_{p,1} + \widetilde y^k_{p,2} \, v^k_{p,2}\big)  \right\| \; \right| \, 
    \displaystyle\leq \, \left\| \sum_{p \in I_1} \widetilde y^k_{p,1} \, v^k_{p,1} + \sum_{p \in I_2}\big(\widetilde y^k_{p,1} \, v^k_{p,1} + \widetilde y^k_{p,2} \, v^k_{p,2}\big) \right\|\\[0.22in]
    \leq& {\displaystyle\max\left\{ \frac{1}{\sum_{p \in I_1} |y^k_{p,1}|}, 1 + \frac{\sum_{p \in I_2} \big(|y^k_{p,1}| + |y^k_{p,2}|\big)} {\sum_{p \in I_1} |y^k_{p,1}|}  \right\} \delta_k}
    \;\longrightarrow_N\; 0,
\end{array}
\]
which further implies $\big\|\sum_{p \in I_1} \widetilde y^k_{p,1} \, v^k_{p,1} \big\| \rightarrow_N 0$ by the boundedness of $\{v^k_{p,1}\}_{k \in N}$ and $\{v^k_{p,2}\}_{k \in N}$ for $p \in I_2$. Now suppose that $\widetilde y^k_{p,1} \rightarrow_N \widetilde y_{p,1}$ for $p \in I_1$. Then from a similar reasoning in \eqref{eq:horizon2normal}, for $p \in I_1$,
\[
    \widetilde y_{p,1} \in {\Limsup_{k(\in N) \rightarrow +\infty}}^\infty \;\partial\varphi^{\uparrow}_p \left(f^{k,\text{upper}}_p(x^{k, i_k+1}; x^{k, i_k})\right)
    \subset \partial^\infty \varphi^\uparrow_p(\bar z_p) = \mathcal N_{\dom\varphi^\uparrow_p} (\bar z_p),
\]
and obviously $\sum_{p \in I_1} |\widetilde y_{p,1}| = 1$. The remaining argument to derive a contradiction to Assumption 5 follows the same steps as the proof of part (a) for the two cases, with the exception that the index set $\{m_1+1, \cdots, m\}$ is replaced by $I_1$. {Thus, the sequences $\{y^k_{p,1}\}_{k \in N}$ for $p \in I_1 \cup I_2$ and $\{y^k_{p,2}\}_{k \in N}$ for $p \in I_1$ are bounded. We can conclude that $\bigcup \{Y^k( x^{k+1} ) \mid k \in N, k \geq K\}$ is bounded for sufficiently large integer $K$, because otherwise we could extract a subsequence of multipliers from $Y^k( x^{k+1} )$ whose norms diverge to $+\infty$ as $k (\in N) \to +\infty$, in contradiction to the result of boundedness that we have shown. Hence, the subsequence $\{Y^k( x^{k+1} )\}_{k \in N}$ is eventually bounded. }
\end{proof}

We make a remark on Lemma \ref{lem:CQ_asymptotic}(b) about the additional assumption. According to the proof of part (b), the assumption $\partial^\infty_A f_p(\bar x) = \{0\}$ for $p \in I_2$ ensures the boundedness of the set $\partial_A f_p(\bar x)$ for $p \in I_2$. {There are some sufficient conditions for $\partial^\infty_A f_p(\bar x) = \{0\}$ to hold: (i) If $f_p$ is locally Lipschitz continuous and bounded {from} below, {by} {Theorem} \ref{thm:eADC_subdiff}(b), we have $\partial^\infty_A f_p(x) = \{0\}$ at any $x \in \dom f_p$ for the approximating sequence generated by the Moreau envelope. (ii) If $f_p$ is icc associated with $\overline{f}_p$ satisfying all assumptions in {Proposition \ref{proposition: icc ADC}}, it then follows from {Proposition \ref{proposition: icc ADC}(b)} that $\partial^\infty_A f_p(x) = \{0\}$ at any $x \in \interior(\dom f_p)$ for the approximating sequence based on the partial Moreau envelope. It is worth mentioning that the icc function $f_p$ under condition (ii) is not necessarily locally Lipschitz continuous.} 

The main convergence result of the prox-ADC method follows.

\begin{theorem}
\label{thm:main_result}
    Suppose that Assumptions 1-5 hold. Let $\{x^k\}$ be the sequence generated by the prox-ADC method. Suppose that $\{x^k\}$ has an accumulation point $\bar x$ and, in addition, $\partial^\infty_A f_p(\bar x) = \{0\}$ for $p \in I_2$. Then $\bar x$ is a weakly A-stationary point of \eqref{eq:cvx_composite_constraints}. 
    Moreover, if {for each $p \in I_2$, the functions $g^k_p$ and $h^k_p$ are $\ell_k$-smooth for all ${k \in \N}$}, i.e., there exists a sequence $\{\ell_k\} $ such that for all ${k \in \N}$,
\begin{equation}\label{assumption: smooth}
	\max\Big\{\, \big\|\nabla g^k_p(x) - \nabla g^k_p(x^\prime)\big\|,\,
	\big\|\nabla h^k_p(x) - \nabla h^k_p(x^\prime)\big\| \,\Big\} \leq \ell_k \|x^\prime - x\|
	\quad\forall\, x, x^\prime \in \R^n, \; p \in I_2,
\end{equation}
then $\bar x$ is also an A-stationary point of \eqref{eq:cvx_composite_constraints}.
\end{theorem}

\begin{proof}
Let $\{x^{k+1}\}_{k \in N}$ be a subsequence converging to $\bar x$. By the stopping conditions \eqref{eq:inner_stopping} and $x^{k, i_k} \rightarrow_N \bar x$, we also have $x^{k, i_k+1} \rightarrow_N \bar x$. 
First, we prove $\bar x \in \bigcap^m_{p=1} \dom F_p$. {From Theorem \ref{thm:converge1}(a), we have $f^k_p(x^{k, i_k + 1}) \leq 0$ for $p = m_1+1,\cdots,m$ and all $k \in \N$. Due to epi-convergence in Assumption 1(c), it holds that
\[
    \delta_{(-\infty,0]}(f_p(\bar x)) \leq \liminf_{k(\in N) \to +\infty} \delta_{(-\infty,0]}(f^k_p(x^{k, i_k+1})) = 0, \qquad p=m_1+1, \cdots, m.
\]
Thus $f_p(\bar x) \leq 0$ for $p = m_1+1,\cdots,m$ and $\bar x \in {\bigcap^m_{p=m_1+1}\dom F_p}$.} {By Assumption 1(a), $\dom \varphi_p = \R^n$ for all $p = 1, \cdots, m_1$. This implies $\bar x \in \bigcap^{m_1}_{p=1} \dom F_p$, and we can conclude that $\bar x \in \bigcap^m_{p=1} \dom F_p$.}

\gap

By Lemma \ref{lem:CQ_asymptotic}(a), for all {sufficiently large} $k \in N$, we have
\begin{equation}\label{eq:thm4_optimality}
    \hskip -0.05in
    0 \in \sum\limits_{p=1}^{m} \Big[\,y^k_{p,1} \left(\partial g^k_p(x^{k,i_k + 1}) - \partial\, h^k_p(x^{k,i_k})\right) +
    y^k_{p,2} \left(\partial g^k_p(x^{k,i_k}) - \partial\, h^k_p(x^{k,i_k+1})\right)\Big] + \lambda(x^{k,i_k + 1} - x^{k, i_k}),
\end{equation}
where $y^k_{p,1} \in \partial\varphi^{\uparrow}_p( f^{k,\text{upper}}_p(x^{k, i_k+1}; x^{k, i_k}))$ and $y^k_{p,2} \in \partial\varphi^{\downarrow}_p( f^{k,\text{lower}}_p(x^{k, i_k+1}; x^{k, i_k}))$ for $p=1, \cdots, m$. 
It follows from Lemma \ref{lem:CQ_asymptotic}(b) that the subsequences $\{y^k_{p,1}\}_{k \in N}$ and $\{y^k_{p,2}\}_{k \in N}$ are bounded for $p=1,\cdots,m$. Suppose that $y^k_{p,1} \to_N \bar y_{p,1}$ and $y^k_{p,2} \to_N \bar y_{p,2}$ for $p=1, \cdots, m$. 
{Recall that the subsequence $\{f^k_p (x^{k,i_k+1})\}_{k \in N}$ is bounded by Assumption 1(b) for $p=1,\cdots,m$. Without loss of generality, assume that $\{f^k_p(x^{k,i_k+1})\}_{k \in N}$ converges to some point $\bar z_p \in T_p(\bar x)$ for $p=1,\cdots,m$. Then, by {\eqref{eq:f_upper}, \eqref{eq:f_lower} and} \eqref{eq:inner_stopping}, $f^{k,\text{upper}}_p(x^{k, i_k+1}; x^{k, i_k}) \to_{N} \bar z_p$ for $p=1,\cdots,m$ and $f^{k,\text{lower}}_p(x^{k, i_k+1}; x^{k, i_k}) \to_{N} \bar z_p$ for $p \in I_2$.}
From the outer semicontinuity of $\partial\varphi^\uparrow_p$ and $\partial\varphi^\downarrow$, we have $\bar y_{p,1} \in \partial\varphi^\uparrow_p(\bar z_p)$ for $p=1,\cdots,m$ and $\bar y_{p,2} \in \partial\varphi^\downarrow_p(\bar z_p)$ for $p \in I_2$.

\gap

To proceed, we prove by contradiction that the sequence $\{y^k_{p,1} \, v^k_{p,1}\}_{k \in N}$ is bounded for $p \in I_1$. Suppose that $\sum_{p \in I_1} \|y^k_{p,1} \, v^k_{p,1}\| \rightarrow_N +\infty$. {For each $p \in I_2$, the boundedness of $\{v^k_{p,1}\}_{k \in N}$ and $\{v^k_{p,2}\}_{k \in N}$ follows from the assumption $\partial^\infty_A f_p(\bar x) = \{0\}$; otherwise, any accumulation point of the unit vectors $\{v^k_{p,1} / \|v^k_{p,1}\|\}_{k \in N}$ would be in $\partial^\infty_A f_p(\bar x)$, leading to a contradiction. Since $\{y^k_{p,1}\}_{k \in N}$ and $\{y^k_{p,1}\}_{k \in N}$ for $p \in I_2$ are also bounded,} we conclude that the subsequence $\big\{\sum_{p \in I_2}(y^k_{p,1} \, v^k_{p,1} + y^k_{p,2} \, v^k_{p,2}) \big\}_{k \in N}$ is bounded. Thus, we can assume that
\[
    \sum_{p \in I_2} \big(y^k_{p,1} \, v^k_{p,1} + y^k_{p,2} \, v^k_{p,2}\big) \rightarrow_N \; \bar w
    ~\left(
    \in \sum_{p \in I_2} (\bar y_{p,1} \, \partial_A f_p(\bar x) + \bar y_{p,2} \, \partial_A f_p(\bar x))
    \right).
\]
{By \eqref{eq:main_converge_error2}}, it follows that $\sum_{p \in I_1} y^k_{p,1} \, v^k_{p,1} \rightarrow_N (-\bar w)$. Consider $\widetilde w^k_p \triangleq y^k_{p,1} \, v^k_{p,1}/\sum_{{p^\prime} \in I_1} \|y^k_{{p^\prime},1} \, v^k_{{p^\prime},1}\|$ for $p \in I_1$, and then $\sum_{p \in I_1} \widetilde w^k_p \rightarrow_N 0$. Given $\sum_{p \in I_1} \|\widetilde w^k_p\| = 1$ for all $k \in N$, there must exist $p_1 \in I_1$ such that $\widetilde w^k_{p_1} \rightarrow_N \widetilde w_{p_1} \neq 0$. For each $p \in I_1$, it then follows from $y^k_{p,1}/\sum_{{p^\prime} \in I_1} \|y^k_{{p^\prime},1} v^k_{{p^\prime},1}\| \to_N 0$ that $\{\widetilde w^k_p\}_{k \in N}$ has a subsequence converging to some element in $\partial^\infty_A f_p(\bar x)$. In particular, $\widetilde w_{p_1} \in \partial^\infty_A f_{p_1} (\bar x) \backslash\{0\}$. Since $\sum_{p \in I_1} \widetilde w^k_p \rightarrow_N 0$, this implies that
\[
    0 \in \left[\,\partial^\infty_A f_{p_1}(\bar x) \backslash \{0\}\,\right] + \sum_{p \in I_1\backslash\{p_1\}} \partial^\infty_A f_p(\bar x),
    \vspace{-0.1in}
\]
which contradicts Assumption 5. Hence, $\{y^k_{p,1} \, v^k_{p,1}\}_{k \in N}$ is bounded for $p \in I_1$.

\gap

We are now ready to prove that $\bar x$ is a weakly A-stationary point. Suppose that $y^k_{p,1} \, v^k_{p,1} \rightarrow_N \bar w_p$ for $p \in I_1$ with $\sum_{p \in I_1} \bar w_p = -\bar w$. It remains to show that for each $p \in I_1$, there exists $\bar y_{p,1} \in \bigcup \{\partial\varphi^\uparrow_p(t_p) \mid t_p \in T_p(\bar x)\}$ such that
\[
    \bar w_p \in \left\{\, \bar y_{p,1} \, \partial_A f_p(\bar x) \,\right\} \cup \left[\, \partial^\infty_A f_p(\bar x)\backslash\{0\} \,\right],
\]
which can be derived similarly as the proof of \eqref{eq:w_infty_split} in Theorem \ref{thm:necessary_cond}. Summarizing these arguments, we conclude that $\bar x$ is a weakly A-stationary point of \eqref{eq:cvx_composite_constraints}.

\gap

Under the additional assumption of the theorem, there exist $y^k_{p,1} \in \partial\varphi^{\uparrow}_p( f^{k,\text{upper}}_p(x^{k, i_k+1}; x^{k, i_k}))$, $y^k_{p,2} \in \partial\varphi^{\downarrow}_p( f^{k,\text{lower}}_p(x^{k, i_k+1}; x^{k, i_k}))$ {for $p=1,\cdots,m$}, and
\[
    v^k_{p,1} \in \left\{\partial g^k_p(x^{k, i_k}) - \partial h^k_p(x^{k, i_k})\right\} \cup \left\{\partial g^k_p(x^{k, i_k + 1}) - \partial h^k_p(x^{k, i_k + 1})\right\}  {\mbox{ for } p \in I_1}
\]
such that
\begin{equation*}\label{eq:main_conv_add}
\begin{aligned}
    & \left\| \sum_{p \in I_1} y^k_{p,1} \, v^k_{p,1} + \sum_{p \in I_2} (y^k_{p,1} + y^k_{p,2}) \left[\nabla g^k_p(x^{k,i_k}) - \nabla h^k_p(x^{k,i_k})\right] \right\|\\
    \overset{(\rm\romannumeral6)}{\leq} \;&\lambda \|x^{k,i_k + 1} - x^{k,i_k}\| + \sum\limits_{p \in I_1} |y^k_{p,1}| \cdot \min\Big\{ {\Hausd \left(\partial g^k_p(x^{k,i_k + 1}), \partial g^k_p(x^{k,i_k})\right),\, \Hausd \left( \partial h^k_p(x^{k,i_k + 1}), \partial h^k_p(x^{k,i_k})\right)} \Big\} \\
    &+\sum_{p \in I_2} \Big(|y^k_{p,1}| \cdot \|\nabla g^k_p(x^{k, i_k}) - \nabla g^k_p(x^{k, i_k + 1})\| + |y^k_{p,2}| \cdot \|\nabla h^k_p(x^{k, i_k + 1}) - \nabla h^k_p(x^{k, i_k})\|\Big) \\[0.05in]
    \overset{(\rm\romannumeral7)}{\leq} \,& \left[\lambda + \left(\sum_{p \in I_1} |y^k_{p,1}| + \sum_{p \in I_2} \big(|y^k_{p,1}| + |y^k_{p,2}|\big)\right) \ell_k \right] \|x^{k,i_k + 1} - x^{k,i_k}\| \\[0.05in]
    \overset{(\rm\romannumeral8)}{\leq} & {\max\left\{1, \sum\limits_{p \in I_1} |y^k_{p,1}| + \sum\limits_{p \in I_2} \big(|y^k_{p,1}| + |y^k_{p,2}|\big) \right\} \delta_k
    \qquad\forall\, k \in \N},
\end{aligned}
\end{equation*}
where $(\rm\romannumeral6)$ is implied by the optimality condition \eqref{eq:thm4_optimality}, $(\rm\romannumeral7)$ employs \eqref{assumption: smooth} and Assumption 3, and $(\rm\romannumeral8)$ follows from conditions \eqref{eq:inner_stopping}. This inequality is a tighter version of \eqref{eq:main_converge_error2} in the sense that, for each $p \in I_2$ and ${k \in \N}$, $v^k_{p,1}$ and $v^k_{p,2}$ are elements taken from the single-valued mapping $\nabla g^k_p(\cdot) - \nabla h^k_p(\cdot)$ evaluated at the same point $x^{k,i_k}$. A straightforward adaptation of the preceding argument confirms that $\bar x$ is an A-stationary point of \eqref{eq:cvx_composite_constraints}. 
\end{proof}

\gap


{
\subsection{Termination {criteria}.}
\label{subsec:terminate}

The previous subsection demonstrates the asymptotic convergence of the algorithm, showing that any accumulation point of the sequence generated by the prox-ADC method is weakly A-stationary. This subsection is dedicated to the non-asymptotic analysis of verifiable termination criteria for practical implementation.

\gap\gap

\begin{center}
\fbox{\parbox{0.98\textwidth}{
\noindent{\bf Assumption 6} ({\bf non-asymptotic constraint qualifications}) Let $\lambda$ be the parameter in Algorithm \ref{alg:double-loop}. For all $k \in \N$ and any pair $(x^\prime, x^{\prime\prime})$ satisfying
\vspace{-0.1in}
\[
    x^{\prime\prime} =
    \left[\begin{array}{cl}
    \displaystyle\operatornamewithlimits{argmin}_{x \in \R^n}&\;\; \sum\limits_{p=1}^{m_1} \widehat{F^k_p}(x;x^{\prime}) + \frac{\lambda}{2}\|x - x^{\prime}\|^2 \\[0.15in]
    \text{subject to}&\;\; f^{k,\text{upper}}_p(x; x^{\prime}) \leq 0,\, p=m_1 + 1, \cdots, m
    \end{array}\right],
\]
if there exist $y^k_{p} \in \mathcal N_{(-\infty,0]} (f_p^{k,\text{upper}}(x^{\prime\prime}; x^\prime))$ for $p=m_1 + 1, \cdots, m$ such that
\vspace{-0.08in}
\[
    0 \in \sum\limits_{p=m_1+1}^{m} y^k_{p} \, \partial f_p^{k,\text{upper}}(x^{\prime\prime}; x^\prime),
    \vspace{-0.08in}
\]
then we must have $y^k_{m_1 + 1} = \cdots = y^k_{m} = 0$.
}}
\end{center}

A direct consequence of Assumption 6 and the nonsmooth Lagrange multiplier rule \cite[Exercise 10.52]{rockafellar2009variational} is that the set of multipliers $Y^k(x^{k + 1})$ is non-empty and compact for any fixed $k \in \N$. This is in contrast with Lemma \ref{lem:CQ_asymptotic}(a), where the results only hold for sufficiently large $k \in N$. We will show below that the result on the eventual boundedness of the subsequence $\{Y^k(x^{k+1})\}_{k \in N}$ can be strengthened to the equi-boundedness under this assumption.

\begin{proposition}[equi-boundedness of multipliers]
\label{prop:equi-bounded}
    Suppose that Assumptions 1-6 hold. Consider any sequence $\{x^{k}\}$ generated by the prox-ADC method. The following statements hold.\\
    (a) If there is a subsequence $\{x^{k+1}\}_{k \in N}$ converging to some $\bar x$ and $\partial^\infty_A f_p(\bar x) = \{0\}$ for $p \in I_2$, then the subsequence $\big\{Y^k( x^{k+1} )\big\}_{k \in N}$ is equi-bounded.\\
    (b) If $\{x^k\}$ is bounded and $\partial^\infty_A f_p(x) = \{0\}$ for any $x \in \bigcap^m_{p=1} \dom F_p$ and $p \in I_2$, then the sequence $\big\{Y^k( x^{k+1} )\big\}$ is equi-bounded, i,e,
    \begin{equation}
    \label{eq:def_D}
        D \,\triangleq\, \sup_{k \in \N} \, \sup_{y \in Y^k(x^{k+1})} \|y\| < +\infty.
    \end{equation}
\end{proposition}

\begin{proof}
    (a) We know from Lemma \ref{lem:CQ_asymptotic}(b) that the subsequence $\{Y^k(x^{k+1})\}_{k \in N}$ is eventually bounded. This implies the existence of an index $K \in N$ such that $\bigcup \{Y^k(x^{k+1}) \mid k \in N, k \geq K\}$ is bounded. 
    On the other hand, it follows from Assumption 6 that $Y^k(x^{k + 1})$ is non-empty and compact for any fixed $k \in N$. Thus, $\bigcup \{Y^k(x^{k+1}) \mid k \in N\}$ is bounded, and $\{Y^k(x^{k+1})\}_{k \in N}$ is equi-bounded.

    (b) Suppose for contradiction that $\{Y^k(x^{k+1})\}$ is not equi-bounded. Then for any nonnegative integer $j$, there is an index $k_j \in \N$ such that $\|y^{k_j}\| \geq j$ for some multiplier $y^{k_j} \in Y^{k_j}\big(x^{k_j+1}\big)$. Observe that the nonnegative sequence of indices $\{k_j\}_{j \in \N}$ is either bounded or unbounded. It suﬃces to consider these two cases separately.
    
    Suppose first that $\{k_j\}_{j \in \N}$ is bounded. There must be an index $\bar k \in \N$ that appears infinitely many times in $\{k_j\}_{j \in \N}$. Consequently, the set $Y^{\bar k}(x^{\bar k + 1})$ is unbounded, a contradiction to Assumption 6.
    
    Suppose next that $\{k_j\}_{j \in \N}$ is unbounded. For some index set $N^\prime \in \N^\sharp_\infty$, we have $k_j \to +\infty$ as $j(\in N^\prime) \to +\infty$. Notice that the subsequence $\{x^{k_j}\}_{j \in N^\prime}$ is bounded since $\{x^k\}$ is bounded. By passing to a subsequence if necessary, we assume that $\{x^{k_j}\}_{j \in N^\prime}$ converges to some $\widetilde x$. Using epi-convergence in Assumption 1(c) and following the same procedure as in the proof of Theorem \ref{thm:main_result}, we can obtain that $\widetilde x \in \bigcap^m_{p=1} \dom F_p$. Then, by the  assumption in  (b), $\partial^\infty_A f_p(\widetilde x) = \{0\}$ for $p \in I_2$. Henceforth, there is a subsequence $\{x^{k_j}\}_{j \in N^\prime} \subset \{x^k\}$ converging to some $\widetilde x$ with a corresponding subsequence of multipliers $\{y^{k_j} \in Y^{k_j}(x^{k_j+1}) \}_{j \in N^\prime}$ such that $\|y^{k_j}\| \to +\infty$ as $j (\in N^\prime) \to +\infty$, which is a contradiction to the result of part (a).

    We have obtained contradictions for the two cases where $\{k_j\}_{j \in \N}$ is  bounded or unbounded. Then we can conclude that $\{Y^k(x^{k+1})\}$ is equi-bounded and the quantity $D$ defined in \eqref{eq:def_D} is finite. 
\end{proof}

After obtaining the equi-boundedness of the multipliers, we next introduce a relaxation of the weakly A-stationary point for preparation of the  termination criteria. For a proper and convex function $f$ and any $\beta > 0$, we denote $\partial^{\beta} f(\bar x) \triangleq \bigcup\{\partial f(x) \mid x \in \mathbb{B}(\bar x, \beta) \}$,  {which is related to the Goldstein's $\beta$-subdifferential \cite{goldstein1977optimization}.}

\gap
\begin{definition}
    Given any $\bar\eta > 0$, $\bar\beta > 0$ and $\bar k \in \N$, we say a point $x$ is a \textsl{$(\bar\eta, \bar\beta, \bar k)$-weakly A-stationary point} of problem \eqref{eq:cvx_composite} if there exists a nonnegative integer $k \geq \bar k$ such that
\[
    \operatorname{dist}\left(0,\, \sum^m_{p=1} \bigcup
    \left\{
    y_{p,1} \big[\partial^{\bar\beta}g^{k}_p(x) - \partial^{\bar\beta}h^{k}_p(x)\big]
    + y_{p,2} \big[\partial^{\bar\beta}g^{k}_p(x) - \partial^{\bar\beta}h^{k}_p(x) \big]
    \,\middle| \begin{array}{c}
    \, y_{p,1} \in \partial^{\bar\beta} \varphi^\uparrow_p (f^k_p(x)), \\
    y_{p,2} \in \partial^{\bar\beta} \varphi^\downarrow_p (f^k_p(x))
    \end{array}
    \right\}
    \right)
    \leq \bar\eta.
\]
\end{definition}

We remark that, if each outer function $\varphi_p$ is an identity function, i.e., $\varphi_p(t) = t$ for any $t \in \R$, and each inner function $f_p$ is DC rather than ADC, the above definition in the context of a DC program is independent of $k$ and says about nearness to a $\bar\eta$-critical point \cite[definition 2]{yao2022large}.
For nonsmooth optimization problem, similar definitions based on the idea of small nearby subgradients, together with the termination criteria, have appeared in the literature \cite{goldstein1977optimization, burke2005robust}.

The following proposition reveals the relationship between a $(\bar\eta, \bar\beta, \bar k)$-weakly A-stationary point and a weakly A-stationary point.

\begin{proposition}
\label{prop:equivalent_weak_A}
    Let $\bar x \in \bigcap^m_{p=1} \dom F_p$ be a feasible point of \eqref{eq:cvx_composite}. Suppose that Assumption 1 holds and $\partial^\infty_A f_p(\bar x) = \{0\}$ for each $p = 1, \cdots,m$. For any nonnegative sequence $(\eta_k, \beta_k) \downarrow 0$ and some index set $N \in \N^\sharp_\infty$, if each $x^k$ is a $(\eta_k, \beta_k, k)$-weakly A-stationary point of \eqref{eq:cvx_composite} for $k \in N$ and $x^k \to_{N} \bar x$, then $\bar x$ is a weakly A-stationary point of \eqref{eq:cvx_composite}.

    
\end{proposition}

\begin{proof}
    By Assumption 1(b), the subsequence $\{ f^k_p(x^k) \}_{k \in N}$ is bounded for each $p$. Then, there is an index set $N^\prime (\subset N) \in N^\sharp_\infty$ such that $\{f^k_p(x^k)\}_{k \in N^\prime}$ converges to some $\bar t_p \in T_p(\bar x)$ for each $p$. Using the outer semicontinuity of the subdifferential mapping of a convex function, we have
    \[
        \Limsup_{k (\in N^\prime) \to +\infty} \; \partial^{\beta_k} \varphi^\uparrow_p(f^k_p(x^k)) \subset \partial\varphi^\uparrow_p(\bar t_p), \qquad
        \Limsup_{k (\in N^\prime) \to +\infty} \; \partial^{\beta_k} \varphi^\downarrow_p(f^k_p(x^k)) \subset \partial\varphi^\downarrow_p(\bar t_p)
    \]
    and
    \[
        \Limsup_{k (\in N^\prime) \to +\infty} \big[\partial^{\beta_k} g^k_p(x^k) - \partial^{\beta_k} h^k_p(x^k)\big] \subset \partial_A f_p(\bar x).
    \]
    Thus, by taking an outer limit of the subdifferentials involved in the condition that $x^k$ is $(\eta_k, \beta_k, k)$-weakly A-stationary for all $k \in N$, we know that $\bar x$ is a weakly A-stationary point of \eqref{eq:cvx_composite}. 
\end{proof}

We conclude this section with our main result on the termination criteria. 

\begin{proposition}[termination criteria]
\label{prop:Termination}
    Suppose that Assumptions 1-6 hold. Let $\{x^k\}$ be the sequence generated by the prox-ADC method. Suppose that $\{x^k\}$ is bounded and $\partial^\infty_A f_p(x) = \{0\}$ for any $x \in \bigcap^m_{p=1} \dom F_p$ and $p \in I_2$. For any $\bar\eta > 0$, $\bar\beta > 0$ and $\bar k \in \N$, there exists a nonnegative integer $k_0 \geq \bar k$ such that 
    \begin{equation}\label{eq:Termination}
        \max\limits_{1 \leq p \leq m} \sum\limits^{+\infty}_{k^\prime=k_0} \widehat{\alpha^{k^\prime}_p} + \epsilon_{k_0} \leq \bar\beta,
        \quad
        \frac{\delta_{k_0}} {\lambda + \ell_{k_0}} \leq \bar\beta,
        \quad
        \delta_{k_0} \leq \bar\eta.
    \end{equation}
   Consequently, $x^{k_0, i_{k_0}+1}$ is a $\left(\bar\eta \cdot \max\{1, \sqrt{2m} D\}, \bar\beta, \bar k\right)$-weakly A-stationary point of problem \eqref{eq:cvx_composite_constraints}, where $D$ is the constant defined in \eqref{eq:def_D}.
\end{proposition}

\begin{proof}
    The existence of $k_0 \geq \bar k$ satisfying \eqref{eq:Termination} is a direct consequence of $\sum^{+\infty}_{k^\prime = 0} \widehat{\alpha^{k^\prime}_p} < +\infty$, $(\epsilon_k, \delta_k) \downarrow 0$, and $\delta_k / (\lambda + \ell_k) \downarrow 0$. 
    Recalling \eqref{eq:main_converge_error} in the proof of {Lemma \ref{lem:CQ_asymptotic}}, at the $k_0$-th outer iteration, we have
\[
\left\{\begin{array}{l}
    \displaystyle\left\| \sum_{p=1}^{m} \big(y^{k_0}_{p,1} \, v^{k_0}_{p,1} + y^{k_0}_{p,2} \, v^{k_0}_{p,2}\big) \right\| \leq \max\left\{ 1, \sum\limits_{p=1}^{m} \big(|y^{k_0}_{p,1}| + |y^{k_0}_{p,2}|\big)\right\} \delta_{k_0}, \\[0.25in]
	\left[
	\begin{array}{cc}
	v^{k_0}_{p,1} \in \partial g^{k_0}_p(x^{k_0, i_{k_0}}) - \partial h^{k_0}_p(x^{k_0, i_{k_0}})\\[0.1in]
	v^{k_0}_{p,2} \in \partial g^{k_0}_p(x^{k_0, i_{k_0} + 1}) - \partial h^{k_0}_p(x^{k_0, i_{k_0} + 1})
	\end{array}
	\right]
	\text{or}
	\left[
	\begin{array}{cc}
	v^{k_0}_{p,1} \in \partial g^{k_0}_p(x^{k_0, i_{k_0} + 1}) - \partial h^{k_0}_p(x^{k_0, i_{k_0} + 1})\\[0.1in]
	v^{k_0}_{p,2} \in \partial g^{k_0}_p(x^{k_0, i_{k_0}}) - \partial h^{k_0}_p(x^{k_0, i_{k_0}})
	\end{array}
	\right],\, p=1, \cdots, m,
\end{array}\right.
\]
where $y^{k_0}_{p,1} \in \partial\varphi^{\uparrow}_p( f^{k_0,\text{upper}}_p(x^{k_0, i_{k_0}+1}; x^{k_0, i_{k_0}}))$, $y^{k_0}_{p,2} \in \partial\varphi^{\downarrow}_p( f^{k_0,\text{lower}}_p(x^{k_0, i_{k_0}+1}; x^{k_0, i_{k_0}}))$ for $p=1, \cdots, m$. Thus, at the point $x^\ast = x^{k_0, i_{k_0} + 1}$, we have
\[
    \operatorname{dist}\left(0,\, \sum^m_{p=1} \bigcup
    \left\{
    y_{p,1} \big[\partial^{\beta}g^{k_0}_p(x^\ast) - \partial^{\beta}h^{k_0}_p(x^\ast)\big]
    + y_{p,2} \big[\partial^{\beta}g^{k_0}_p(x^\ast) - \partial^{\beta}h^{k_0}_p(x^\ast) \big]
    \,\middle| \begin{array}{c}
    \, y_{p,1} \in \partial^{\beta} \varphi^\uparrow_p (f^{k_0}_p(x^\ast)), \\
    y_{p,2} \in \partial^{\beta} \varphi^\downarrow_p (f^{k_0}_p(x^\ast))
    \end{array}
    \right\}
    \right)
    \leq \eta,
\]
where the parameters $\beta$ and $\eta$ are given by
\[
\begin{array}{l}
    \beta = \max\left\{\begin{array}{c}
        \max\limits_{1 \leq p \leq m} \left[f^{k_0,\text{upper}}_p(x^\ast\,; x^{k_0, i_{k_0}}) - f^{k_0}_p(x^\ast) \right],\\
        \max\limits_{1 \leq p \leq m}
        \left[ f^{k_0}_p(x^\ast) - f^{k_0,\text{lower}}_p(x^\ast\,; x^{k_0, i_{k_0}})\right],\\
        \|x^\ast - x^{k_0, i_{k_0}}\|
        \end{array}\right\}
        \overset{\eqref{eq:inner_stopping}}{\leq} \max\left\{\begin{array}{c}
            \max\limits_{1 \leq p \leq m} \sum\limits^{+\infty}_{k^\prime=k_0} \widehat{\alpha^{k^\prime}_p} + \epsilon_{k_0},\\[0.15in]
        \delta_{k_0} / (\lambda + \ell_{k_0}) \end{array}\right\}
        \overset{\eqref{eq:Termination}}{\leq} \bar\beta, \\[0.4in]
    \eta =  {\max\left\{ 1, \sum\limits_{p=1}^{m} \big(|y^{k_0}_{p,1}| + |y^{k_0}_{p,2}|\big)\right\} \delta_{k_0}}
    \overset{\text{Proposition } \ref{prop:equi-bounded}}{\leq} \max\{1,\, \sqrt{2m} D\} \cdot \delta_{k_0}
    \overset{\eqref{eq:Termination}}{\leq} \bar\eta \cdot \max\{1,\, \sqrt{2m} D\}.
\end{array}
\]
Henceforth, for $k_0$ satisfying \eqref{eq:Termination}, $x^\ast = x^{k_0,i_{k_0}+1}$ is a $\left(\bar\eta \cdot \max\{1,\, \sqrt{2m} D\}, \bar\beta, \bar k\right)$-weakly A-stationary point of problem \eqref{eq:cvx_composite_constraints}. 
\end{proof}

\section{Numerical {examples}.}
\label{sec:numerical}
We present some preliminary experiments to illustrate the performance of our algorithm on the inverse optimal value optimization with or without constraints. The first experiment aims to demonstrate the practical performance of the prox-ADC method under the termination criteria in section \ref{subsec:terminate}, by varying different approximating sequences and initial points. To demonstrate the computation of ADC constrained problems, especially the choice of the quantity $\widehat{\alpha^k_p}$ and a feasible initial point in Assumption 2, we further consider the constrained inverse optimal value optimization. These experiments were tested on a MacBook Air laptop with an Apple M1 chip and 16GB of memory using Julia 1.10.2.
}

\subsection{Inverse optimal value optimization with simple constraints.}
\label{subsec:Ex1}
{
Based on the setting in \eqref{eq:inverse}, we aim to find a vector $x \in [-1, 1]^n$ to minimize the errors between the observed optimal values ${\{\nu_p\}^m_{p=1}}$ and true optimal values $\{f_p(x)\}^m_{p=1}$:
\begin{equation}\label{numerical1}
    \Min_{x \in [-1, 1]^n} \;F(x) \triangleq \sum_{p=1}^{m} \left|{\nu_p} - f_p(x)\right|,
\end{equation}
where each $f_p$ is the optimal value function as defined in \eqref{eq:value_func}. We fix $n = 10$, $m = 11$, $d = 10$, and the number of inequality constraints $\ell = 5$ in the minimization problem \eqref{eq:value_func}. Vectors $b^{\,p}$ and $c^{\,p}$, and matrices $A^{\,p}, B^{\,p}, C^{\,p}$ are randomly generated with each entry independent and normally distributed with mean $\mu=0$ and variance $\sigma=1$. For numerical stability, we then normalize matrices $C^{\,p}$ and $A^{\,p}$ by a factor of $\sqrt{n}$. We also generate a positive definite matrix $Q^{\, p}$ and a random solution $x^\ast = u /\|u\|$ with $u \sim \mbox{Normal}(0, I_n)$. We set ${\nu_p} = f_p(x^\ast)$ for each $p$ and, therefore, $F(x^\ast) = 0$.

We adopt the ADC decomposition in \eqref{eq:PME}, denoted by $f^k_p = (g_p)_{\gamma_k} - (h_p)_{\gamma_k}$ with a sequence $\{\gamma_k = 1/ (k+1)^\rho\}$ for some exponent $\rho > 0$. Consequently, $\ell_k = 1 / \gamma_k = (k+1)^\rho$. We apply the prox-ADC algorithm to solve this example with $\epsilon_k = \delta_k = 1 / (k+1)^\rho$ and $\lambda = 5$. In this example, the strongly convex subproblem \eqref{eq:subproblem} can be easily reformulated to a problem with linear objective and convex quadratic constraints, which is  solved by Gurobi in our experiments.

We first investigate the performance of our algorithm under the termination criteria with different values of parameters. Figure \ref{figEx1:varying_terminate} displays the logarithm of the objective values against the number of outer iterations and the total number of inner iterations. We mark three different points on the curve where the termination criteria \eqref{eq:Termination} with $\bar\eta = \bar\beta = 10^{-1}, 10^{-2}, 10^{-3}$ and $\bar k = 10, 20, 40$ are satisfied. 

\begin{figure}[H]
    \begin{subfigure}
    \centering
    \includegraphics[scale=0.38]{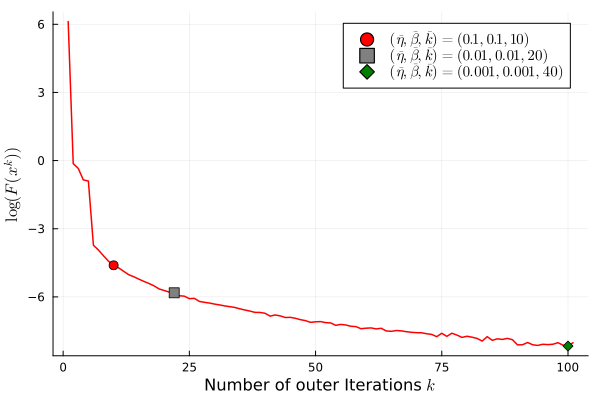}
    \end{subfigure}
    \begin{subfigure}
    \centering
    \includegraphics[scale=0.38]{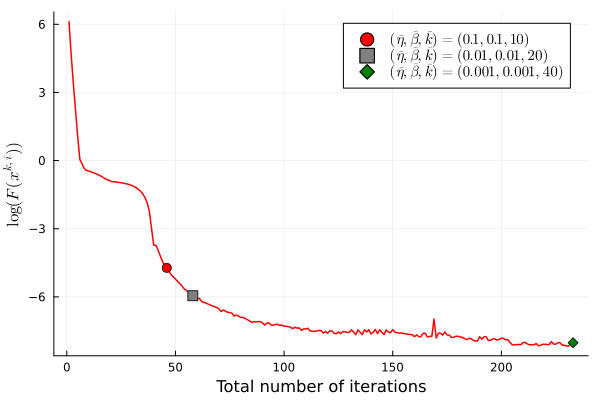}
    \end{subfigure}
    \caption{Performance of the prox-ADC method for problem \eqref{numerical1}, under the termination criteria \eqref{eq:Termination} with $\bar\eta = \bar\beta = 10^{-1}, 10^{-2}, 10^{-3}$ and $\bar k = 10, 20, 40$, for a fixed exponent $\rho = 1.5$ and a fixed initial point.}
    \label{figEx1:varying_terminate}
\end{figure}

We have also experimented with various values of exponent $\rho$ that determine the convergence rate of the approximating sequence and various initial points. In both cases, we terminate the algorithm under the conditions \eqref{eq:Termination} with $\bar\eta = \bar\beta = 10^{-2}$ and $\bar k = 10$. In Figure \ref{figEx1:varying_exponent}, we observe that setting different values of $\rho$ under the same termination criteria leads to candidate solutions with similar objective values, and there are roughly two phases of convergence in terms of the total number of iterations. Initially, the objective value decreases faster for smaller $\rho$, corresponding to poorer approximation. When the objective value is sufficiently small ($10^{-3}$ on this particular instance), larger $\rho$ results in faster convergence to high accuracy. We remark that for $\rho = 0.5$, the algorithm reaches the maximum number of outer iterations and does not output a $(10^{-2}, 10^{-2}, 10)$-weakly A-stationary point. 
Figure \ref{figEx1:varying_initials} demonstrates the influence of using various initial points that are uniformly distributed on $[-1, 1]^n$. On this instance, two of the initial points find $(10^{-2}, 10^{-2}, 10)$-weakly A-stationary points with large objective values. For these two initial points, we rerun the algorithm with $\bar\eta = \bar\beta = 10^{-3}$, and the algorithm still terminates with large objective values.



\begin{figure}[H]
    \begin{subfigure}
    \centering
    \includegraphics[scale=0.38]{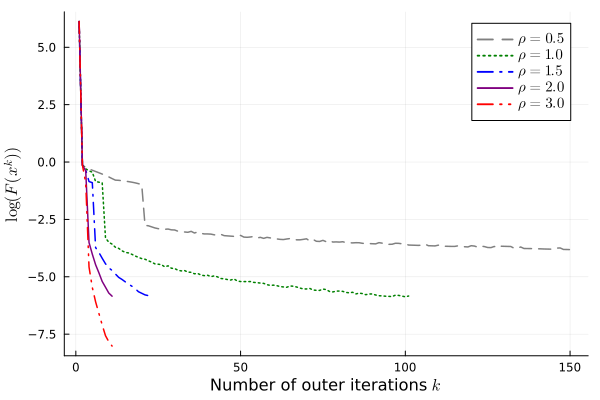}
    \end{subfigure}
    \begin{subfigure}
    \centering
    \includegraphics[scale=0.38]{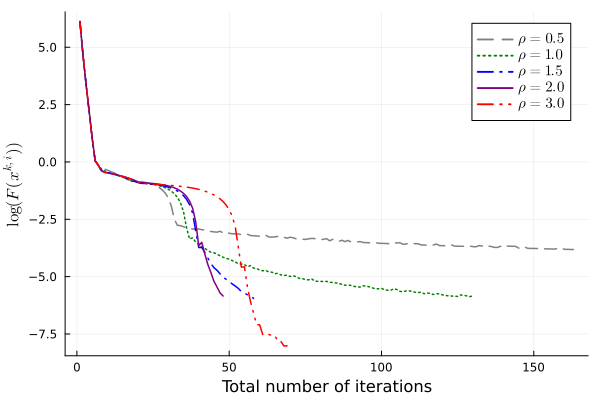}
    \end{subfigure}
    \caption{Performance of the prox-ADC method for problem \eqref{numerical1} using different values of exponent $\rho$, under the termination criteria \eqref{eq:Termination} with $\bar\eta = \bar\beta = 10^{-2}$ and $\bar k = 10$, for a fixed initial point.}
    \label{figEx1:varying_exponent}
\end{figure}

\begin{figure}[H]
    \begin{subfigure}
    \centering
    \includegraphics[scale=0.38]{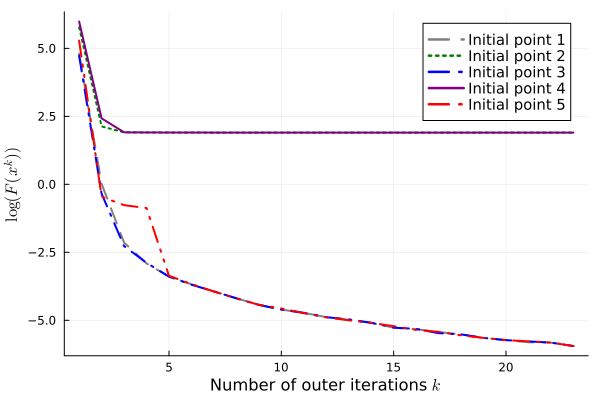}
    \end{subfigure}
    \begin{subfigure}
    \centering
    \includegraphics[scale=0.38]{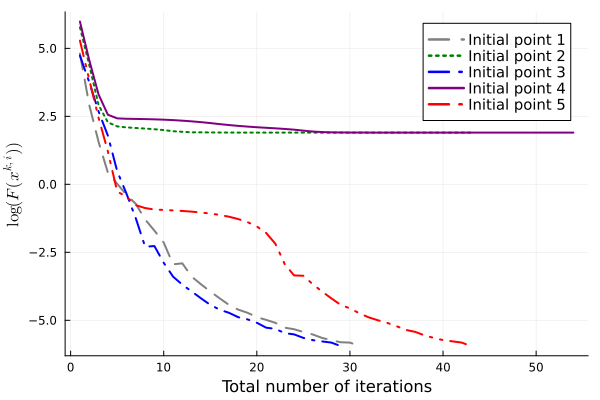}
    \end{subfigure}
    \caption{Performance of the prox-ADC method  for problem \eqref{numerical1} using five initial points uniformly distributed on $[-1, 1]^n$, under the termination criteria \eqref{eq:Termination} with $\bar\eta = \bar\beta = 10^{-2}$ and $\bar k = 10$, for a fixed exponent $\rho = 1.5$.}
    \label{figEx1:varying_initials}
\end{figure}

\subsection{Inverse optimal value optimization with ADC constraints.}
\label{subsec:Ex2}
We consider a variant of the inverse optimal value optimization that is defined as follows:
\begin{equation}
\label{eq:numerical_Ex2}
\begin{array}{cl}
    \Min_{x \in [-1,1]^n} &\;\, \displaystyle F(x) = \sum^{m_1}_{p=1} \left| {\nu_p} - f_p(x) \right| \\[0.2in]
    \mbox{subject to} &\;\, \displaystyle\frac{{\nu_p} - f_p(x)}{\max\{1, |{\nu_p}|\}} \leq \varepsilon,\quad
    \frac{f_p(x) - {\nu_p}}{\max\{1, |{\nu_p}|\}} \leq \varepsilon,
    \quad p=m_1+1, \cdots, m.
\end{array}
\end{equation}
In this formulation, the observations of the optimal values ${\{\nu_p\}^m_{p=1}}$ are divided into two groups indexed by $\{1, \cdots, m_1\}$ and $\{m_1+1, \cdots, m\}$. We aim to minimize the errors for the first group while ensuring the relative errors for the second group do not exceed a specified feasibility tolerance, denoted by $\varepsilon$. In our experiment, we fix $n = 10$, $m = 11$, $m_1 = 8$, $\varepsilon = 10^{-1}$, $d = 10$, and the number of inequality constraints $\ell = 5$ in the minimization problem \eqref{eq:value_func}. The solution $x^\ast$ and the data, including ${\{\nu_p\}^m_{p=1}}$, are randomly generated in the same way as in section \ref{subsec:Ex1}. We can see that $x^\ast$ is feasible to \eqref{eq:numerical_Ex2} and attains the minimal objective value $F(x^\ast) = 0$.  

Similar to the first example, we adopt the ADC decomposition in \eqref{eq:PME}, denoted by $f^k_p = (g_p)_{\gamma_k} - (h_p)_{\gamma_k}$ with a sequence $\{\gamma_k = 1/ (k + \tilde k)^\rho\}$ for some positive integer $\tilde k$ and $\rho > 0$. Due to the feasibility problem in Assumption 2, we introduce the additional parameter $\tilde k$ to control the approximating sequences, which will be explained in details later. We also note that treating $\frac{{\nu_p} - f_p(x)}{\max\{1, |{\nu_p}|\}} \leq \varepsilon$ and $\frac{f_p(x) - {\nu_p}}{\max\{1, |{\nu_p}|\}} \leq \varepsilon$ as two separate constraints for $p={m_1+1},\cdots,m$  leads to the failure of the asymptotic constraint qualification in {Assumption 5} because the approximate subdifferentials of the ADC functions $f_p$ and $-f_p$ are linearly dependent. This issue can be resolved by rewriting the constraints in a composite ADC form $\frac{|{\nu_p} - f_p(x)|}{\max\{1, |{\nu_p}|\}} \leq \varepsilon$ and assuming a corresponding version of {Assumption 5}. We omit this technical detail since the main focus of this section is to illustrate the practical implementation of our algorithm.

To verify Assumption 2 that states the existence of a strictly feasible point, we first follow the discussion after Assumption 2 to construct the quantity $\widehat{\alpha^k_p} = (\gamma_k - \gamma_{k+1}) L_p^2 / (2 \max\{1, |{\nu_p}|\})$ where $L_p$ is the Lipschitz constant of $\overline f_p(\cdot, x)$ for all $x \in [-1, 1]^n$. We can derive the Lipschitz constant $L_p$ by characterizing the subdifferential $\partial_1 \overline{f_p}(\cdot, x)$ for a fixed $x$ based on Danskin's Theorem \cite[Theorem 2.1]{clarke1975generalized} and then upper bounding the norm of this subdifferential over $x \in [-1, 1]^n$. The extra denominator $\max\{1, |{\nu_p}|\}$ in the expression of $\widehat{\alpha^k_p}$ is due to the scaling of the constraints in \eqref{eq:numerical_Ex2}. Then, consider the following problem:
\begin{equation}
\label{eq:feasibility}
   \Min_{x \in [-1, 1]^n} \;V(x) \triangleq \sum^m_{p=m_1+1} \max\Bigg\{0,\, \big|{\nu_p} - f^0_p(x)\big| - \underbrace{\left(\varepsilon - \sum^{+\infty}_{k^\prime=0} \widehat{\alpha^{k^\prime}_p}\right) \max\{1, |{\nu_p}|\}}_{\triangleq\, s_p \,(\text{constant})} \Bigg\},
\tag{Feas}
\end{equation}
where the objective is the sum of the compositions of univariate convex functions $\varphi_p (t) = \max\{0, |{\nu_p} - t| - s_p\}$ and DC functions $f^0_p$. Notice that problem \eqref{eq:feasibility} takes the same form as \eqref{eq:cvx_composite_perturb}. Thus, we can apply the inner loop of the prox-ADC method to solve it approximately. If solving this problem gives a solution $x^0$ with $ V(x^0) = 0$, then
\begin{equation}
\label{eq:strict_feasible}
    \frac{|{\nu_p} - f^0_p(x)|}{\max\{1, |{\nu_p}|\}} \leq \varepsilon - \sum^{+\infty}_{k^\prime=0} \widehat{\alpha^{k^\prime}_p}
    = \varepsilon - \frac{(L_p)^2}{2 \, \tilde k^{\rho} \cdot \max\{1, |{\nu_p}|\}},
\end{equation}
and $x^0$ is a strictly feasible point satisfying Assumption 2. We emphasize that using the inner loop of the prox-ADC method for solving problem \eqref{eq:feasibility} to obtain a strictly feasible point is merely a heuristic. Although this approach works well in our experiments, it is generally not easy to verify Assumption 2. We make a final remark on the role of $\tilde k$. For small values of $\rho$ and $\tilde k$, it is possible that $\varepsilon < \frac{(L_p)^2}{2 \, \tilde k^\rho \cdot \max\{1, |{\nu_p}|\}}$, and, from \eqref{eq:strict_feasible}, there is no strictly feasible point satisfying Assumption 2 for this fixed approximating sequence. Henceforth, the flexibility of the parameter ${\tilde k}$ is necessary to ensure the validity of Assumption 2.

We implement the above procedure to find an initial point and then apply the prox-ADC method with $\epsilon_k = \delta_k = 1 / (k+1)^\rho$ and $\lambda = 5$. On most of the randomly generated instances, we observe that the point given by solving \eqref{eq:feasibility} is also feasible to the original problem \eqref{eq:numerical_Ex2} along the iterations, although this result cannot be implied by Assumption 2. 
In Figure \ref{figEx2}, we again plot the logarithm of the objective values against the total number of iterations, using various combination of $\tilde k$ and $\rho$ and various initial points. It is worth mentioning that for this constrained problem, the random initial point is not directly utilized in the prox-ADC method. Instead, it is first used in problem \eqref{eq:feasibility} to generate a strictly feasible point satisfying Assumption 2, and the candidate solution for \eqref{eq:feasibility} then becomes the initial point of the prox-ADC method for solving \eqref{eq:numerical_Ex2}.
}

\begin{figure}[H]
    \subfigure[\,\footnotesize Varying sequences $\{\gamma_k\}$.]{
    \begin{minipage}{.49\textwidth}
    \centering
    \includegraphics[scale=0.38]{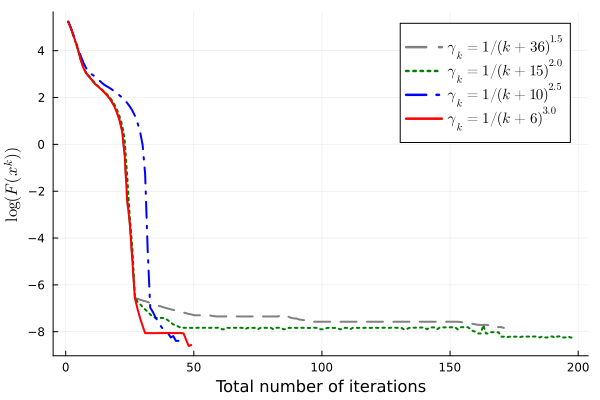}
    \end{minipage}}
    \subfigure[\,\footnotesize Varying initial points.]{
    \begin{minipage}{.49\textwidth}
    \centering
    \includegraphics[scale=0.38]{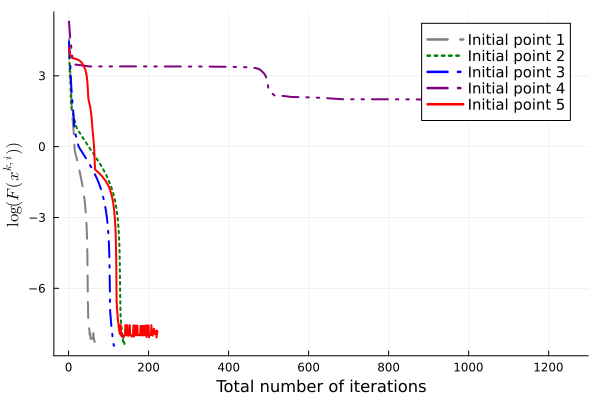}
    \end{minipage}}
    \caption{Performance of the prox-ADC method for problem \eqref{eq:numerical_Ex2} under the termination criteria \eqref{eq:Termination} with $\bar\eta = \bar\beta = 2 \times 10^{-2}$ and $\bar k = 5$. (a): using different sequences $\{\gamma_k\}$ for a fixed initial point. (b): using five initial points uniformly distributed on $[-1, 1]^n$ for a fixed sequence $\{\gamma_k = 1/ (k + 10)^{2.5}\}$.}
    \label{figEx2}
\end{figure}

%
%
%

\section*{Acknowledgments.} The authors are partially supported by the National Science Foundation under grants CCF-2416172 and DMS-2416250, and the  National Institutes of Health under grant 1R01CA287413-01.
{The authors are grateful
to the associate editor and two reviewers for their careful reading and constructive comments that have substantially improved the paper.}

\bibliographystyle{plain}
\bibliography{reference} 

\section*{Appendix A. Proofs of Proposition \ref{proposition: icc ADC} and Proposition \ref{proposition: ADC of Var}.}
\begin{proof}[Proof of Proposition \ref{proposition: icc ADC}.]

(a) {We first generalize the {convergence result} of the classical Moreau envelopes when $\gamma_k \downarrow 0$ (see, e.g., \cite[Theorem 1.25]{rockafellar2009variational}) to the partial Moreau envelopes. Fixing any $\gamma_0 > 0$, we consider the function $\psi(z,x,\gamma) \triangleq \overline f(z,x) + \delta_{\dom f}(x) + \psi_0(z,x,\gamma)$ with
\[
    \psi_0(z,x,\gamma) \triangleq \left\{
    \begin{array}{cl}
	\|z - x\|^2/(2\gamma) &\text{ if }\gamma \in (0,\gamma_0], \\
	0 &\text{ if }\gamma = 0, z = x,\\
	\infty &\text{ otherwise.}
    \end{array}\right.
\]
Notice that $f^k(x) = g_{\gamma_k}(x) - h_{\gamma_k}(x) + \delta_{\dom f}(x) = \inf_z \psi(z,x,\gamma_k)$. It is easy to verify that $\psi$ is proper and lsc based on our assumptions. Under the assumption that $\overline{f}$ is bounded from below on $\dom f \times \dom f$, we can also show by contradiction that $\psi(z,x,\gamma)$ is level-bounded in $z$ locally uniformly in $(x, \gamma)$. Consequently, it follows from \cite[Theorem 1.17]{rockafellar2009variational} that $f^k(x) = \inf_z \psi(z,x,\gamma_k) \uparrow f(x)$ for any fixed $x$ and each $f^k$ is lsc.}

Hence, $f^k \Epiconv f$ is a direct consequence of \cite[Proposition 7.4(d)]{rockafellar2009variational} by $f^k(x) \uparrow f(x)$ for all $x$ and the lower semicontinuity of $f^k$. If $\dom f = {\R^n}$, then $f$ is continuous, and thus $f^k \Contconv f$ by \cite[Proposition 7.4(c-d)]{rockafellar2009variational}. 
We then complete the proof of (a).

(b) For any $\bar x \in \interior(\dom f)$,
\[
\begin{array}{rl}
	\partial_{A} f(\bar x)
	=\;\;\,& \bigcup\limits_{x^k \rightarrow \bar x} \Limsup_{k \rightarrow +\infty} \;\left\{\partial {g_{\gamma_k}}(x^k) - \partial {h_{\gamma_k}}(x^k)\right\}\\
	\overset{(\rm\romannumeral1^\prime)}{=}\;\;\,& \bigcup\limits_{x^k \rightarrow \bar x}\Limsup_{k \rightarrow +\infty} \;\left\{\frac{x^k}{\gamma_k} - \partial_2 (-\overline{f})(z^k, x^k) - \frac{z^k}{\gamma_k} \;\middle|\; z^k = \argmin_{z \in {\R^{n}} } \left[\,\overline{f}(z,x^k) + \frac{1}{2\gamma_k}\|z - x^k\|^2\right] \right\}\\
	\,\overset{(\rm\romannumeral2^\prime)}{\subset}\;\;\,& \bigcup\limits_{(x^k, z^k) \rightarrow (\bar x, \bar x)} \Limsup_{k \rightarrow +\infty} \;\left[\partial_1 \overline{f}(z^k, x^k) - \partial_2 (-\overline{f})(z^k, x^k) \right]\\
	\overset{(\rm\romannumeral3^\prime)}{=}\;\; & \partial_1 \overline{f}(\bar x,\bar x) - \partial_2 (-\overline{f})(\bar x,\bar x),
\end{array}
\]
where $(\rm\romannumeral1^\prime)$ follows from the convexity of $(-\overline f)(z,\cdot)$ for any $z \in \dom f$ and Danskin's Theorem \cite[Theorem 2.1]{clarke1975generalized}; $(\rm\romannumeral2^\prime)$ is due to the optimality condition for $z^k$,  and $z^k \rightarrow \bar x$  is obtained by similar arguments in the proof of Theorem \ref{thm:eADC_subdiff}(b) due to our assumption that $\overline{f}$ is bounded from below on $\dom f \times \dom f$; and $(\rm\romannumeral3^\prime)$ uses the outer semicontinuity of $\partial_1 \overline{f}$ and $\partial_2 (-\overline{f})$ at $(\bar x, \bar x)$ \cite[Lemma 5]{li2022decomposition}. Therefore, for any $\bar x \in \interior(\dom f)$, $\partial f(\bar x) \subset \partial_{A} f(\bar x) \subset \partial_1 \overline{f}(\bar x,\bar x) - \partial_2 (-\overline{f})(\bar x,\bar x)$. 
Moreover, due to the local boundedness of the mappings $\partial_1 \overline{f}$ and $\partial_2 (-\overline{f})$ at $(\bar x, \bar x)$ \cite[Lemma 5]{li2022decomposition}, it follows from \cite[Example 4.22]{rockafellar2009variational} that $\partial^\infty_A f(\bar x) = \{0\}$. 
\end{proof}

\vspace{0.15in}

\begin{proof}[Proof of Proposition \ref{proposition: ADC of Var}.]
    (a) Note that for any $x \in \R^n$, ${\mbox{CVaR}_{\alpha}}[\, c(x,Z) \,]$ is well-defined and takes finite value due to $\mathbb{E}[\, |c(x,Z)|\,] < +\infty$. 
    Since $c(x,Z)$ follows a continuous distribution for any $x \in \R^n$, we know {from \cite[Theorem 1]{rockafellar2000optimization} and \cite{acerbi2002spectral} that CVaR has the following equivalent representations:}
    \[
        {\mbox{CVaR}_{\alpha}}[\,c(x,Z)\,] = \inf_{t \in \R}\left\{ t + \frac{1}{1-\alpha} \, \mathbb{E}\left[\,\max\{c(x,Z)-t,0\}\right] \right\} = \frac{1}{1-\alpha} \int_{\alpha}^{1} \mbox{VaR}_t[\,c(x,Z)\,] \, \mbox{d}t.
    \]
    Moreover, ${\mbox{CVaR}_\alpha}[\,c(\cdot,Z)\,]$ is convex by the convexity of $c(\cdot,z)$ for any fixed $z \in \R^m$ (cf. \cite[Theorem 2]{rockafellar2000optimization}). Therefore, both $g^k$ and $h^k$ defined in \eqref{eq:VaR_ADC} are convex. By the definitions of $g^k$ and $h^k$, we have
\[
\begin{array}{rl}
    g^k(x) - h^k(x) &= [k(1-\alpha)+1] \, \mbox{CVaR}_{\alpha-1/k}[c(x,Z)] - k(1-\alpha) \, \mbox{CVaR}_{\alpha}[c(x,Z)]\\[0.08in]
    &= \displaystyle\frac{k(1-\alpha) + 1}{1 - (\alpha - 1/k)} \int^{1}_{\alpha - 1/k} \operatorname{VaR}_t [c(x,Z)] \,\mbox{d}t - \frac{k(1-\alpha)}{1 - \alpha} \int^{1}_{\alpha} \operatorname{VaR}_t [c(x,Z)] \,\mbox{d}t\\[0.15in]
    & = k \displaystyle\int^{1}_{\alpha - 1/k} \operatorname{VaR}_t [c(x,Z)] \,\mbox{d}t - k \int^1_\alpha \operatorname{VaR}_t [c(x,Z)] \,\mbox{d}t\\[0.15in]
    &= k \displaystyle\int^{\alpha}_{\alpha - 1/k} \operatorname{VaR}_t [c(x,Z)] \,\mbox{d}t.
\end{array}
\]
Note that $\operatorname{VaR}_t[c(x,Z)]$ is nondecreasing as a function of $t$ for any fixed $x \in \R^n$. Namely,
\[
     \displaystyle\int^{\alpha}_{\alpha - 1/k} \operatorname{VaR}_{\alpha - 1/k}[c(x,Z)] \,\mbox{d}t
     \leq \displaystyle\int^{\alpha}_{\alpha - 1/k} \operatorname{VaR}_{t}[c(x,Z)] \,\mbox{d}t
     \leq \displaystyle\int^{\alpha}_{\alpha - 1/k} \operatorname{VaR}_{\alpha}[c(x,Z)] \,\mbox{d}t.
\]
Thus, $\mbox{VaR}_{\alpha-1/k}[\, c(x,Z) \,] \leq g^k(x) - h^k(x) \leq \mbox{VaR}_{\alpha}[\, c(x,Z) \,]$ for any $x \in \R^n$ and $k > 1/\alpha$. 
    Since $\mbox{VaR}_t[\, c(x,Z) \,]$ as a function of $t$ on $(0,1)$ is left-continuous, it follows that $[g^k(x) - h^k(x)] \uparrow \mbox{VaR}_{\alpha}[\, c(x,Z) \,]$ for all $x$.
    Observe that
    \[
        \{x \mid \mbox{VaR}_{\alpha}[\, c(x,Z) \,] \leq r\} = \{ x \mid \mathbb{P}(c(x,Z) \leq r) \geq \alpha \}.
    \]
    Based on our assumptions and \cite[Proposition 2.2]{van2020discussion}, for any $r \in \R$, the probability function $x \mapsto -\mathbb{P}(c(x,Z) \leq r)$ is lsc, which implies the closedness of the level set $\{x \mid \mathbb{P}(c(x,Z) \leq r) \geq \alpha\}$ for any $(r, \alpha) \in \R \times (0,1)$. Hence, $\mbox{VaR}_\alpha [\, {c(\cdot,Z)}]$ is lsc for any given $\alpha \in (0,1)$ and is continuous if $c(\cdot, \cdot)$ is further assumed to be continuous. {Then (a) is a direct consequence of \cite[Proposition 7.4(c-d)]{rockafellar2009variational} by the monotonicity $[g^k(x) - h^k(x)] \uparrow \operatorname{VaR}_{\alpha}[\, c(x,Z)]$ and the continuity of $\mbox{VaR}_{\alpha}[\, c(\cdot,Z)]$.}
    
    (b) {We use $\mathcal{L}_1(\Omega, \mathcal{F}, \mathbb{P})$ to denote the space of all random variables ${X}: \Omega \to \R$ with $\mathbb{E}[\,|{X}(\omega)|] < +\infty$. According to \cite[Example 6.19]{shapiro2021lectures}, the function ${\mbox{CVaR}_{\alpha}}:\mathcal{L}_1(\Omega, \mathcal{F}, \mathbb{P}) \to \R$ is subdifferentiable (see \cite[(9.281)]{shapiro2021lectures} for the definition). Consider any fixed $x \in \R^n$. Given that $c(x, Z)$ is a continuous random variable {in $\mathcal{L}_1 (\Omega, \mathcal{F}, \mathbb{P})$}, it follows from \cite[(6.81)]{shapiro2021lectures} that the subdifferential of ${\mbox{CVaR}_{\alpha}}[\cdot]$ at $c(x,Z)$ is:}
    \begin{equation}\label{eq:CVaR_subdiff}
        {\partial \left(\mbox{CVaR}_{\alpha} [\,\cdot\,]\right)[\,c(x,Z)]} = \left\{
        \phi \in \mathcal{L}_\infty(\Omega, \mathcal{F}, \mathbb{P}) \,\middle|
        \begin{array}{ll}
            \phi(\omega) = (1-\alpha)^{-1} & \text{if } c(x,Z(\omega)) > \mbox{VaR}_{\alpha}[\,c(x,Z)] \\
            {\phi(\omega) \in [0, (1-\alpha)^{-1}]} & {\text{if } c(x,Z(\omega)) = \mbox{VaR}_{\alpha}[\,c(x,Z)]} \\
            \phi(\omega) = 0 & \text{if } c(x,Z(\omega)) < \mbox{VaR}_{\alpha}[\,c(x,Z)]
        \end{array}
        \right\}.
    \end{equation}
    {We would like to mention that the event $\{ \omega \in \Omega \mid c(x,Z(\omega)) = \mbox{VaR}_{\alpha}[\,c(x,Z)]\}$ has zero probability and, thus, $\mathbb{E}[\,\phi\,] = (1-\alpha)^{-1} \cdot (1-\alpha) = 1$ for every random variable $\phi \in \partial \left(\mbox{CVaR}_{\alpha} [\,\cdot\,]\right) \,[\,c(x,Z)]$.} Let $\mathbb{P}_Z$ denote the probability measure associated with $Z$. By using \cite[Theorem 6.14]{shapiro2021lectures}, we obtain {the subdifferential of the convex function ${\mbox{CVaR}_{\alpha}}[\, c(\cdot,Z) \,]$ at $x$:}
    \begin{equation}\label{eq:CVaR_comp_subdiff}
    \begin{split}
    \begin{array}{rl}
        \partial \left({\mbox{CVaR}_{\alpha}}[\, c(\cdot,Z) \,]\right) (x) =\;\;& \operatorname{cl}\left(\displaystyle\bigcup_{\phi \in \partial ({\operatorname{CVaR}_{\alpha}} [\,\cdot\,]) \,[\,c(x,Z)]} \int \partial_1 \,c(x,Z(\omega)) \, \phi(\omega) \,\mbox{d}\mathbb{P}_Z(\omega)\right) \\[0.27in]
        \overset{(\rm\romannumeral4^\prime)}{=}& {\operatorname{cl}\left( \displaystyle\int \partial_1 \,c(x,Z(\omega)) \, \bar\phi(\omega) \,\mbox{d}\mathbb{P}_Z(\omega)\right)
        \quad\forall\, \bar\phi \in \partial (\mbox{CVaR}_{\alpha} [\,\cdot\,]) \,[\,c(x,Z)]} \\[0.17in]
        \overset{(\rm\romannumeral5^\prime)}{=}\;\,& {\displaystyle\int \partial_1 \,c(x,Z(\omega)) \, \bar\phi(\omega) \,\mbox{d}\mathbb{P}_Z(\omega)
        \qquad\quad\;\;\forall\, \bar\phi \in \partial (\mbox{CVaR}_{\alpha} [\,\cdot\,]) \,[\,c(x,Z)].}
    \end{array}
    \end{split}
    \end{equation}
    To see $(\rm\romannumeral4^\prime)$, it suffices to show that, for arbitrary two elements $\phi_1$ and $\phi_2$ in the set $\partial (\mbox{CVaR}_{\alpha} [\,\cdot\,]) \,[\,c(x,Z)]$, we have
    \begin{equation}
    \label{eq:int_a.s.}
        \displaystyle\int \partial_1 \,c(x,Z(\omega)) \, \phi_1(\omega) \,\mbox{d}\mathbb{P}_Z(\omega)
        = \int \partial_1 \,c(x,Z(\omega)) \, \phi_2(\omega) \,\mbox{d}\mathbb{P}_Z(\omega).
    \end{equation}
    To this end, we take any measurable selection $a(x, Z(\omega)) \in \partial_1 c(x, Z(\omega))$. By the assumption that $|c(x, z) - c(x^\prime, z)| \leq \kappa(z) \|x - x^\prime\|$ for all $x, x^\prime \in \R^n$ and $z \in \R^m$, it holds that $\|a(x, Z)\| \leq \kappa(Z)$ since subgradients of a convex function are uniformly bounded in norm by the Lipschitz constant. Consequently, both $a(x, Z(\omega)) \, \phi_1(\omega)$ and $a(x, Z(\omega)) \, \phi_2(\omega)$ are integrable as $|\phi_1(\omega)| \leq (1 - \alpha)^{-1}$ and $|\phi_2(\omega)| \leq (1 - \alpha)^{-1}$ for any $\omega$ by \eqref{eq:CVaR_subdiff} and $\mathbb{E}[\|a(x,Z)\|] \leq \mathbb{E}[\kappa(Z)] < +\infty$ by our assumption. Observing that $a(x, Z(\omega)) \, \phi_1(\omega) = a(x, Z(\omega)) \, \phi_2(\omega)$ almost surely, we can conclude from \cite[Proposition 2.23]{folland1999real} that $\int a(x,Z(\omega)) \, \phi_1(\omega) \,\mbox{d}\mathbb{P}_Z(\omega) = \int a(x,Z(\omega)) \, \phi_2(\omega) \,\mbox{d}\mathbb{P}_Z(\omega)$. This completes the proof of \eqref{eq:int_a.s.}.
    
    Next, we will explain why the closure can be removed in \eqref{eq:CVaR_comp_subdiff}. By the convexity of $c(\cdot,z)$ for any fixed $z \in \R^m$ and the existence of a measurable function $\kappa$, it follows from \cite[Theorem 2.7.2]{clarke1990optimization} that
    \[
        \int \partial_1 c(x,Z(\omega)) \, \bar\phi(\omega) \, \mbox{d}\mathbb{P}_Z(\omega) = \partial {\left(\int c(\cdot,Z(\omega)) \, \bar\phi(\omega) \, \mbox{d}\mathbb{P}_Z(\omega) \right) (x)},
    \]
    where the right-hand-side is the subdifferential of a convex function and, thus, is a closed set. Then, we can omit the closure to obtain the equation $(\rm\romannumeral5^\prime)$ in \eqref{eq:CVaR_comp_subdiff}.

    \gap
    
    Now we use the expression of $\partial \left(\mbox{CVaR}_{\alpha}[\, c(\cdot,Z) \,]\right) (x)$ to characterize $\partial_A \mbox{VaR}_\alpha [\,c(\cdot, Z)] (\bar x)$. For any $k > 1/\alpha$, taking any $\phi_3 \in \partial (\operatorname{CVaR}_{\alpha - 1/k} [\,\cdot\,]) \,[\,c(x,Z)]$ and $\phi_4 \in \partial (\operatorname{CVaR}_{\alpha} [\,\cdot\,]) \,[\,c(x,Z)]$, we have
    \[
    \begin{array}{rl}
        \partial g^k(x) - \partial h^k(x)
        =\;\;& [k(1 - \alpha) + 1] \, \partial \operatorname{CVaR}_{\alpha - 1 / k} [\,c(\cdot, Z)] (x) - k(1 - \alpha) \,\partial \operatorname{CVaR}_{\alpha} [\,c(\cdot, Z)] (x) \\[0.05in]
        \overset{\eqref{eq:CVaR_comp_subdiff}}{=}& \displaystyle\int \partial_1 c(x, Z(\omega)) \cdot \big([k(1 - \alpha) + 1] \, \phi_3(\omega) - k (1-\alpha) \,\phi_4(\omega) \big) \, \mbox{d}\mathbb{P}_Z(\omega) \\
         \overset{\eqref{eq:CVaR_subdiff}}{=}& \displaystyle\int \partial_1 c(x, Z(\omega)) \cdot \phi(\omega) \, \mbox{d}\mathbb{P}_Z(\omega),
    \end{array}
    \]
    with
    \[
        \phi(\omega) \triangleq \left\{
        \begin{array}{cl}
        {0} &{\text{ if } c(x,Z(\omega)) > \mbox{VaR}_{\alpha}[\,c(x,Z)]} {\text{ or } c(x,Z(\omega)) < \mbox{VaR}_{\alpha-1/k}[\,c(x,Z)]}\\
        {[0, k]} &{\text{ if } c(x,Z(\omega)) = \mbox{VaR}_{\alpha}[\,c(x,Z)]} \text{ or } c(x,Z(\omega)) = \mbox{VaR}_{\alpha-1/k}[\,c(x,Z)] \\
        k &\text{ if } \mbox{VaR}_{\alpha-1/k}[\,c(x,Z)] < c(x,Z(\omega)) < \mbox{VaR}_{\alpha}[\,c(x,Z)] 
        \end{array}\right. .
    \]
    Since the event $\{\omega \in \Omega \mid c(x, Z(\omega)) = \mbox{VaR}_\alpha [\, c(x,Z)] \text{ or } \mbox{VaR}_{\alpha - 1/k} [\, c(x,Z)]\}$ has zero probability, we have
    \[
    \begin{array}{rl}
        \partial g^k(x) - \partial h^k(x)
        =& \displaystyle\int \partial_1 c(x, Z(\omega)) \, k \, \mathbf{1}\{\mbox{VaR}_{\alpha-1/k}[\,c(x,Z)] < c(x,Z) < \mbox{VaR}_{\alpha}[\,c(x,Z)]\} \, \mbox{d} \mathbb{P}_Z(\omega) \\[0.03in]
        =& \displaystyle\int \partial_1 c(x, Z(\omega)) \cdot \frac{\mathbf{1}\{\mbox{VaR}_{\alpha-1/k}[\,c(x,Z)] < c(x,Z) < \mbox{VaR}_{\alpha}[\,c(x,Z)]\}}{\mathbb P(\mbox{VaR}_{\alpha-1/k}[\,c(x,Z)] < c(x,Z) < \mbox{VaR}_{\alpha}[\,c(x,Z)])} \, \mbox{d} \mathbb{P}_Z(\omega) \\[0.15in]
        =& \mathbb{E}\left[\,\partial_1 c(x,Z) \,\middle|\, \mbox{VaR}_{\alpha-1/k}[\,c(x,Z)] < c(x,Z) < \mbox{VaR}_{\alpha}[\,c(x,Z)] \,\right].
    \end{array}
    \]
    By the definition of the approximate subdifferential, the proof is then completed. 
\end{proof}

\vspace{0.15in}

\section*{Appendix B. {Proof} of Proposition \ref{prop:assumption5}.}

We start with the chain rules for $\partial(\varphi\circ f)$ and $\partial^\infty(\varphi\circ f)$, where the inner function $f$ is merely lsc. These results are extensions of the nonlinear rescaling \cite[Proposition 10.19(b)]{rockafellar2009variational} to the case where $\varphi$ may lack the strictly increasing property at a given point. {One can also derive the same results through a general chain rule of the coderivative for composite set-valued mappings \cite[Theorem 5.1]{mordukhovich1994generalized}. However, to avoid the complicated computations accompanied by the introduction of the coderivative, we give an alternative proof below that is more straightforward.} To prepare for the chain rules, we need a technical lemma about the proximal normal cone.

\begin{lemma}
\label{lem:prox_normal_cone}
    Let $f: \R^n \to \R$ be a lsc function. {For $\bar\alpha > f(\bar x)$}, it holds that
    \[
        \mathcal{N}^p_{\operatorname{epi} f} (\bar x, \bar \alpha) \subset \mathcal{N}^p_{\operatorname{epi} f} (\bar x, f(\bar x)).
    \]
\end{lemma}
\begin{proof}[Proof of Lemma \ref{lem:prox_normal_cone}.]
    For $\bar\alpha > f(\bar x)$. By \cite[Example 6.16]{rockafellar2009variational}, we have
    \[
        \mathcal{N}^p_{\operatorname{epi} f} (\bar x, \bar \alpha) = \left\{\lambda[(x, \alpha) - (\bar x,\bar\alpha)] \mid (x, \alpha)\in\R^{n+1} \text{ such that } (\bar x, \bar\alpha) \in \Pi_{\operatorname{epi}f}(x,\alpha), \lambda \geq 0\right\},
    \]
    where $\Pi_{\operatorname{epi} f}: \R^{n+1} \to \operatorname{epi} f$ is the projection operator. For any $(x, \alpha) \in \R^{n+1}$ with $(\bar x, \bar\alpha) \in \Pi_{\operatorname{epi}f}(x,\alpha)$, we have
    \[
        (\bar x, \bar\alpha) \in \displaystyle\operatornamewithlimits{argmin}_{(u,t) \in \operatorname{epi}f} \|(u,t) - (x,\alpha)\|^2,
    \]
    which can be equivalently written as
    \[
        (\bar x, f(\bar x)) \in \displaystyle\operatornamewithlimits{argmin}_{(u,t + \bar\alpha - f(\bar x)) \in \operatorname{epi}f} \|(u,t) - (x,\alpha - \bar\alpha + f(\bar x))\|^2.
    \]
     Then restrict the feasible region of the above problem to a subset $\{(u,t)\mid (u,t) \in \operatorname{epi} f\}$. Since $(\bar x, f(\bar x))$ is still a feasible point in this subset, we have $(\bar x, f(\bar x)) \in \Pi_{\operatorname{epi}f}(x,\alpha-\bar\alpha+f(\bar x))$. Hence, $(\bar x, \bar\alpha) \in \Pi_{\operatorname{epi}f}(x,\alpha)$ implies $(\bar x, f(\bar x)) \in \Pi_{\operatorname{epi}f}(x,\alpha-\bar\alpha+f(\bar x))$. Using this result and the expression of $\mathcal{N}^p_{\operatorname{epi} f} (\bar x, \bar \alpha)$, we conclude that $\mathcal{N}^p_{\operatorname{epi} f} (\bar x, \bar \alpha) \subset \mathcal{N}^p_{\operatorname{epi} f} (\bar x, f(\bar x))$ for $\bar\alpha > f(\bar x)$. 
\end{proof}

We present the chain rules with a self-contained proof in the following lemma.

\begin{lemma}[chain rules for the limiting subdifferential]
\label{lem:chainrule_univar}
	Let $\varphi: \R \rightarrow \ER$ be proper, lsc,  convex, and nondecreasing with $\sup\varphi = +\infty$,  and $f: \R^n \rightarrow \R$ be lsc. 
 Consider $\bar x\in \dom (\varphi \circ f)$.  If the only scalar $y \in \operatorname{Lim \, sup}_{x \rightarrow_{(\varphi \circ f)} \bar x} \,\mathcal N_{\dom\varphi} (f(x))$ with $0 \in y \cdot \operatorname{Lim \, sup}_{x \rightarrow \bar x} \partial f(x)$ is $y = 0$, then
	\[
	\begin{array}{ll}
		\quad\partial (\varphi \circ f)(\bar x)
		\subset \bigcup
		\left\{\, y \cdot \Limsup_{x \rightarrow \bar x}\partial f(x) \,\middle|\, y \in \Limsup_{x \rightarrow_{(\varphi\circ f)} \,\bar x}\partial \varphi(f(x)) \,\right\}
        \cup
        \left[\,{\Big(}
            {\displaystyle\Limsup_{x \rightarrow \bar x}}^\infty \partial f(x)
        {\Big)}
        \backslash\{0\}\,\right] , \\[0.2in]
		\,\partial^\infty (\varphi \circ f)(\bar x)
		\subset \bigcup
		\left\{ \, y \cdot \Limsup_{x \rightarrow \bar x}\partial f(x) \,\middle|\, y \in \Limsup_{x \rightarrow_{(\varphi\circ f)} \,\bar x}\mathcal N_{\dom\varphi} (f(x)) \,\right\}
        \cup
        \left[\,{\Big(}
            {\displaystyle\Limsup_{x \rightarrow \bar x}}^\infty \partial f(x)
        {\Big)}
        \backslash\{0\}\,\right].
	\end{array}
	\]	
\end{lemma}

\begin{proof}[Proof of Lemma \ref{lem:chainrule_univar}.]
    The basic idea is to rewrite $\varphi \circ f$ as a parametric minimization problem and apply \cite[Theorem 10.13]{rockafellar2009variational}. Note that $\varphi \big(f(x)\big) = \inf_{\alpha} \; [g(x, \alpha) \triangleq \delta_{\operatorname{epi} f}(x, \alpha) + \varphi(\alpha)]$ for $x \in \dom (\varphi \circ f)$. Define the corresponding set of optimal solutions as {$M(x)$} for any $x \in \dom (\varphi \circ f)$. Then, we have $f(\bar x) \in {M(\bar x)}$ and $\varphi(\alpha) = \varphi(f(\bar x))$ for any $\alpha \in {M(\bar x)}$. 
    {By our assumptions, it is obvious that $\dom\varphi \in \{(-\infty, b), (-\infty,b]\}$ for some $b \in \R \cup \{+\infty\}$.}
    Based on our assumption that $\sup\varphi = +\infty$ and $f$ is lsc, it is easy to verify that $g$ is proper, lsc, and level-bounded in $\alpha$ locally uniformly in $x$. Then we apply \cite[Theorem 10.13]{rockafellar2009variational} to obtain
    \begin{equation}\label{eq:chain_rule}
         \partial (\varphi \circ f)(\bar x) \subset \left\{ v \mid (v, 0) \in \partial g(\bar x, \bar\alpha), \bar\alpha \in M(\bar x)\right\}
         ,\;
         \partial^\infty (\varphi \circ f)(\bar x) \subset \left\{ v \mid (v, 0) \in \partial^\infty g(\bar x, \bar\alpha), \bar\alpha \in M(\bar x)\right\}.
    \end{equation}

	{\underline{\em Step 1:}} We will show that for any $\bar\alpha \in {M(\bar x)}$,
	\begin{equation}\label{eq:sum_rule}
		\mathcal N_{\operatorname{epi}f}(\bar x, \bar\alpha) \cap \big(\,\{0\} \times [-\mathcal N_{\dom\varphi}(\bar\alpha)]\,\big) = \{0\}.
	\end{equation}
    We divide the proof of \eqref{eq:sum_rule} into two cases.

    \gap
    \noindent\underline{\textsl{Case 1.}} If ${M(\bar x)}$ is a singleton $\{f(\bar x)\}$, we can characterize $\mathcal N_{\operatorname{epi}f}(\bar x, f(\bar x))$ by using the result in \cite[Theorem 8.9]{rockafellar2009variational}. 
    Since $\partial f(\bar x) \subset \operatorname{Lim \, sup}_{x \rightarrow \bar x} \partial f(x)$ and $\mathcal{N}_{\dom\varphi}(f(\bar x)) \subset \operatorname{Lim \, sup}_{x \rightarrow_{(\varphi\circ f)} \,\bar x} \,\mathcal{N}_{\dom\varphi}(f(x))$,
    it follows from our assumption that either $0 \notin \partial f(\bar x)$ or $\mathcal N_{\dom\varphi} (f(\bar x)) = \{0\}$. Hence, based on the characterization of $\mathcal N_{\operatorname{epi}f} (\bar x, f(\bar x))$, \eqref{eq:sum_rule} is satisfied.
    
    \gap
    
    \noindent\underline{\textsl{Case 2.}} Otherwise, there exists $\bar\alpha_{\max} \in (f(\bar x), +\infty)$ such that ${M(\bar x)} = [f(\bar x), \bar\alpha_{\max}]$ since $\varphi$ is lsc, nondecreasing and $\sup\varphi = +\infty$. Thus, from \eqref{eq:chain_rule},
    \begin{equation}\label{eq:chainrule}
	\begin{split}
		\partial (\varphi \circ f)(\bar x) &\subset
		\big[\{ v \mid (v, 0) \in \partial g(\bar x, f(\bar x))\}
        \cup
        \{ v \mid (v, 0) \in \partial g(\bar x, \bar\alpha), f(\bar x) < \bar\alpha \leq \bar\alpha_{\max}\}\big],\\[0.08in]
        \partial^\infty (\varphi \circ f)(\bar x) &\subset
		\big[\{ v \mid (v, 0) \in \partial^\infty g(\bar x, f(\bar x))\}
        \cup
        \{ v \mid (v, 0) \in \partial^\infty g(\bar x, \bar\alpha), f(\bar x) < \bar\alpha \leq \bar\alpha_{\max}\}\big].
        \end{split}
    \end{equation}
    Let ${M_1(\bar x)} \triangleq \left\{ \bar\alpha \in (f(\bar x), \bar\alpha_{\max}] \,\middle|\, \exists \,x^k \rightarrow \bar x \text{ with } f(x^k) \rightarrow \bar\alpha\right\}$ and ${M_2(\bar x) \triangleq M(\bar x) \backslash M_1(\bar x)}$. 
    In the following, we characterize $\mathcal{N}_{\operatorname{epi}f} (\bar x, \bar\alpha)$ and verify \eqref{eq:sum_rule} separately for $\bar\alpha \in {M_1(\bar x)}$ and $\bar\alpha \in {M_2(\bar x)}$.

    \gap
	\noindent\underline{\textsl{Case 2.1.}} For any $\bar\alpha \in {M_1(\bar x)}$, we first prove the inclusion:
    \begin{equation}\label{eq:2.1normalcone_epi}
        \mathcal N_{\operatorname{epi}f}(\bar x, \bar\alpha)
		\subset \bigg[\left\{ \lambda (v, -1) \,\middle|\, v \in \Limsup_{x \rightarrow \bar x} \partial f(x), \lambda > 0 \right\} \cup \left\{ (v, 0) \,\middle|\, v \in {\displaystyle\Limsup_{x \rightarrow \bar x}}^\infty \partial f(x) \right\}\bigg].
    \end{equation}
    Observe that for any $\bar\alpha \in {M_1(\bar x)}$, it holds that
    \begin{equation}\label{eq:normals}
\mathcal{N}_{\operatorname{epi}f}(\bar x, \bar\alpha)
        \subset \Limsup_{(x, \alpha)(\in \operatorname{epi}f) \rightarrow (\bar x, \bar\alpha)} \;\mathcal{N}^p_{\operatorname{epi}f} (x, \alpha)
    \subset  \Limsup_{x \rightarrow \bar x} \;\mathcal{N}^p_{\operatorname{epi}f} (x,f(x))
        \subset \Limsup_{x \rightarrow \bar x} \;\mathcal{N}_{\operatorname{epi}f} (x,f(x)),
    \end{equation}
    where the first inclusion is because any normal vector is a limit of proximal normals at nearby points \cite[Exercise 6.18]{rockafellar2009variational}; the second one uses Lemma \ref{lem:prox_normal_cone}; 
 the last inclusion follows {from the fact that the proximal normal cone is a subset of the limiting normal cone \cite[Example 6.16]{rockafellar2009variational}}. Based on the the result of \cite[Theorem 8.9]{rockafellar2009variational} that
    \[
    \mathcal{N}_{\operatorname{epi}f} (x,f(x))
        = \left\{ \lambda (v, -1) \,\middle|\, v \in\partial f(x), \lambda > 0 \right\} \cup \left\{ (v, 0) \,\middle|\, v \in \partial^\infty f(x) \right\},
    \]
    we conclude that $\mathcal{N}_{\operatorname{epi}f} (\bar x,\bar\alpha) \subset \R^n \times \R_-$ for any $\bar\alpha \in {M_1(\bar x)}$. For any $(v, -1) \in \mathcal{N}_{\operatorname{epi}f} (\bar x,\bar\alpha)$ with $\bar\alpha \in {M_1(\bar x)}$, there exist $x^k \rightarrow \bar x$, $v^k \rightarrow v$ with $v^k \in \partial f(x^k)$. Then $v \in \operatorname{Lim \, sup}_{x \rightarrow \bar x} \partial f(x)$.
    
    To prove \eqref{eq:2.1normalcone_epi}, it remains to show that $v \in \operatorname{Lim \, sup}_{x \rightarrow \bar x}^\infty \partial f(x)$ whenever $(v, 0) \in \mathcal{N}_{\operatorname{epi}f} (\bar x, \bar\alpha)$. It follows from \eqref{eq:normals} that $(v, 0)$ is a limit of proximal normals of $\operatorname{epi}f$ at $(x^k, f(x^k))$ for some sequence $x^k \rightarrow \bar x$. 
    (\romannumeral1) First consider the case $(v^k, 0) \rightarrow (v, 0)$ with $(v^k, 0) \in \mathcal{N}^p_{\operatorname{epi}f} (x^k,f(x^k))$. Following the argument in the proof of \cite[Theorem 8.9]{rockafellar2009variational}, we can derive $v^k \in \partial^\infty f(x^k)$. Therefore,
    \[
        v \in \Limsup_{k \rightarrow +\infty} \;\partial^\infty f(x^k)
        \;\subset\; \Limsup_{k \rightarrow +\infty} \left( \bigcup_{ x^{k,i} \rightarrow_f \, x^k}{\displaystyle\Limsup_{i \rightarrow +\infty}}^\infty \partial f(x^{k,i})\right)
        \;\subset\; \bigcup_{ x^j \rightarrow \bar x} {\displaystyle\Limsup_{j \rightarrow +\infty}}^\infty \;\partial f(x^j),
    \]
    where the first inclusion is due to the definition of the horizon subdifferential, and the last inclusion follows from a standard diagonal extraction procedure. 
    (\romannumeral2) In the other case, we have $\lambda_k (v^k, -1) \rightarrow (v, 0)$ with $\lambda_k \downarrow 0$ and $v^k \in \partial f(x^k)$ for all ${k \in \N}$. It is easy to see $v \in \operatorname{Lim \, sup}_{x \rightarrow \bar x}^\infty \partial f(x)$. So far, we obtain inclusion \eqref{eq:2.1normalcone_epi}. Since $\bar\alpha \in {M_1(\bar x)}$, we have $\mathcal{N}_{\dom\varphi} (\bar\alpha) \subset \operatorname{Lim \, sup}_{x \rightarrow_{(\varphi \circ f)} \bar x} \,\mathcal{N}_{\dom\varphi}(f(x))$, and our assumption implies that $\lambda = 0$ is the unique solution satisfying $0 \in \lambda \cdot \operatorname{Lim \, sup}_{x \rightarrow \bar x} \partial f(x)$ with $\lambda \in \mathcal{N}_{\dom\varphi}(\bar\alpha)$. Combining this with \eqref{eq:2.1normalcone_epi}, we immediately obtain \eqref{eq:sum_rule}.

    \gap
    \noindent\underline{\textsl{Case 2.2.}} For any $\bar\alpha \in {M_2(\bar x)}$, consider any sequence $\big\{(x^k, \alpha^k)\big\}  \subset \operatorname{epi}f$ converging to $(\bar x, \bar\alpha)$. Then $\alpha^k > f(x^k)$ for all sufficiently large $k$ since $\bar\alpha \notin {M_1(\bar x)}$. It is easy to see that $\mathcal{N}^p_{\operatorname{epi}f} (x^k, \alpha^k) \subset \R^n \times \{0\}$, which gives us $\mathcal{N}_{\operatorname{epi}f} (x^k, \alpha^k) \subset \R^n \times \{0\}$ due to \cite[Exercise 6.18]{rockafellar2009variational}. By following a similar pattern as the final part of Case 2.1, it is not difficult to obtain, for any $\bar\alpha \in {M_2(\bar x)}$,
    \begin{equation}\label{eq:2.2normalcone_epi}
        \mathcal N_{\operatorname{epi}f}(\bar x, \bar\alpha)
		\subset \left\{ (v, 0) \,\middle|\, v \in {\Limsup_{x \rightarrow \bar x}}^\infty \partial f(x) \right\}.
    \end{equation}
	In this case, \eqref{eq:sum_rule} holds trivially. {Hence, we have verified \eqref{eq:sum_rule} for cases 1 and 2.

    \gap\gap
    
    \underline{\em Step 2:} Based on \eqref{eq:sum_rule} in step 1, we can now apply the sum rule \cite[corollary 10.9]{rockafellar2009variational} for $\partial g(\bar x, \bar\alpha)$ to obtain
    \begin{equation}\label{eq:sum_rule_g}
        \partial  g(\bar x, \bar\alpha) \subset \mathcal{N}_{\operatorname{epi} f} (\bar x, \bar\alpha) + \{0\} \times \partial\varphi(\bar\alpha), \qquad
        \partial^\infty  g(\bar x, \bar\alpha) \subset \mathcal{N}_{\operatorname{epi} f} (\bar x, \bar\alpha) + \{0\} \times \mathcal{N}_{\dom\varphi}(\bar\alpha).
    \end{equation}
    
    \noindent\underline{\textsl{Case 1.}} For ${M(\bar x)} = \{f(\bar x)\}$, by combining \eqref{eq:sum_rule_g} with \eqref{eq:chain_rule}, we can derive the stated results for $\partial(\varphi \circ f)(\bar x)$ and $\partial^\infty (\varphi \circ f)(\bar x)$ based on the observations that $\partial\varphi(f(\bar x)) \subset \operatorname{Lim \, sup}_{x \rightarrow_{(\varphi\circ f)} \bar x} \varphi(f(x))$ and $\partial^\infty f(\bar x) \subset \operatorname{Lim \, sup}_{x \rightarrow \bar x}^\infty \partial f(x)$.

    \gap
    
    \noindent\underline{\textsl{Case 2.}} Otherwise, by \eqref{eq:sum_rule_g}}, we have
	\[
	\begin{array}{cl}
		& \{ v \mid (v, 0) \in \partial g(\bar x, \bar\alpha), f(\bar x) < \bar\alpha \leq \bar\alpha_{\max}\}\\
        \overset{\eqref{eq:2.1normalcone_epi} \eqref{eq:2.2normalcone_epi}}{\subset}&
        \bigcup \left\{ y \cdot \Limsup_{x \rightarrow \bar x} \partial f(x)\,\middle|\, y \in \partial\varphi(\bar\alpha), \bar\alpha \in {M_1}(\bar x) \right\} \cup
        \left[\, {\bigcup } \left\{{\displaystyle\Limsup_{x \rightarrow \bar x}}^\infty \partial f(x) \,\middle|\, 0 \in \partial\varphi(\bar\alpha), f(\bar x) < \bar\alpha \leq \bar\alpha_{\max}\right\} \right]\\[0.15in]
		\subset&
		\bigcup \left\{\, y \cdot \Limsup_{x \rightarrow \bar x}\partial f(x) \,\middle|\, y \in \Limsup_{x \rightarrow_{(\varphi\circ f)} \,\bar x}\partial \varphi(f(x)) \,\right\} \cup \left[\,{\Big(}
            {\displaystyle\Limsup_{x \rightarrow \bar x}}^\infty \partial f(x)
        {\Big)} \backslash\{0\}\,\right],
	\end{array}
	\]
    {where the last inclusion is because $0$ will be included in the first set if $0 \in \partial\varphi(\bar\alpha)$ for some $\bar\alpha \in(f(\bar x), \bar\alpha_{\max}]$ and the second set will be empty otherwise. }Similarly,
	\[
	\begin{array}{rl}
		& \{ v \mid (v, 0) \in \partial g^\infty(\bar x, \bar\alpha), f(\bar x) < \bar\alpha \leq \bar\alpha_{\max}\} \\[0.08in]
		\subset&
		\bigcup \left\{\, y \cdot \Limsup_{x \rightarrow \bar x}\partial f(x) \,\middle|\, y \in \Limsup_{x \rightarrow_{(\varphi\circ f)} \,\bar x}\mathcal{N}_{\dom\varphi}(f(x)) \,\right\} \cup \left[\,{\Big(}
            {\displaystyle\Limsup_{x \rightarrow \bar x}}^\infty \partial f(x)
        {\Big)}\backslash\{0\}\,\right].
	\end{array}
	\]
	We then complete the proof by using the inclusions in  \eqref{eq:chainrule}. 
\end{proof}

Equipped with the chain rules, we are now ready to prove Proposition \ref{prop:assumption5}.
\gap

\begin{proof}[Proof of Proposition \ref{prop:assumption5}.]
    Let $\bar x$ be any feasible point, i.e., $\bar x \in {\bigcap_{p=1}^{m}} \dom F_p$. Suppose for contradiction that \eqref{eq:CQ1} does not hold at $\bar x$. Thus, there exist $p_1 \in\{1, \cdots, m\}$, $\{x^k\} \in S_{p_1}(\bar{x})$ and an index set $N \in \mathbb{N}_\infty^\sharp$ such that $0 \in \partial_{C} f^k_{p_1}(x^k)$ and $\mathcal{N}_{\dom\varphi_{p_1}}\left(f^k_{p_1}(x^k)\right) \neq \{0\}$ for all $k \in N$. 
    Take an arbitrary nonzero scalar $y^k \in \mathcal{N}_{\dom\varphi_{p_1}}\left(f^k_{p_1}(x^k)\right)$ for all $k \in N$. Let $\widetilde y$ be any accumulation point of the unit scalars $\{y^k/|y^k|\}_{k \in N}$. Then, we have $(0 \neq)\widetilde y \in \bigcup \{\mathcal{N}_{\dom\varphi_{p_1}}(t_{p_1}) \mid t_{p_1} \in T_{p_1}(\bar x)\}$ and $0 \in \Conv\partial_{A} f_{p_1} (\bar x)$, contradicting Assumption 5. This proves condition \eqref{eq:CQ1}.

    For any fixed $p = 1,\cdots,m$, let $y_{p^\prime} = 0$ for any $p^\prime \in \{1, \cdots, m\} \backslash \{p\}$ in Assumption 5. Then the only scalar $y_p \in \bigcup \left\{\mathcal{N}_{\dom\varphi_{p}} (t_p) \mid t_p \in T_p(\bar{x}) \right\}$ with $0 \in y_p \, \Conv\partial_{A} f_{p}(\bar x)$ is $y_p = 0$, which completes the proof of \eqref{eq:CQ2}.
    
    To derive the constraint qualification \eqref{eq:CQ3}, we consider two cases.
    
    \noindent\underline{\textsl{Case 1.}} For $p \in I_2$, we have $\mathcal{N}_{\dom\varphi_p} (f_p(\bar x)) \subset \bigcup \big\{\mathcal{N}_{\dom\varphi_p}(t_p) \mid t_p \in T_p(\bar x)\big\}$ due to $f^k_p \Epiconv f_p$ and $\partial(y f_p)(\bar x) \subset y \,\partial_C f_p(\bar x) \subset y \cdot \Conv\partial_{A} f_p(\bar x)$ for any $y$ by {Theorem \ref{thm:eADC_subdiff}(a)}. Together with Assumption 5, we deduce that the only scalar $y \in \mathcal{N}_{\dom\varphi_p}(f_p(\bar x))$ with $0 \in \partial(y f_p)(\bar x)$ is $y = 0$. From this condition and the local Lipschitz continuity of $f_p$ for $p \in I_2$, we can apply the chain rule \cite[Theorem 10.49]{rockafellar2009variational} to get
    \begin{equation}\label{eq:horizon1}
        \partial^\infty(\varphi_p \circ f_p)(\bar x) 
        \subset \bigcup
        \left\{ y \cdot \Conv\partial_{A} f_p(\bar x) \mid y \in \mathcal{N}_{\dom\varphi_p}(t_p), t_p \in T_p(\bar{x}) \right\}.
    \end{equation}
    
    \noindent\underline{\textsl{Case 2.}} For $p \in I_1$, to utilize the chain rules ({Lemma} \ref{lem:chainrule_univar}) for $\partial^\infty(\varphi_p \circ f_p)$, we must first confirm the validity of the condition:
    \begin{equation}\label{eq:prop10_case2}
        \left[0 \in y \cdot \Limsup_{x \rightarrow \bar x} \partial f_p(x), \quad y \in \Limsup_{x \rightarrow_{F_p} \bar x} \;\mathcal{N}_{\dom\varphi_p} (f_p(x))\right]
        \quad\Longrightarrow\quad y = 0.
    \end{equation}
    Indeed, it suffices to consider the case of $\dom\varphi^\uparrow_p = (-\infty,r_p)$ or $(-\infty,r_p]$ for some $r_p \in \R$, because the statement holds trivially when $\varphi^\uparrow_p$ is real-valued. For any element $\bar y \in \operatorname{Lim \, sup}_{x \rightarrow_{F_p} \,\bar x} \,\mathcal N_{\dom\varphi_p} (f_p(x))$, there exist $(x^k, y^k) \rightarrow (\bar x, \bar y)$ with $y^k \in \mathcal{N}_{\dom\varphi_p} (f_p(x^k))$ and $F_p(x^k) \rightarrow F_p(\bar x)$. Since $\bar x \in \dom F_p$, we must have $x^k \in \dom F_p$ for all sufficiently large k, i.e., $f_p(x^k) \in \dom\varphi^\uparrow_p$, and $\{f_p(x^k)\}_{k \in \N}$ is bounded from above {due to $\dom\varphi^\uparrow_p = (-\infty,r_p)$ or $(-\infty,r_p]$}. The sequence $\{f_p(x^k)\}_{k \in \N}$ is also bounded from below since $f_p$ is lsc as a consequence of $f^k_p \Epiconv f_p$. Then, we can assume that the bounded sequence $\{f_p(x^k)\}_{k \in \N}$ converges to some $\bar z_p$. Note that $\bar z_p \in \dom\varphi_p$ due to $F_p(\bar x) = \liminf_{k \rightarrow +\infty} \varphi_p(f_p(x^k)) \geq \varphi_p(\bar z_p)$. 
    Thus, by the outer semicontinuity, $y^k \rightarrow \bar y\in \mathcal{N}_{\dom\varphi_p}(\bar z_p)$. By $f^k_p \Epiconv f_p$, each $f_p(x^k)$ can be expressed as the limit of a sequence $\{f^i_p(x^{k,i})\}_{i \in \N}$ with $x^{k,i} \rightarrow x^k$ for any fixed ${k \in \N}$. Using a standard diagonal extraction procedure, one can extract a subsequence $f^{i_k}_p(x^{k,i_k}) \rightarrow \bar z_p$ with $x^{k,i_k} \rightarrow \bar x$. Hence, $\bar z_p \in T_p(\bar x)$ and
    \begin{equation}\label{eq:embed_NormalCone}
        \Limsup_{x \rightarrow_{F_p} \,\bar x}\, \mathcal{N}_{\dom\varphi_p} (f_p(x)) \subset \bigcup \{\mathcal{N}_{\dom\varphi_p}(t_p) \mid t_p \in T_p(\bar{x})\}.
    \end{equation}
    Using the subdifferentials relationships in {Theorem} \ref{thm:eADC_subdiff} and the outer semicontinuity of $\partial_{A} f_p$, we  have
    \begin{equation}\label{eq:embed_SubdiffPhi}
        \displaystyle\Limsup_{x \rightarrow \bar x} \;\partial f_p(x) \subset \Limsup_{x \rightarrow \bar x} \;\partial_A f_p(x) = \partial_{A} f_p(\bar x).
    \end{equation}
    By \eqref{eq:embed_NormalCone}, \eqref{eq:embed_SubdiffPhi} and Assumption 5, we immediately get \eqref{eq:prop10_case2}. Thus, we can apply the chain rule in {Lemma} \ref{lem:chainrule_univar}, and use \eqref{eq:embed_NormalCone}, \eqref{eq:embed_SubdiffPhi} again to obtain
    \begin{equation}\label{eq:horizon2}
    \begin{array}{rl}
        \partial^\infty (\varphi_p \circ f_p)(\bar x)
        &\displaystyle\subset \bigcup \left\{ \, y \, \partial_{A} f_p(x) \,\middle|\, y \in \mathcal N_{\dom\varphi_p} (t_p),\, t_p \in T_p(\bar{x}) \right\} \cup \left[\,{\displaystyle\Limsup_{x \to \bar x}}^\infty {\partial} f_p(x) \backslash\{0\}\,\right]\\[0.12in]
        &\displaystyle\subset \bigcup \left\{ \, y \, \partial_{A} f_p(x) \,\middle|\, y \in \mathcal N_{\dom\varphi_p} (t_p),\, t_p \in T_p(\bar{x}) \right\} \cup \left[\,\partial^\infty_{A} f_p(\bar x) \backslash\{0\}\,\right].
    \end{array}
    \end{equation}
    For the last inclusion, we use $\operatorname{Lim \, sup}_{x \rightarrow \bar x}^\infty \;\partial f_p(x) \subset \operatorname{Lim \, sup}_{x \rightarrow \bar x}^\infty \;\partial_A f_p(x) \subset \partial^\infty_A f_p(\bar x)$ by Theorem \ref{thm:eADC_subdiff}(a) and using a standard diagonal extraction procedure. Combining inclusions \eqref{eq:horizon1}, \eqref{eq:horizon2} for two cases with Assumption 5, we derive \eqref{eq:CQ3} and complete the proof. 
\end{proof}




\end{document}